\documentclass[12pt,a4paper]{article}
\usepackage{modnatbib}
\usepackage{amssymb,amsmath,amsthm,epsfig,verbatim,pstricks,url}
\usepackage{tikz,pgfplots,tikz-3dplot}
\usepackage{ifpdf}  
\newtheorem{thm}{\sc Theorem.}[section]
\newtheorem{lem}[thm]{\sc Lemma.}
\newtheorem{rem}[thm]{\sc Remark.}
\renewcommand{\theequation}{\arabic{section}.\arabic{equation}}
\newenvironment{AMS}%
{{\upshape\bfseries AMS subject classifications. }\ignorespaces}{}
\newenvironment{keywords}{{\upshape\bfseries Key words. }\ignorespaces}{}

\newcommand{\bRplus}{{\mathbb R}_{>0}}
\newcommand{\bRgeq}{{\mathbb R}_{\geq 0}}

\newcommand{\RZ}{{\mathbb R} \slash {\mathbb Z}}
\newcommand{\RpisZ}{{\mathbb R} \slash (2\,\pi\,\mathfrak s\,{\mathbb Z})}
\newcommand{\RpiZ}{{\mathbb R} \slash (2\,\pi\,{\mathbb Z})}
\newcommand{\bR}{{\mathbb R}}
\newcommand{\bN}{{\mathbb N}}

\newcommand{\bD}{{\mathbb D}}
\newcommand{\bH}{{\mathbb H}}
\newcommand{\distg}{\operatorname{dist}_g}

\newcommand{\spa}{\operatorname{span}}

\newcommand{\ratio}{{\mathfrak r}}

\newcommand{\drho}{\;{\rm d}\rho}
\newcommand{\dz}{\;{\rm d}\vec z}

\newcommand{\Vh}{\underline{V}^h}

\newcommand{\Vhpartial}{\underline{V}^h_\partial}

\newcommand{\nabsg}{\nabla_{\!s_g}}
\newcommand{\Id}{\rm Id}
\newcommand{\id}{\rm id}

\newcommand{\dd}[1]{\frac{\rm d}{{\rm d}#1}}
\newcommand{\ddt}{\dd{t}}

\newcommand{\normal}{{\rm n}}
\newcommand{\ek}{e}
\newcommand{\ttau}{\Delta t}
\newcommand{\BGNmckappa}{\mathcal{A}}
\newcommand{\GDmckappa}{\mathcal{B}}
\newcommand{\BGNmc}{\mathcal{C}}
\newcommand{\GDmc}{\mathcal{D}}
\newcommand{\BGNsd}{\mathcal{E}}
\newcommand{\BGNsdstab}{\mathcal{F}}
\newcommand{\BGNwf}{\mathcal{U}}
\newcommand{\BGNwfwf}{\mathcal{W}}
\def\epsilon{\varepsilon}

\newcommand{\mat}[1]{\underline{\underline{#1}}\rule{0pt}{0pt}}

\newcommand{\myz}{{\mathfrak z}}

\newcommand{\errorXx}{\|\Gamma - \Gamma^h\|_{L^\infty}}
\def\arraystretch{1.15}

\allowdisplaybreaks

\begin{document}
\title{
Numerical approximation of curve evolutions \\ in Riemannian manifolds
\\
}
\author{John W. Barrett\footnotemark[2] \and 
        Harald Garcke\footnotemark[3]\ \and 
        Robert N\"urnberg\footnotemark[2]}

\renewcommand{\thefootnote}{\fnsymbol{footnote}}
\footnotetext[2]{Department of Mathematics, 
Imperial College London, London, SW7 2AZ, UK.
email: \texttt{\{j.barrett|robert.nurnberg\}@imperial.ac.uk}}
\footnotetext[3]{Fakult{\"a}t f{\"u}r Mathematik, Universit{\"a}t Regensburg, 
93040 Regensburg, Germany. email: \texttt{harald.garcke@ur.de}}

\date{}

\maketitle

\begin{abstract}
We introduce variational
approximations for curve evolutions in two-dimensional
Riemannian manifolds that are conformally flat, i.e.\ 
conformally equivalent to the Euclidean space. Examples include the 
hyperbolic plane, the hyperbolic disk, the elliptic plane as well as any
conformal parameterization of a two-dimensional surface in 
$\bR^d$, $d\geq 3$. In these spaces we introduce stable numerical
schemes for curvature flow and curve diffusion, and we also formulate a
scheme for elastic flow. Variants of the schemes can also be applied to
geometric evolution equations for axisymmetric hypersurfaces in
$\bR^d$. Some of the schemes have very good properties with
respect to the distribution of mesh points, which is demonstrated
with the help of several numerical computations. 
\end{abstract} 


\renewcommand{\thefootnote}{\arabic{footnote}}

\setcounter{equation}{0}
\section{Introduction} \label{sec:intro}
The evolution of curves in a two-dimensional manifold
driven by a velocity involving the (geodesic)
curvature of the curve appears in many situations in geometry and in
applications. 
Examples are curve straightening via the elastic energy
or image processing on surfaces. The first mathematical results on such
flows go back to the work of \cite{GageH86},
who studied curvature flow in the Euclidean plane. Later
evolutions in more complex ambient spaces have been studied, see
e.g.\ \cite{Grayson89,Cabezas-RivasM07,AndrewsC17}. 
In the Euclidean case it can be shown that closed
curves shrink to a point in finite time and they become more and more
round as they do so, see \cite{GageH86} and \cite{Grayson87}. In the
case of a general ambient space the solution behaviour is more
complex. For example, some solutions exist for arbitrary times and
others can become unbounded in finite or infinite time, see
e.g.\ \cite{Grayson89}.

Curvature flow is a second order flow. However, also fourth order
flows are of interest. Here we mention the elastic (Willmore) flow of
curves and curve diffusion, both of which are highly nonlinear. 
Elastic flow is the $L^2$--gradient flow of the elastic energy, and
in the hyperbolic plane and on the sphere it was recently studied by
\cite{DallAcquaS17preprint,DallAcquaS18preprint} and \cite{DallAcquaLLPS18},
respectively. 
The curve diffusion flow, sometimes also called surface diffusion flow, 
is the $H^{-1}$--gradient flow for the length of the curve, and, like 
the elastic flow, it also features second derivatives of the curvature. 

In this paper, we also want to study situations, in which a curve
evolves in a two-dimensional manifold that is not necessarily
embedded in $\bR^3$. An important example is the hyperbolic plane
$\mathbb{H}^2$, which due to Hilbert's classical theorem cannot be embedded
into $\bR^3$, see \cite{Hilbert1901} and e.g.\ \citet[\S11.1]{Pressley10}.
It will turn out that we can derive stable numerical
schemes for curve evolutions in two-dimensional Riemannian manifolds
that are conformally equivalent to the Euclidean space. This means
that charts exist such that in the parameter domain the metric tensor is a
possibly inhomogeneous 
scalar multiple of the classical Euclidean metric. This in particular
implies that the chart is angle preserving, we refer to
\cite{Kuhnel15} for more details.

The numerical approximation of the evolution of curves in an Euclidean ambient
space is very well developed, with many papers on parametric as well as
level set methods. We refer to \cite{DeckelnickDE05} for an
overview. However, for more general ambient spaces only a few papers
dealing with numerical methods exist. Some numerical work is devoted
to the evolution of curves on two-dimensional surfaces in
$\bR^3$. We refer to \cite{MikulaS06,curves3d,BenninghoffG16}
for methods using a parametric approach. Besides, also a level set
setting is possible in order to numerically move curves that are constrained on
surfaces, see \cite{ChengBMO02,SpiraK07}.

The setting in this paper is as follows. 
Let $I=\RZ$ be the periodic interval $[0,1]$ and let $\vec x : I \to \bR^2$ 
be a parameterization of a closed curve $\Gamma \subset \bR^2$.
On assuming that
\begin{equation} \label{eq:xrho}
|\vec x_\rho| \geq c_0 > 0 \qquad \forall\ \rho \in I\,,
\end{equation}
we introduce the arclength $s$ of the curve, i.e.\ $\partial_s =
|\vec{x}_\rho|^{-1}\,\partial_\rho$, and set
\begin{align}
\vec\tau(\rho) = \vec x_s(\rho) = 
\frac{\vec x_\rho(\rho)}{|\vec x_\rho(\rho)|} \qquad \mbox{and}
\qquad \vec\nu(\rho) = -[\vec\tau(\rho)]^\perp\,,
\label{eq:tau}
\end{align}
where $\cdot^\perp$ denotes a clockwise rotation by $\frac{\pi}{2}$.

On an open set $H\subset\mathbb{R}^2$ we define a metric tensor as
\begin{equation} \label{eq:g}
[(\vec v, \vec w)_g](\vec z) = g(\vec z)\,\vec v\,.\,\vec w \quad
\forall\ \vec v, \vec w \in \bR^2
\qquad \text{ for } \vec z \in H\,,
\end{equation}
where $\vec v\,.\,\vec w = \vec v^T\,\vec w$ is the standard Euclidean inner
product, and where $g:H \to \bRplus$ is a smooth positive weight 
function. This is the setting one obtains for a two-dimensional
Riemannian manifold that is conformally equivalent to the Euclidean
plane. In local coordinates the metric is precisely given by
\eqref{eq:g}, see e.g.\ \cite{Jost05,Kuhnel15,Schippers07}. 
Let us mention that a
two-dimensional Riemannian manifold locally allows for a conformal
chart, see e.g.\ \citet[\S5.10]{Taylor11I}. Examples of such
situations are the hyperbolic plane, the hyperbolic disc and the
elliptic plane. Other examples are given by curves on two-dimensional
surfaces in $\mathbb{R}^d$, $d\ge 3$, that can be conformally
parameterized, such as spheres without pole(s), catenoids and torii.
Coordinates $(x_1,x_2)\in H$ together with a metric
$g$ as in \eqref{eq:g} are called isothermal coordinates, i.e.\ in all
situations considered in this paper we assume that we have isothermal
coordinates. We refer to Section~\ref{sec:old1} and 
\citet[3.29 in \S3D]{Kuhnel15} for more information. 

For a time-dependent curve $\vec{x}$ the simplest curvature driven
flow is given as
\begin{equation}\label{eq:geod}
  \mathcal{V}_g = \varkappa_g\,.
\end{equation}
Here $\mathcal{V}_g = g^{\frac12}(\vec{x})\,\vec{x}_t\,.\,\vec{\nu}$ is the
normal velocity with respect to the metric \eqref{eq:g}, and
\begin{equation}\label{eq:kappag}
  \varkappa_g = g^{-\frac12}(\vec{x})\left[\varkappa -\tfrac12\,
  \vec{\nu}\,.\, \nabla\, \ln g(\vec{x})\right]
\end{equation}
is the curvature of the curve with respect to the metric $g$. The
vector $\vec{\nu}$, defined in \eqref{eq:tau} is the classical
Euclidean normal, and $\varkappa$ is the classical Euclidean curvature
of the curve. It satisfies the property
\begin{equation} \label{eq:varkappa}
\varkappa\,\vec\nu = \vec\varkappa = \vec\tau_s = \vec x_{ss} = 
\frac1{|\vec x_\rho|} \left[ \frac{\vec x_\rho}{|\vec x_\rho|} \right]_\rho,
\end{equation}
see \cite{DeckelnickDE05}.

In the Euclidean case, i.e.\ in the case
$g\equiv 1$, the right hand side in the curvature flow \eqref{eq:geod}
is equal to $\varkappa$, and in particular the parameterization
$\vec{x}$ only appears via $\vec{x}_\rho$, cf.\ \eqref{eq:kappag},
\eqref{eq:varkappa}. 
This is crucial for stability proofs for
numerical methods that have been introduced earlier,
cf.\ \cite{Dziuk88,triplej,DeckelnickDE05}. In the case of a general
ambient space, additional nonlinearities involving the variable
$\vec{x}$ itself appear in $\varkappa_g$, so that the variational structure of
(\ref{eq:varkappa}) is lost. This makes the design of stable
schemes highly non-trivial. 
In fact, no such schemes appear in the literature so far. 
We will introduce stable fully discrete
schemes with the help of a non-standard convex-concave splitting. In
particular, the splitting has to be chosen in terms of $g^{\frac12}$.
With the help of the splitting, we propose in 
Section~\ref{sec:fd} a semi-implicit scheme for which stability can be shown.

The outline of this paper is as follows. In Section~\ref{sec:old1} we
derive the governing equations for curvature flow, curve diffusion
and elastic flow, provide weak formulations and relate the introduced flows
to geometric evolution equations for axisymmetric
hypersurfaces. In Section~\ref{sec:fd} we introduce finite element
approximations and show existence and uniqueness as well as stability
results. Section~\ref{sec:nr} is devoted to several numerical results,
which demonstrate convergence rates as well as a qualitatively good
mesh behaviour. In two appendices we derive exact solutions and derive
the geodesic curve evolution equations for a conformal
parameterization. 

\setcounter{equation}{0}
\section{Mathematical formulations} \label{sec:old1}

It is the aim of this paper to introduce numerical schemes for the
situation where a curve $\Gamma = \vec x(I)$ evolves in a
two-dimensional Riemannian manifold that is conformally equivalent to
the Euclidean space. Curvature flow is the $L^2$--gradient flow of the
length functional and we first review how length is defined with
respect to the metric $g$. 
The length induced by (\ref{eq:g}) is defined as
\begin{equation} \label{eq:normg}
[|\vec v|_g](\vec z) = \left([(\vec v, \vec v)_g](\vec z) \right)^\frac12
= g^\frac12(\vec z)\,|\vec v| \quad \forall\ \vec v \in \bR^2
\qquad \text{ for } \vec z \in H\,.
\end{equation}
The distance between two points $\vec z_0$, $\vec z_1$ in $H$ is defined as
\begin{equation} \label{eq:distg}
\distg(\vec z_0, \vec z_1) = \inf \left\{
\int_0^1 [|\vec\gamma_\rho(\rho)|_g](\vec\gamma(\rho)) \drho : 
\vec\gamma \in C^1([0,1], H)\,,\ \vec\gamma(0) = \vec z_0\,,\ \vec\gamma(1)
= \vec z_1
\right\}.
\end{equation}
It can be shown that $(H,\distg)$ is a metric space, see
\citet[\S1.4]{Jost05}.

On recalling (\ref{eq:normg}), 
the total length of the closed curve $\Gamma \subset H$ is given by
\begin{equation} \label{eq:Lg}
L_g(\vec x) = \int_I [|\vec x_\rho|_g](\vec x) \drho
= \int_I g^\frac12(\vec x)\,|\vec x_\rho|\drho\,.
\end{equation}
If $\Gamma=\vec{x}(I)$ encloses a domain $\Omega \subset H$, with 
$\partial\Omega = \Gamma$, we define the total enclosed area as
\begin{equation} \label{eq:Ag}
A_g(\Omega) = \int_{\Omega} g(\vec z) \dz\,.
\end{equation}
For later use we observe that if $\Gamma = \vec x(I)=\partial\Omega$
is parameterized
clockwise, then $\vec\nu \circ \vec x^{-1}$, recall (\ref{eq:tau}), 
denotes the outer normal to
$\Omega$ on $\partial\Omega = \Gamma$. An anti-clockwise parameterization,
on the other hand, yields that $\vec\nu \circ \vec x^{-1}$ is the inner normal.

We remark that if we take (\ref{eq:g}) with
\begin{subequations}
\begin{equation} \label{eq:ghypbol}
g(\vec z) = (\vec z \,.\,\vec\ek_2)^{-2}
\quad \text{ and } \quad
H = \bH^2 := \{ \vec z \in \bR^2 : \vec z \,.\,\vec\ek_2 > 0 \}\,,
\end{equation}
then we obtain the Poincar\'e half-plane model
which serves as a model for the hyperbolic plane.
Clearly,
\begin{equation} \label{eq:geuclid}
g(\vec z) = 1 \quad \text{ and } \quad H = \bR^2 
\end{equation}
simplifies to the standard Euclidean situation.
In the context of the numerical approximation of geometric evolution
equations for axisymmetric surfaces in $\bR^3$, in the recent papers
\cite{aximcf,axisd} the authors considered gradient flows, and their numerical
approximation, of the energy
\begin{equation} \label{eq:A}
A_{\mathcal{S}}(\vec x) = 
2\,\pi\,\int_I \vec x\,.\,\vec\ek_2\,|\vec x_\rho| \drho\,.
\end{equation}
Here we note that as the authors in \cite{aximcf,axisd} considered surfaces
that are rotationally symmetric with respect to the $x_2$--axis, they in fact
considered (\ref{eq:A}) with $\vec\ek_2$ replaced by $\vec\ek_1$.
We note that (\ref{eq:Lg}) collapses to (\ref{eq:A}) for the choice
\begin{equation} \label{eq:ge1}
g(\vec z) = 4\,\pi^2\,(\vec z\,.\,\vec\ek_2)^2
\quad\text{ and } \quad
H = \bH^2\,.
\end{equation}
We also consider more general variants of (\ref{eq:ghypbol}), namely
\begin{equation} \label{eq:gmu}
g(\vec z) = (\vec z \,.\,\vec\ek_2)^{-2\,\mu}\,,\ \mu \in \bR\,,
\quad \text{ and } \quad H = \bH^2 \,,
\end{equation}
\end{subequations}
so that (\ref{eq:ghypbol}) corresponds to $\mu=1$, while formally
(\ref{eq:geuclid}) corresponds to $\mu=0$. As the latter choice leads to a
constant metric, a suitable translation of the initial data in the 
$\vec\ek_2$ direction will ensure that any evolution for (\ref{eq:geuclid}) is
confined to $\bH^2$, and so (\ref{eq:geuclid}) and (\ref{eq:gmu}) with $\mu=0$
are equivalent.
In addition, (\ref{eq:ge1}), up to the 
constant factor $4\,\pi^2$, corresponds to $\mu=-1$. 
For the evolution equations we consider in this paper, the 
constant factor $4\,\pi^2$ will only affect the time scale of the evolutions.

We remark that for $\mu \not= 1$ the metric space $\bH^2$ with
the metric (\ref{eq:distg}) induced by (\ref{eq:gmu}) is not complete.
To see this, we observe that the distance (\ref{eq:distg}) 
between $a\,\vec\ek_2$ and $b\,\vec\ek_2$, for $a<b$, is bounded from above by
\[
\int_a^b u^{-\mu}\,{\rm d}u = (1-\mu)^{-1}\left( b^{1-\mu} - a^{1-\mu}\right).
\]
Hence, in the case $\mu > 1$, 
the distance converges to zero as $a,b\to\infty$,
and so $(n\,\vec\ek_2)_{n\in\bN}$ is a Cauchy sequence without a limit in
$\bH^2$. In the case $\mu < 1$ we can argue similarly for the Cauchy
sequence $(n^{-1}\,\vec\ek_2)_{n\in\bN}$, as its limit $\vec 0 \not\in\bH^2$. 
The Hopf--Rinow theorem, cf.\ \cite{Jost05},  then implies
that the metric space $\bH^2$ with the metric induced by (\ref{eq:gmu}) 
for $\mu \not=1$ is not geodesically complete.
Of course, in the special case $\mu=0$ we can choose $H=\bR^2$ to obtain the
complete Euclidean space, (\ref{eq:geuclid}). 

Further examples are given by the family of metrics
\begin{equation} \label{eq:galpha}
g(\vec z) = \frac4{(1 - \alpha\, |\vec z|^2)^2}
\quad \text{ and } H = \begin{cases}
\bD_{\alpha} 
= \{ \vec z \in \bR^2 : |\vec z| < \alpha^{-\frac12} \}
& \alpha > 0\,, \\
\bR^2 & \alpha \leq 0\,.
\end{cases}
\end{equation} 
see e.g.\ 
\citet[Definition~4.4]{Schippers07}.
We note that (\ref{eq:galpha}) with $\alpha=1$ gives a model for the
hyperbolic disk, see also \citet[Definition~2.7]{KrausR13}. 
The metric (\ref{eq:galpha}) with $\alpha=-1$, on the other hand, models the 
geometry of the elliptic plane. This is obtained by doing a stereographic 
projection of the sphere onto the plane, see (\ref{eq:gstereo}), below, for
more details.

We note that the sectional curvature of $g$, also called the 
Gaussian curvature of $g$, can be computed by
\begin{equation} \label{eq:Gaussg}
S_0(\vec z) = - \frac{\Delta\,\ln g(\vec z)}{2\,g(\vec z)}
\qquad \vec z \in H\,,
\end{equation}
see e.g.\ \citet[Definition~2.4]{KrausR13}. We observe that for (\ref{eq:gmu}) 
it holds that
\begin{equation} \label{eq:mu_S0}
S_0(\vec z) = -\mu\,(\vec z\,.\,\vec\ek_2)^{2\,(\mu-1)} 
\qquad \vec z \in H\,,
\end{equation}
while for (\ref{eq:galpha}) it holds that
\begin{equation} \label{eq:alpha_S0}
S_0(\vec z) =-\alpha\qquad \vec z \in H\,.
\end{equation}
Of special interest are metrics with constant sectional curvature. For example, 
(\ref{eq:geuclid}) gives $S_0=0$, 
(\ref{eq:ghypbol}), i.e.\ (\ref{eq:gmu}) with $\mu=1$, gives 
$S_0=-1$, while (\ref{eq:galpha}) gives 
$S_0=-\alpha$\,.

{From} now on we consider a family of curves $\Gamma(t)$, parameterized by
$\vec x(\cdot,t) : I \to H \subset \bR^2$. It then holds that
\begin{equation}
\ddt\, L_g(\vec x(t)) = \int_I \left[ 
\nabla\,g^\frac12(\vec x)\,.\,\vec x_t + g^\frac12(\vec x)\,
\frac{(\vec x_t)_\rho\,.\,\vec x_\rho}{|\vec x_\rho|^2}
\right] |\vec x_\rho| \drho\,.
\label{eq:dLdt}
\end{equation}

Let
\begin{equation} \label{eq:sg}
\partial_{s_g} = |\vec x_\rho|_g^{-1} \,\partial_\rho =
g^{-\frac12}(\vec x)\,|\vec x_\rho|^{-1} \,\partial_\rho
= g^{-\frac12}(\vec x)\,\partial_s\,.
\end{equation}
We introduce 
\begin{equation} \label{eq:nug}
\vec\nu_g = g^{-\frac12}(\vec x) \,\vec\nu
= - g^{-\frac12}(\vec x) \,\vec x_s^\perp = - \vec x_{s_g}^\perp
\quad \text{and} \quad
\vec\tau_g = \vec x_{s_g}\,,
\end{equation}
so that $\vec\tau_g\,.\,\vec\nu_g =  0$ and
$|\vec\tau_g|_g^2 = |\vec\nu_g|_g^2 = (\vec\nu_g, \vec\nu_g)_g = g(\vec
x)\,\vec\nu_g\,.\,\vec\nu_g=1$, and let
\begin{equation} \label{eq:Vg}
\mathcal{V}_g = (\vec x_t, \vec\nu_g)_g = g^\frac12(\vec x)\,\vec
x_t\,.\,\vec\nu = g^\frac12(\vec x)\,\mathcal{V}\,.
\end{equation}
It follows from (\ref{eq:varkappa}) that
\begin{align}
\nabla\,g^\frac12(\vec x) & =
[\vec\nu\,(\vec\nu\,.\,\nabla) + \vec\tau\,(\vec\tau\,.\,\nabla)]
\,g^\frac12(\vec x) =
\vec\nu\,(\vec\nu\,.\,\nabla)\,g^\frac12(\vec x)
+ \vec\tau\,\frac1{|\vec x_\rho|}\left[g^\frac12(\vec x)\right]_\rho 
\nonumber \\ &
= \vec\nu\,(\vec\nu\,.\,\nabla)\,g^\frac12(\vec x)
+ \frac1{|\vec x_\rho|} \left[
g^\frac12(\vec x)\,\frac{\vec x_\rho}{|\vec x_\rho|}\right]_\rho
- g^\frac12(\vec x)\,\frac1{|\vec x_\rho|}
 \left[ \frac{\vec x_\rho}{|\vec x_\rho|}\right]_\rho
\nonumber \\ &
= \vec\nu\,(\vec\nu\,.\,\nabla)\,g^\frac12(\vec x)
+ \frac1{|\vec x_\rho|}\left[
g^\frac12(\vec x)\,\frac{\vec x_\rho}{|\vec x_\rho|}\right]_\rho
-g^\frac12(\vec x)\,\varkappa\,\vec\nu\,.
\label{eq:ng}
\end{align}
Combining (\ref{eq:dLdt}), (\ref{eq:ng}) and (\ref{eq:Vg}) yields that
\begin{align}
\ddt\, L_g(\vec x(t)) & = 
\int_I \left(\nabla\,g^\frac12(\vec x)
- \frac1{|\vec x_\rho|}
 \left[ g^\frac12(\vec x)\,\frac{\vec x_\rho}{|\vec x_\rho|}\right]_\rho 
\right).\,\vec x_t \,|\vec x_\rho| 
\drho 
\nonumber \\ & 
= \int_I \left[ 
\vec\nu\,.\,\nabla\,g^\frac12(\vec x)
-g^\frac12(\vec x)\,\varkappa \right] \vec\nu\,. \,\vec x_t \,|\vec x_\rho| 
 \drho \nonumber \\ & 
= \int_I \left[ \vec\nu_g\,.\,\nabla\,g^\frac12(\vec x)
- \varkappa\right] \mathcal{V}_g \,|\vec x_\rho| \drho \nonumber \\ & 
= - \int_I g^{-\frac12}(\vec x) \left[\varkappa - 
\vec\nu_g\,.\,\nabla\,g^\frac12(\vec x) \right] \mathcal{V}_g 
\,|\vec x_\rho|_g \drho \nonumber \\ &
= - \int_I \varkappa_g\,\mathcal{V}_g \,|\vec x_\rho|_g \drho \,,
\label{eq:dLdtV}
\end{align}
where, on recalling (\ref{eq:nug}),
\begin{equation} \label{eq:varkappag}
\varkappa_g = g^{-\frac12}(\vec x) \left[\varkappa - 
\vec\nu_g\,.\,\nabla\,g^\frac12(\vec x) \right] 
= g^{-\frac12}(\vec x)\left[\varkappa - 
\tfrac12\,\vec\nu\,.\,\nabla\,\ln g(\vec x) \right].
\end{equation}
Clearly, the curvature $\varkappa_g$ is the first variation of the length
(\ref{eq:Lg}). 

For the metric (\ref{eq:gmu}) we obtain that
\begin{align}
\varkappa_g = (\vec x\,.\,\vec\ek_2)^\mu
\left[ \varkappa + \mu\,\frac{\vec\nu\,.\,\vec\ek_2}{\vec x\,.\,\vec\ek_2}
\right] ,
\label{eq:mu_varkappag} 
\end{align}
while for (\ref{eq:galpha}) we have
\begin{equation} \label{eq:alpha_varkappag}
\varkappa_g = \tfrac12\,(1 - \alpha\,|\vec x|^2)\left[
\varkappa - 2\,\alpha\,(1 - \alpha\,|\vec x|^2)^{-1}\,\vec x\,.\,\vec\nu
\right].
\end{equation}

In addition, combining (\ref{eq:varkappag}), (\ref{eq:nug}) and (\ref{eq:ng}) 
yields that
\begin{equation} \label{eq:gkgnu}
g(\vec x)\,\varkappa_g\,\vec\nu = 
\frac1{|\vec x_\rho|} \left[ g^\frac12(\vec x)\,
\frac{\vec x_\rho}{|\vec x_\rho|}\right]_\rho - \nabla\,g^\frac12(\vec x)\,.
\end{equation}
Weak formulations of (\ref{eq:varkappa}) and (\ref{eq:gkgnu}) will play an
important role in this paper, and so we state them here for later reference.
The natural weak formulation of (\ref{eq:varkappa}) is
\begin{equation} \label{eq:weak_varkappa}
\int_I \varkappa\,\vec\nu\,.\,\vec\eta\, |\vec x_\rho| \drho
+ \int_I (\vec x_\rho\,.\,\vec\eta_\rho)\,|\vec x_\rho|^{-1} \drho = 0
\quad \forall\ \vec\eta \in [H^1(I)]^2\,,
\end{equation}
while a natural weak formulation of (\ref{eq:gkgnu}) is
\begin{equation} \label{eq:weak_gkgnu}
 \int_I g(\vec x)\,\varkappa_g\,\vec\nu\,.\,
\vec\eta\,|\vec x_\rho|\drho
+ \int_I \left[\nabla\,g^\frac12(\vec x)\,.\,\vec\eta
+ g^\frac12(\vec x)\,\frac{\vec x_\rho\,.\,\vec\eta_\rho}{|\vec x_\rho|^2}
\right] |\vec x_\rho| \drho 
= 0 \quad \forall\ \vec\eta \in [H^1(I)]^2\,.
\end{equation}

\subsection{Curvature flow}

It follows from (\ref{eq:dLdtV}) that
\begin{equation} \label{eq:Vgkg}
\mathcal{V}_g = \varkappa_g
\end{equation}
is the natural $L^2$--gradient flow of $L_g$ with respect to the metric induced
by $g$, i.e.\
\begin{equation} \label{eq:gL2gradflow}
\ddt\, L_g(\vec x(t)) + \int_I \varkappa_g^2 \,|\vec x_\rho|_g \drho = 0\,.
\end{equation}
On recalling (\ref{eq:Vg}) and(\ref{eq:varkappag}), 
we can rewrite (\ref{eq:Vgkg}) equivalently as
\begin{equation} \label{eq:Vk}
g(\vec x)\,\vec x_t\,.\,\vec\nu = 
\varkappa - 
\tfrac12\,\vec\nu\,.\,\nabla\,\ln g(\vec x)\,.
\end{equation}

We consider the following weak formulation of (\ref{eq:Vk}). \\ \noindent
$(\BGNmckappa)$:
Let $\vec x(0) \in [H^1(I)]^2$. For $t \in (0,T]$
find $\vec x(t) \in [H^1(I)]^2$ and $\varkappa(t)\in L^2(I)$ such that
(\ref{eq:weak_varkappa}) holds and
\begin{align}
& \int_I g(\vec x)\,\vec x_t\,.\,\vec\nu\,\chi\,|\vec x_\rho|\drho
= \int_I \left(\varkappa - \tfrac12\,\vec\nu\,.\,\nabla\,\ln g(\vec x) 
\right) \chi\,|\vec x_\rho| \drho 
\quad \forall\ i.e.\ \chi \in L^2(I)\,. \label{eq:xtweak} 
\end{align}

An alternative strong formulation of curvature flow to (\ref{eq:Vk}) 
is given by 
\begin{equation} \label{eq:mcdziuk}
g(\vec x)\,\vec x_t = 
\vec\varkappa - \tfrac12\,[\vec\nu\,.\,\nabla\,\ln g(\vec x)]\,\vec\nu\,,
\end{equation}
where we recall (\ref{eq:varkappa}). We observe that (\ref{eq:mcdziuk}) 
fixes $\vec x_t$ to be totally 
in the normal direction, in contrast to (\ref{eq:Vk}). 
We consider the following weak formulation of (\ref{eq:mcdziuk}).
\\ \noindent
$(\GDmckappa)$:
Let $\vec x(0) \in [H^1(I)]^2$. For $t \in (0,T]$
find $\vec x(t) \in [H^1(I)]^2$ and $\vec\varkappa(t)\in [L^2(I)]^2$ such that
\begin{subequations}
\begin{align}
& \int_I g(\vec x)\,\vec x_t\,.\,\vec\chi\,|\vec x_\rho|\drho
= \int_I \left(\vec\varkappa\,.\,\vec\chi 
- \tfrac12\,[\vec\nu\,.\,\nabla\,\ln g(\vec x)]\,\vec\nu\,.\,\vec\chi\right)
|\vec x_\rho| \drho \quad \forall\ \vec\chi \in [L^2(I)]^2\,,
\label{eq:Dziuka} \\
& \int_I \vec\varkappa\,.\,\vec\eta\, |\vec x_\rho| \drho
+ \int_I (\vec x_\rho\,.\,\vec\eta_\rho)\,|\vec x_\rho|^{-1} \drho = 0
\quad \forall\ \vec\eta \in [H^1(I)]^2\,.
\label{eq:Dziukb}
\end{align}
\end{subequations}

In order to develop stable approximations, we
investigate alternative formulations based on (\ref{eq:weak_gkgnu}). 
Firstly, we note that combining (\ref{eq:Vk}) and (\ref{eq:varkappag}) yields
\begin{equation} \label{eq:Vkg}
g(\vec x)\,\vec x_t\,.\,\vec\nu = 
g^\frac12(\vec x)\,\varkappa_g\,.
\end{equation}
We then consider the following weak formulation of (\ref{eq:Vkg}).
\\ \noindent
$(\BGNmc)$:
Let $\vec x(0) \in [H^1(I)]^2$. For $t \in (0,T]$
find $\vec x(t) \in [H^1(I)]^2$ and $\varkappa_g(t) \in L^2(I)$ such that
(\ref{eq:weak_gkgnu}) holds and
\begin{align}
& \int_I g(\vec x)\,\vec x_t\,.\,\vec\nu\,
\chi\,|\vec x_\rho|\drho
= \int_I g^\frac12(\vec x)\,\varkappa_g\,\chi\,
|\vec x_\rho|\drho \quad \forall\ \chi \in L^2(I)\,. \label{eq:bgnnewa} 
\end{align}
Clearly, choosing $\chi = \varkappa_g$ in (\ref{eq:bgnnewa}) and
$\vec\eta=\vec x_t$ in (\ref{eq:weak_gkgnu}) yields (\ref{eq:gL2gradflow}),
on noting (\ref{eq:Vg}), (\ref{eq:nug}) and (\ref{eq:dLdtV}). 

On recalling (\ref{eq:nug}), we introduce
\begin{equation} \label{eq:veckappag}
\vec\varkappa_g = \varkappa_g\,\vec\nu_g = 
g^{-\frac12}(\vec x)\,\varkappa_g\,\vec\nu\,,
\end{equation}
so that an alternative formulation of curvature flow to (\ref{eq:Vgkg}) is
given by
\begin{equation} \label{eq:vecVgkg}
\vec x_t = \mathcal{V}_g\,\vec\nu_g = \vec\varkappa_g\,,
\end{equation}
where we have recalled (\ref{eq:Vg}) and (\ref{eq:nug}).
Similarly to (\ref{eq:mcdziuk}), the flow (\ref{eq:vecVgkg}) is again totally
in the normal direction.
On recalling (\ref{eq:weak_gkgnu}) and (\ref{eq:veckappag}), 
we consider the following weak formulation of (\ref{eq:vecVgkg}). 
\\ \noindent
$(\GDmc)$:
Let $\vec x(0) \in [H^1(I)]^2$. For $t \in (0,T]$
find $\vec x(t) \in [H^1(I)]^2$ 
and $\vec\varkappa_g(t) \in [L^2(I)]^2$ such that
\begin{subequations}
\begin{align}
& \int_I g(\vec x)\,\vec x_t\,.\,\vec\chi\,|\vec x_\rho|\drho
= \int_I g(\vec x)\,\vec\varkappa_g\,.\,\vec\chi\,
|\vec x_\rho|\drho \quad\forall\ \vec\chi \in [L^2(I)]^2\,,
\label{eq:Dziuknewa} \\ &
\int_I g^\frac32(\vec x)\,\vec\varkappa_g\,.\,\vec\eta\,|\vec x_\rho|\drho
+ \int_I \left[\nabla\,g^\frac12(\vec x)\,.\,\vec\eta
+ g^\frac12(\vec x)\,\frac{\vec x_\rho\,.\,\vec\eta_\rho}{|\vec x_\rho|^2}
\right] |\vec x_\rho| \drho 
= 0 \quad \forall\ \vec\eta \in [H^1(I)]^2\,.
\label{eq:Dziuknewb} 
\end{align}
\end{subequations}
Choosing $\vec\chi = g^\frac12(\vec x)\,\vec\varkappa_g$ 
in (\ref{eq:Dziuknewa}) and
$\vec\eta=\vec x_t$ in (\ref{eq:Dziuknewb}) yields 
\begin{equation} \label{eq:gL2gradflow2}
\ddt\, L_g(\vec x(t)) + \int_I 
g^{\frac32}(\vec x)\,|\vec\varkappa_g|^2\,|\vec x_\rho| \drho = 0\,,
\end{equation}
which is equivalent to (\ref{eq:gL2gradflow}),
on recalling (\ref{eq:dLdtV}), (\ref{eq:veckappag}) and (\ref{eq:normg}). 

We observe that the variable $\varkappa_g$ can be
eliminated from $(\BGNmc)$, by choosing 
$\chi = g^{\frac12}(\vec x)\,\vec\nu\,.\,\vec\eta$ 
in (\ref{eq:bgnnewa}), and then combining
(\ref{eq:bgnnewa}) and (\ref{eq:weak_varkappa}), to yield 
\begin{align} \label{eq:gL2flowbgn}
& \int_I g^\frac32(\vec x)\,(\vec x_t\,.\,\vec\nu)\,(\vec\eta\,.\,\vec\nu)\,
|\vec x_\rho| \drho
+ \int_I \left[\nabla\,g^\frac12(\vec x)\,.\,\vec\eta
+ g^\frac12(\vec x)\,\frac{\vec x_\rho\,.\,\vec\eta_\rho}{|\vec x_\rho|^2}
\right] |\vec x_\rho| \drho 
 = 0 \nonumber \\ & \hspace{11cm}
\quad \forall\ \vec\eta \in [H^1(I)]^2\,.
\end{align}
Similarly, $\vec\varkappa_g$ can be eliminated from $(\GDmc)$
by choosing $\vec\chi = g^\frac12(\vec x)\,\vec\eta$ in (\ref{eq:Dziuknewa}) 
to yield
\begin{equation} \label{eq:gL2flow}
\int_I g^\frac32(\vec x)\,\vec x_t\,.\,\vec\eta\,|\vec x_\rho| \drho
+ \int_I \left[\nabla\,g^\frac12(\vec x)\,.\,\vec\eta
+ g^\frac12(\vec x)\,\frac{\vec x_\rho\,.\,\vec\eta_\rho}{|\vec x_\rho|^2}
\right] |\vec x_\rho| \drho = 0
\quad \forall\ \vec\eta \in [H^1(I)]^2\,.
\end{equation}

\subsection{Curve diffusion}
We consider the flow
\begin{equation} \label{eq:sdg}
\mathcal{V}_g = - (\varkappa_g)_{s_gs_g}
= - g^{-\frac12}(\vec x)\left[ g^{-\frac12}(\vec x)\left[
\varkappa_g \right]_s\right]_s
= - \frac1{g^{\frac12}(\vec x)\,|\vec x_\rho|}
\left[ \frac{[\varkappa_g]_\rho}{g^{\frac12}(\vec x)\,|\vec
x_\rho|}\right]_\rho\,,
\end{equation}
where we have recalled (\ref{eq:sg}). On noting (\ref{eq:dLdtV}), 
and similarly to (\ref{eq:gL2gradflow}), it follows that (\ref{eq:sdg}) 
is the natural $H^{-1}$--gradient flow of $L_g$ with respect to the metric 
induced by $g$, i.e.\
\begin{equation} \label{eq:gH-1gradflow}
\ddt\, L_g(\vec x(t)) + \int_I (\partial_{s_g}\,\varkappa_g)^2 
\,|\vec x_\rho|_g \drho = 0\,.
\end{equation}
Moreover, if $\Gamma(t) = \vec x(I,t)$ encloses a domain
$\Omega(t) \subset \bR^2$, with $\vec\nu \circ \vec x^{-1}$ denoting
the outer normal on $\partial\Omega(t) = \Gamma(t)$, on
recalling (\ref{eq:Ag}), (\ref{eq:normg}) and (\ref{eq:Vg}), it
follows from a transport theorem, see e.g.\
\citet[(2.22)]{DeckelnickDE05}, that
\begin{align}
\ddt\,A_g(\Omega(t)) = \ddt\,\int_{\Omega(t)} g(\vec z) \dz
= \int_I g(\vec x)\,\mathcal{V}\,|\vec x_\rho| \drho
= \int_I \mathcal{V}_g\,|\vec x_\rho|_g \drho\,.
\label{eq:dAgdt}
\end{align}
Hence solutions to (\ref{eq:sdg}) satisfy, on noting (\ref{eq:normg}), that
\begin{align}
\ddt\,A_g(\Omega(t)) = - \int_I (\varkappa_g)_{s_gs_g}\,|\vec x_\rho|_g \drho
= - \int_I \left[ g^{-\frac12}(\vec x)\,[\varkappa_g]_\rho
\right]_\rho \drho = 0\,,
\label{eq:dAgdt0}
\end{align}
and so the total enclosed area is preserved.

Our weak formulations are going be to based on the equivalent equation
\begin{equation} \label{eq:sdg2}
g(\vec x)\,\vec x_t\,.\,\vec\nu = 
- \frac1{|\vec x_\rho|} 
\left(  \frac{[\varkappa_g]_\rho}{g^{\frac12}(\vec x)\,|\vec
x_\rho|}\right)_\rho ,
\end{equation}
recall (\ref{eq:Vg}). 

We consider the following weak formulation of (\ref{eq:sdg2}),
on recalling (\ref{eq:varkappag}). \\ \noindent
$(\BGNsd)$:
Let $\vec x(0) \in [H^1(I)]^2$. For $t \in (0,T]$
find $\vec x(t) \in [H^1(I)]^2$ and $\varkappa(t)\in H^1(I)$ such that
(\ref{eq:weak_varkappa}) holds and
\begin{align}
& \int_I g(\vec x)\,\vec x_t\,.\,\vec\nu\,\chi\,|\vec x_\rho|\drho
= \int_I g^{-\frac12}(\vec x)\,
\left(g^{-\frac12}(\vec x)\left[\varkappa 
- \tfrac12\,\vec\nu\,.\,\nabla\,\ln g(\vec x) \right]
 \right)_\rho \chi_\rho\,|\vec x_\rho|^{-1} \drho 
\nonumber \\ & \hspace{9cm}
\quad \forall\ \chi \in H^1(I)\,. \label{eq:sdweaka} 
\end{align}

We also introduce the following alternative weak formulation 
for (\ref{eq:sdg2}), which treats the curvature $\varkappa_g$ as an unknown.  
\\ \noindent
$(\BGNsdstab)$:
Let $\vec x(0) \in [H^1(I)]^2$. For $t \in (0,T]$
find $\vec x(t) \in [H^1(I)]^2$ and $\varkappa_g(t) \in H^1(I)$ such that
(\ref{eq:weak_gkgnu}) holds and 
\begin{align}
& \int_I g(\vec x)\,(\vec x_t\,.\,\vec\nu)\,
\chi\,|\vec x_\rho|\drho
= \int_I g^{-\frac12}(\vec x)\,[\varkappa_g]_\rho\,\chi_\rho\,
|\vec x_\rho|^{-1} \drho \quad \forall\ \chi \in H^1(I)\,. \label{eq:sdstaba} 
\end{align}
Choosing $\chi = \varkappa_g$ in (\ref{eq:sdstaba}) and
$\vec\eta=\vec x_t$ in (\ref{eq:weak_gkgnu}) yields 
that (\ref{eq:gH-1gradflow}) holds, on noting from (\ref{eq:sg})
that
\begin{equation} \label{eq:Fstab}
(\partial_{s_g}\,\varkappa_g)^2 \,|\vec x_\rho|_g 
= g^{-1}(\vec x)\,|\vec x_\rho|^{-2}\,(\partial_\rho\,\varkappa_g)^2\,
g^{\frac12}(\vec x)\,|\vec x_\rho| = g^{-\frac12}(\vec x)\,
(\partial_\rho\,\varkappa_g)^2\,|\vec x_\rho|^{-1}\,.
\end{equation}

\subsection{Elastic flow}
Here we consider an appropriate $L^2$--gradient flow of the elastic energy
$W_g(\vec x)$, where on recalling (\ref{eq:veckappag}), (\ref{eq:g}),
(\ref{eq:varkappag}) and (\ref{eq:normg}), we set
\begin{align} \label{eq:Wg}
W_g(\vec x) & = \tfrac12\,\int_I |\vec\varkappa_g|_g^2 \,|\vec x_\rho|_g \drho
= \tfrac12\,\int_I \varkappa_g^2 \,|\vec x_\rho|_g \drho \nonumber \\ & 
= \tfrac12\,\int_I g^{-\frac12}(\vec x)
\left(\varkappa - \tfrac12\,\vec\nu\,.\,\nabla\,\ln g(\vec x) \right)^2 
|\vec x_\rho| \drho 
= \tfrac12\,\int_I g^{-\frac12}(\vec x)\,
\widetilde\varkappa_g^2\, |\vec x_\rho| \drho\,.
\end{align}
In the above, on recalling (\ref{eq:varkappag}), we have defined 
\begin{equation} \label{eq:tildekappa}
\widetilde\varkappa_g = g^\frac12(\vec x)\,\varkappa_g 
=\varkappa - \myz\,,\quad \text{with}\quad \myz 
= \vec\nu_g\,.\,\nabla\,g^\frac12(\vec x) 
= \tfrac12\,\vec\nu\,.\,\nabla\,\ln g(\vec x)\,.
\end{equation}
In the following we often omit the dependence of $g$ on $\vec x$, and we
simply write $g$ for $g(\vec x)$ and so on. 
It follows from (\ref{eq:Wg}), (\ref{eq:tildekappa}), (\ref{eq:Vg}) and
(\ref{eq:normg}) that
\begin{align}
& \ddt\,W_g(\vec x(t)) \nonumber \\ & \quad 
= \tfrac12\,\int_I (g^{-\frac12})_t\, \widetilde\varkappa_g^2 \,
|\vec x_\rho| \drho 
+ \int_I g^{-\frac12}\, (\widetilde\varkappa_g)_t\,\widetilde\varkappa_g\,
|\vec x_\rho| \drho
+ \tfrac12\,\int_I g^{-\frac12}\, \widetilde\varkappa_g^2 \,
\vec x_\rho\,.\,(\vec x_t)_\rho \, |\vec x_\rho|^{-1} \drho 
\nonumber \\ & \quad
= \tfrac12\,\int_I (\vec x_t\,.\,\nabla\,g^{-\frac12})\,
 \widetilde\varkappa_g^2 \, |\vec x_\rho| \drho 
+ \int_I g^{-\frac12}\, (\widetilde\varkappa_g)_t\,\widetilde\varkappa_g\,
|\vec x_\rho| \drho
- \tfrac12\,\int_I (g^{-\frac12}\, \widetilde\varkappa_g^2 \,
\vec x_s)_s\,.\,\vec x_t \, |\vec x_\rho| \drho
\nonumber \\ & \quad
= \tfrac12\,\int_I (\vec x_t\,.\,\nabla\,g^{-\frac12})\, 
\widetilde\varkappa_g^2 \, |\vec x_\rho| \drho 
+ \int_I g^{-\frac12}\, (\widetilde\varkappa_g)_t\,\widetilde\varkappa_g\,
|\vec x_\rho| \drho
\nonumber \\ & \qquad\quad
- \tfrac12\,\int_I \left[ 
(\vec x_s\,.\,\nabla\,g^{-\frac12})\,\widetilde\varkappa_g^2 \,\vec x_s
+ 2\,g^{-\frac12}\,\widetilde\varkappa_g\,(\widetilde\varkappa_g)_s\,\vec x_s
+ g^{-\frac12}\,\widetilde\varkappa_g^2 \,\varkappa\,\vec\nu\right]
.\,\vec x_t \, |\vec x_\rho| \drho \,.
\nonumber \\ & \quad
= \tfrac12\,\int_I (\vec\nu\,.\,\nabla\,g^{-\frac12})\, 
\widetilde\varkappa_g^2 \,\mathcal{V}\, |\vec x_\rho| \drho 
+ \int_I g^{-\frac12}\,\widetilde\varkappa_g\,[ 
(\widetilde\varkappa_g)_t - \,(\widetilde\varkappa_g)_s\,\vec x_s\,.\,\vec x_t]
\,|\vec x_\rho| \drho
\nonumber \\ & \qquad\quad
- \tfrac12\,\int_I g^{-\frac12}\,\widetilde\varkappa_g^2 \,\varkappa\,
\mathcal{V}\, |\vec x_\rho| \drho 
\nonumber \\ & \quad
= \tfrac12\,\int_I \left[(\vec\nu\,.\,\nabla\,g^{-\frac12})
- g^{-\frac12}\,\varkappa \right] 
\varkappa_g^2 \,\mathcal{V}_g\, |\vec x_\rho|_g \drho 
+ \int_I \varkappa_g\,[ 
(\widetilde\varkappa_g)_t - \,(\widetilde\varkappa_g)_s\,\vec x_s\,.\,\vec x_t]
\,|\vec x_\rho| \drho
\,.
\label{eq:T123}
\end{align}
We have from (\ref{eq:tildekappa}) that
\begin{equation} \label{eq:nung12}
\varkappa_g - g^{-\frac12}\,\varkappa = 
- \tfrac12\, g^{-\frac12}\,\vec\nu\,.\,\nabla\,\ln g = 
\vec\nu\,.\,\nabla\,g^{-\frac12}\,,
\end{equation}
and so it follows from (\ref{eq:T123}) that
\begin{align}
\ddt\,W_g(\vec x(t)) & = \tfrac12\,\int_I \left[ \varkappa_g 
- 2\,g^{-\frac12}\,\varkappa \right] 
\varkappa_g^2 \,\mathcal{V}_g\, |\vec x_\rho|_g \drho 
+ \int_I \varkappa_g\,[ 
(\widetilde\varkappa_g)_t - \,(\widetilde\varkappa_g)_s\,\vec x_s\,.\,\vec x_t]
\,|\vec x_\rho| \drho \,. 
\label{eq:T45}
\end{align}
In order to deal with the last integral in (\ref{eq:T45}), we observe the
following. It follows from (\ref{eq:tau}), (\ref{eq:varkappa}) 
and (\ref{eq:Vg}) that
\begin{subequations}
\begin{align} 
\vec\nu_s & = - \varkappa\,\vec x_s\,,\quad 
\vec\nu_{ss} = - \varkappa_s\,\vec x_s - \varkappa^2\,\vec\nu\,,
\label{eq:nuss} \\
\vec\nu_t & = - ((\vec x_s)_t\,.\,\vec\nu)\,\vec x_s 
= - ((\vec x_\rho\,|\vec x_\rho|^{-1})_t\,.\,\vec\nu)\,\vec x_s 
= - ((\vec x_t)_s\,.\,\vec\nu)\,\vec x_s\,, \label{eq:nut} \\
\vec\nu_t - (\vec x_s\,.\,\vec x_t)\,\vec\nu_s &=- \mathcal{V}_s\,\vec x_s\,.
\label{eq:nuts}
\end{align}
\end{subequations}
Combining (\ref{eq:nuss},b) yields, on recalling (\ref{eq:Vg}), that
\begin{align}
\varkappa_t &= - (\vec x_s)_t \,.\,\vec\nu_s - \vec x_s\,.\,(\vec\nu_s)_t
= \varkappa\,(\vec x_s)_t \,.\,\vec x_s - \vec x_s\,.\,(\vec\nu_s)_t
= - \vec x_s\,.\,(\vec\nu_s)_t
=  - \vec x_s\,.\,(\vec\nu_\rho\,|\vec x_\rho|^{-1})_t \nonumber \\ & 
=  - \vec x_s\,.\,(\vec\nu_t)_s + (\vec x_s\,.\,\vec\nu_s)\,
\vec x_s\,.\,(\vec x_t)_s
=  - \vec x_s\,.\,(\vec\nu_t)_s - \varkappa\, \vec x_s\,.\,(\vec x_t)_s
=  - \vec x_s\,.\,(\vec\nu_t)_s + \vec\nu_s\,.\,(\vec x_t)_s
\nonumber \\ &
= \vec x_s\,.\left[ ((\vec x_t)_s\,.\,\vec\nu)\,\vec x_s \right]_s
+ \vec\nu_s\,.\,(\vec x_t)_s
= ((\vec x_t)_s\,.\,\vec\nu)_s + \vec\nu_s\,.\,(\vec x_t)_s
\nonumber \\ &
= (\vec x_t\,.\,\vec\nu)_{ss} - (\vec x_t\,.\,\vec\nu_s)_s
+ \vec\nu_s\,.\,(\vec x_t)_s
= (\vec x_t\,.\,\vec\nu)_{ss} - \vec x_t\,.\,\vec\nu_{ss}\
= (\vec x_t\,.\,\vec\nu)_{ss} 
+ \vec x_t\,.\,[\varkappa_s\,\vec x_s + \varkappa^2\,\vec\nu]
\nonumber \\ &
= \mathcal{V}_{ss} + \varkappa^2\,\mathcal{V} 
+ \varkappa_s\,\vec x_s\,.\,\vec x_t\,,
\label{eq:kappat}
\end{align}
compare also with \citet[(A.3)]{nsns2phase}.
It follows from (\ref{eq:tildekappa}), (\ref{eq:kappat}) 
and (\ref{eq:nuts}) that
\begin{align} \label{eq:tildekappat}
& (\widetilde\varkappa_g)_t - \,(\widetilde\varkappa_g)_s\,\vec x_s\,.\,\vec x_t
= \varkappa_t - \varkappa_s\,\vec x_s\,.\,\vec x_t
- (\myz_t - \myz_s\,\vec x_s\,.\,\vec x_t)
= \mathcal{V}_{ss} + \varkappa^2\,\mathcal{V} 
- (\myz_t - \myz_s\,\vec x_s\,.\,\vec x_t) \nonumber \\ & \quad
= \mathcal{V}_{ss} + \varkappa^2\,\mathcal{V} 
- \tfrac12\,(\vec\nu_t - (\vec x_s\,.\,\vec x_t)\,\vec\nu_s)\,.\,\nabla\,\ln g
- \tfrac12\,( (\nabla\,\ln g)_t - (\vec x_s\,.\,\vec x_t)\,(\nabla\,\ln g)_s)
\,.\,\vec \nu \nonumber \\ & \quad
= \mathcal{V}_{ss} + \varkappa^2\,\mathcal{V} 
+ \tfrac12\, \mathcal{V}_s\,\vec x_s\,.\,\nabla\,\ln g
- \tfrac12\,( (\nabla\,\ln g)_t - (\vec x_s\,.\,\vec x_t)\,(\nabla\,\ln g)_s)
\,.\,\vec \nu \,.
\end{align}
Combining (\ref{eq:T45}) and (\ref{eq:tildekappat}) yields,
on noting (\ref{eq:Vg}), (\ref{eq:normg}), (\ref{eq:varkappa}),
(\ref{eq:tildekappa}) and (\ref{eq:sg}), that
\begin{align}
& \ddt\,W_g(\vec x(t)) = \tfrac12\,\int_I \left[ \varkappa_g 
- 2\,g^{-\frac12}\,\varkappa \right] 
\varkappa_g^2 \,\mathcal{V}_g\, |\vec x_\rho|_g \drho 
+ \int_I \varkappa_g\,[ 
\mathcal{V}_{ss} + \varkappa^2\,\mathcal{V} ]\,|\vec x_\rho| \drho
\nonumber \\ & \hspace{3cm}
+ \tfrac12\,\int_I \varkappa_g\,[
\mathcal{V}_s\,\vec x_s\,.\,\nabla\,\ln g
- ( (\nabla\,\ln g)_t - (\vec x_s\,.\,\vec x_t)\,(\nabla\,\ln g)_s)
\,.\,\vec \nu ]\,|\vec x_\rho| \drho \nonumber \\
& = \tfrac12\,\int_I \left[ \varkappa_g 
- 2\,g^{-\frac12}\,\varkappa \right] 
\varkappa_g^2 \,\mathcal{V}_g\, |\vec x_\rho|_g \drho 
+ \int_I [(\varkappa_g)_{ss} + \varkappa^2\,\varkappa_g]\,\mathcal{V} 
\,|\vec x_\rho| \drho \nonumber \\ & \quad
- \tfrac12\,\int_I [(\varkappa_g)_s\,\vec x_s\,.\,\nabla\,\ln g 
+ \varkappa_g\,\vec x_{ss}\,.\,\nabla\,\ln g 
+ \varkappa_g\,\vec x_s\,.\,(\nabla\,\ln g)_s
]\,\mathcal{V}\,|\vec x_\rho| \drho
\nonumber \\ & \quad
- \tfrac12\,\int_I \varkappa_g\,[ 
( (\nabla\,\ln g)_t - (\vec x_s\,.\,\vec x_t)\,(\nabla\,\ln g)_s)
\,.\,\vec \nu ]\,|\vec x_\rho| \drho \nonumber \\
& = \tfrac12\,\int_I \left[ \varkappa_g 
- 2\,g^{-\frac12}\,\varkappa \right] 
\varkappa_g^2 \,\mathcal{V}_g\, |\vec x_\rho|_g \drho 
+ \int_I g^{-1}\,[(\varkappa_g)_{ss} + \varkappa^2\,\varkappa_g]\,\mathcal{V}_g 
\,|\vec x_\rho|_g \drho \nonumber \\ & \quad
- \tfrac12\,\int_I [(\varkappa_g)_s\,(\ln g)_s
+ 2\,\varkappa_g\,\varkappa\,\myz 
+ \varkappa_g\,\vec x_s\,.\,(\nabla\,\ln g)_s
]\,\mathcal{V}\,|\vec x_\rho| \drho
\nonumber \\ & \quad
- \tfrac12\,\int_I \varkappa_g\,[ 
( (\nabla\,\ln g)_t - (\vec x_s\,.\,\vec x_t)\,(\nabla\,\ln g)_s)
\,.\,\vec \nu ]\,|\vec x_\rho| \drho \nonumber \\
& =\int_I \left[  \tfrac12\,\varkappa_g^3
- g^{-\frac12}\,\varkappa\,\varkappa_g^2 + (\varkappa_g)_{s_gs_g} 
- (g^{-\frac12})_s\,(\varkappa_g)_{s_g}
+ g^{-1}\,\varkappa^2\,\varkappa_g 
\right]\mathcal{V}_g 
\,|\vec x_\rho|_g \drho \nonumber \\ & \quad
+ \int_I [(\varkappa_g)_{s_g}\,(g^{-\frac12})_s
- g^{-1}\,\varkappa_g\,\varkappa\,\myz ]\,\mathcal{V}_g\,|\vec x_\rho|_g \drho
\nonumber \\ & \quad
- \tfrac12\,\int_I \varkappa_g\,[ 
( (\nabla\,\ln g)_t - (\vec x_s\,.\,\vec x_t)\,(\nabla\,\ln g)_s)
\,.\,\vec \nu + \vec x_s\,.\,(\nabla\,\ln g)_s\,\mathcal{V}]
\,|\vec x_\rho| \drho \nonumber \\
& =\int_I \left[ (\varkappa_g)_{s_gs_g} + \tfrac12\,\varkappa_g^3
- g^{-1}\,\varkappa\,\varkappa_g \left( g^{\frac12}\,\varkappa_g
- \varkappa + \myz \right)
\right]\mathcal{V}_g 
\,|\vec x_\rho|_g \drho \nonumber \\ & \quad
- \tfrac12\,\int_I \varkappa_g\,[ 
( (\nabla\,\ln g)_t - (\vec x_s\,.\,\vec x_t)\,(\nabla\,\ln g)_s)
\,.\,\vec \nu + \vec x_s\,.\,(\nabla\,\ln g)_s\,\mathcal{V}]
\,|\vec x_\rho| \drho \nonumber \\
& =\int_I \left[ (\varkappa_g)_{s_gs_g} + \tfrac12\,\varkappa_g^3
\right]\mathcal{V}_g \,|\vec x_\rho|_g \drho \nonumber \\ & \quad
- \tfrac12\,\int_I \varkappa_g\,[ 
( (\nabla\,\ln g)_t - (\vec x_s\,.\,\vec x_t)\,(\nabla\,\ln g)_s)
\,.\,\vec \nu + \vec x_s\,.\,(\nabla\,\ln g)_s\,\mathcal{V}]
\,|\vec x_\rho| \drho\,. 
\label{eq:T67}
\end{align}
It remains to deal with the final integral in (\ref{eq:T67}). To this end,
we note that
\begin{align}
& ( (\nabla\,\ln g)_t - (\vec x_s\,.\,\vec x_t)\,(\nabla\,\ln g)_s)
\,.\,\vec \nu + \vec x_s\,.\,(\nabla\,\ln g)_s\,\mathcal{V}
\nonumber \\ & \quad
= \vec\nu\,.\,(D^2\,\ln g)\,\vec x_t - (\vec x_s\,.\,\vec x_t)\,\vec\nu\,.\,
(D^2\,\ln g)\,\vec x_s + \mathcal{V}\,\vec x_s\,.\,(D^2\,\ln g)\,\vec x_s
\nonumber \\ & \quad
= \mathcal{V}\,\vec\nu\,.\,(D^2\,\ln g)\,\vec \nu
+ \mathcal{V}\,\vec x_s\,.\,(D^2\,\ln g)\,\vec x_s
= \mathcal{V}\,\Delta\,\ln g\,.
\label{eq:S0a}
\end{align}
Combining (\ref{eq:T67}) and (\ref{eq:S0a}) yields, on noting
(\ref{eq:Vg}), (\ref{eq:normg}) and (\ref{eq:Gaussg}), that
\begin{align}
\ddt\,W_g(\vec x(t)) & = 
\int_I \left[ (\varkappa_g)_{s_gs_g} + \tfrac12\,\varkappa_g^3
\right]\mathcal{V}_g \,|\vec x_\rho|_g \drho 
- \tfrac12\,\int_I \varkappa_g\,(\Delta\,\ln g)\,\mathcal{V}\,
|\vec x_\rho| \drho \nonumber \\ & 
= \int_I \left[ (\varkappa_g)_{s_gs_g} + \tfrac12\,\varkappa_g^3
+ S_0(\vec x)\,\varkappa_g
\right]\mathcal{V}_g \,|\vec x_\rho|_g \drho \,. 
\label{eq:newT67}
\end{align}
It follows from (\ref{eq:newT67}) that elastic flow is given by
\begin{equation} \label{eq:g_elastflow}
\mathcal{V}_g = - (\varkappa_g)_{s_gs_g} - \tfrac12\,\varkappa_g^3 
- S_0(\vec x)\,\varkappa_g\,.
\end{equation}

\begin{rem} \label{rem:geuclid}
We note that in the special case (\ref{eq:geuclid}),
it follows from (\ref{eq:varkappag}) and (\ref{eq:Gaussg}) 
that $\varkappa_g = \varkappa$ and $S_0=0$, 
and so (\ref{eq:g_elastflow}) collapses to
\begin{equation} \label{eq:euclid_elastflow}
\mathcal{V} = - \varkappa_{ss} - \tfrac12\,\varkappa^3\,,
\end{equation}
i.e.\ to elastic flow in the Euclidean plane,
compare e.g.\ \citet[(1.8)]{willmore}. 
\end{rem}

\begin{rem} \label{rem:DAS17}
In the special case (\ref{eq:ghypbol}), i.e.\ (\ref{eq:gmu}) with $\mu=1$, it
follows from (\ref{eq:mu_S0}) that sectional curvature $S_0=-1$ is constant,
and so (\ref{eq:g_elastflow}) collapses to
\begin{equation} \label{eq:hb_elastflow}
\mathcal{V}_g = - (\varkappa_g)_{s_gs_g} - \tfrac12\,\varkappa_g^3 
+ \varkappa_g\,,
\end{equation}
which is also called hyperbolic elastic flow.
In order to show that (\ref{eq:hb_elastflow}) is equivalent to
(5) in \cite{DallAcquaS17preprint},
for the length parameter $\lambda = 0$, i.e.\ to
\begin{equation} \label{eq:DAS17}
\vec x_t = - (\nabsg^\perp)^2\,\vec\varkappa_g 
- \tfrac12\,|\vec\varkappa_g|_g^2\,\vec\varkappa_g + \vec\varkappa_g
= - (\nabsg^\perp)^2\,\vec\varkappa_g 
- \tfrac12\,\varkappa_g^3\,\vec\nu_g + \varkappa_g\,\vec\nu_g\,,
\end{equation}
we make the following observations. 
It follows from  (\ref{eq:veckappag}), (\ref{eq:mu_varkappag}) for $\mu=1$, 
(\ref{eq:sg}), (\ref{eq:nug}) and $\vec\ek_2^\perp = \vec\ek_1$ that
\begin{align}
\vec\varkappa_g & 
= (\vec x\,.\,\vec\ek_2)^2\left[\varkappa + 
\frac{\vec\nu\,.\,\vec\ek_2}{\vec x\,.\,\vec\ek_2}\right]\vec\nu
= \vec x\,.\,\vec\ek_2 \left[ (\vec x\,.\,\vec\ek_2)\,\vec x_s \right]_s
- (\vec x\,.\,\vec\ek_2)\,(\vec x_s\,.\,\vec\ek_2)\,\vec x_s
+ (\vec x\,.\,\vec\ek_2)\,(\vec\nu\,.\,\vec\ek_2)\,\vec\nu \nonumber \\ &
= \vec x_{s_gs_g} - (\vec x\,.\,\vec\ek_2)\left[
(\vec x_s\,.\,\vec\ek_2)\,\vec x_s
- (\vec\nu\,.\,\vec\ek_2)\,\vec\nu \right] 
= \vec x_{s_gs_g} - (\vec x\,.\,\vec\ek_2)^{-1}\left[
(\vec x_{s_g}\,.\,\vec\ek_2)\,\vec x_{s_g}
- (\vec\nu_g\,.\,\vec\ek_2)\,\vec\nu_g \right] \nonumber \\ &
= \vec x_{s_gs_g} - (\vec x\,.\,\vec\ek_2)^{-1}\left[
(\vec x_{s_g}\,.\,\vec\ek_2)\,\vec x_{s_g}
- (\vec x_{s_g}^\perp\,.\,\vec\ek_2)\,\vec x_{s_g}^\perp \right] \nonumber \\ &
= \vec x_{s_gs_g} - (\vec x\,.\,\vec\ek_2)^{-1}\left[
(\vec x_{s_g}\,.\,\vec\ek_2)\,\vec x_{s_g}
+ (\vec x_{s_g}\,.\,\vec\ek_1)\,\vec x_{s_g}^\perp \right] \nonumber \\ &
= \vec x_{s_gs_g} + (\vec x\,.\,\vec\ek_2)^{-1}\left[
- 2\,(\vec x_{s_g}\,.\,\vec\ek_1)\,(\vec x_{s_g}\,.\,\vec\ek_2)\,\vec\ek_1
+ \left( (\vec x_{s_g}\,.\,\vec\ek_1)^2 - (\vec x_{s_g}\,.\,\vec\ek_2)^2 
\right) \vec\ek_2 \right] ,
\label{eq:hb_veckappag2}
\end{align}
which agrees with \citet[(12)]{DallAcquaS17preprint}. Alternatively, one can
also write (\ref{eq:hb_veckappag2}), on noting the last equation on its 
second line, as 
\begin{equation} \label{eq:dagger}
\vec\varkappa_g = \nabsg\,\vec x_{s_g}\,,
\end{equation}
where the covariant derivative is defined by
\begin{equation} \label{eq:covariantderiv}
\nabsg\,\vec f = \vec f_{s_g} + 
(\vec x\,.\,\vec\ek_2)^{-1}\left[
(\vec f\,.\,\vec\ek_1)\,\vec\nu_g - (\vec f\,.\,\vec\ek_2)\,\vec x_{s_g}
 \right] ,
\end{equation}
on recalling 
(\ref{eq:nug}) and that $\vec\ek_1^\perp = -\vec\ek_2$.
We remark that (\ref{eq:dagger}) agrees with the expression under (1) in
\cite{DallAcquaS17preprint}, on noting the expression for $\nabsg$ on
the top of page 5 in \cite{DallAcquaS17preprint}.
In addition, we define
\begin{equation} \label{eq:covperp}
\nabsg^\perp\,\vec f = \nabsg\,\vec f - (\nabsg\,\vec f, \vec x_{s_g})_g\,
\vec x_{s_g} = (\nabsg\,\vec f , \vec\nu_g)_g\,\vec\nu_g\,,
\end{equation}
see \citet[(13)]{DallAcquaS17preprint}. It follows from
(\ref{eq:covperp}) and (\ref{eq:covariantderiv}), on recalling (\ref{eq:nug}), 
that
\begin{equation} \label{eq:hb_covperp}
\nabsg^\perp\,\vec f = \left[(\vec f_{s_g}, \vec\nu_g)_g
+ (\vec x\,.\,\vec\ek_2)^{-1}\,(\vec f\,.\,\vec\ek_1)\right]\vec\nu_g\,.
\end{equation}
We now compute $\nabsg^\perp\,\vec\varkappa_g$. 
On recalling (\ref{eq:nug}) and (\ref{eq:veckappag}), 
we have that $(\vec\varkappa_g)_{s_g} = (\varkappa_g\,\vec\nu_g)_{s_g} = 
(\vec x\,.\,\vec\ek_2\,\varkappa_g\,\vec\nu)_{s_g}$, and so, on recalling
(\ref{eq:sg}), we have that
\begin{align}
((\vec\varkappa_g)_{s_g}, \vec\nu_g)_g
& = (\vec x\,.\,\vec\ek_2)^{-1}
\left[ (\vec x\,.\,\vec\ek_2)\,\varkappa_g\,\vec\nu\right]_{s_g}\vec\nu
= (\vec x\,.\,\vec\ek_2)^{-1} 
\left[ (\vec x\,.\,\vec\ek_2)\,\varkappa_g\right]_{s_g}
= (\varkappa_g)_{s_g} + \vec x_s\,.\,\vec\ek_2\,\varkappa_g\,.
\label{eq:hb_kgsgnugg}
\end{align}
Hence it follows from (\ref{eq:hb_covperp}), (\ref{eq:hb_kgsgnugg}),
(\ref{eq:nug}), (\ref{eq:veckappag}) and (\ref{eq:tau}) that
\begin{align}
\nabsg^\perp\,\vec\varkappa_g & = 
\left[ (\varkappa_g)_{s_g} + \left[ \vec x_s\,.\,\vec\ek_2
+ (\vec x\,.\,\vec\ek_2)^{-1}\,\vec\nu_g\,.\,\vec\ek_1 \right]
\varkappa_g \right]\vec\nu_g
= \left[ (\varkappa_g)_{s_g} + \left[ \vec x_s\,.\,\vec\ek_2
+ \vec\nu\,.\,\vec\ek_1 \right] \varkappa_g \right]\vec\nu_g
\nonumber \\ & 
= (\varkappa_g)_{s_g}\,\vec\nu_g\,.
\label{eq:hb_covperpkg}
\end{align}
Therefore (\ref{eq:veckappag}) and (\ref{eq:hb_covperpkg}) yield that
\begin{equation} \label{eq:covperp2kg}
(\nabsg^\perp)^2\,\vec\varkappa_g = 
\nabsg^\perp\,[\nabsg^\perp\,(\varkappa_g\,\vec\nu_g)] =
\nabsg^\perp\,[(\varkappa_g)_{s_g}\,\vec\nu_g] =
(\varkappa_g)_{s_gs_g}\,\vec\nu_g\,.
\end{equation}
On combining (\ref{eq:covperp2kg}) and (\ref{eq:DAS17}), we have that
\begin{equation} \label{eq:DAS17b}
\vec x_t = \left [ - (\varkappa_g)_{s_gs_g}
- \tfrac12\,\varkappa_g^3 + \varkappa_g\right] \vec\nu_g\,,
\end{equation}
which agrees with (\ref{eq:hb_elastflow}) in the normal direction on noting
(\ref{eq:Vg}).
\end{rem}

Our weak formulations of (\ref{eq:g_elastflow}) 
are going be to based on the equivalent equation
\begin{equation} \label{eq:g_elastflow2}
g(\vec x)\,\vec x_t\,.\,\vec\nu = 
- \frac1{|\vec x_\rho|} 
\left(  \frac{[\varkappa_g]_\rho}{g^{\frac12}(\vec x)\,|\vec
x_\rho|}\right)_\rho 
- \tfrac12\,g^\frac12(\vec x)\,\varkappa_g^3 
- g^\frac12(\vec x)\,S_0(\vec x)\,\varkappa_g\,,
\end{equation}
where we have recalled (\ref{eq:Vg}) and (\ref{eq:sg}). Note the similarity
between (\ref{eq:g_elastflow2}) and (\ref{eq:sdg2}). 
On recalling (\ref{eq:varkappag}),
we consider the following weak formulation of (\ref{eq:g_elastflow2}),
in the spirit of $(\BGNsd)$ for (\ref{eq:sdg2}). \\ \noindent
$(\BGNwf)$:
Let $\vec x(0) \in [H^1(I)]^2$. For $t \in (0,T]$
find $\vec x(t) \in [H^1(I)]^2$ and $\varkappa(t)\in H^1(I)$ such that
(\ref{eq:weak_varkappa}) holds and
\begin{align}
& \int_I g(\vec x)\,\vec x_t\,.\,\vec\nu\,\chi\,|\vec x_\rho|\drho
= \int_I g^{-\frac12}(\vec x)\,
\left( g^{-\frac12}(\vec x)\left[\varkappa - 
\tfrac12\,\vec\nu\,.\,\nabla\,\ln g(\vec x)
 \right] \right)_\rho \chi_\rho\,|\vec x_\rho|^{-1} \drho 
\nonumber \\ & \qquad
- \tfrac12\, \int_I g^{-1}(\vec x) \left[
\varkappa -\tfrac12\,\vec\nu\,.\,\nabla\,\ln g(\vec x)
  \right]^3 \chi\,|\vec x_\rho| \drho 
\nonumber \\ & \qquad
- \int_I S_0(\vec x) 
\left[ \varkappa -\tfrac12\,\vec\nu\,.\,\nabla\,\ln g(\vec x)
\right] \chi\,|\vec x_\rho| \drho 
\quad \forall\ \chi \in H^1(I)\,. \label{eq:sdwfa} 
\end{align}
We also introduce the following alternative weak formulation for 
(\ref{eq:g_elastflow2}), which treats the curvature $\varkappa_g$ as an 
unknown, in the spirit of $(\BGNsdstab)$ for (\ref{eq:sdg2}). \\ \noindent
$(\BGNwfwf)$:
Let $\vec x(0) \in [H^1(I)]^2$. For $t \in (0,T]$
find $\vec x(t) \in [H^1(I)]^2$ and $\varkappa_g(t) \in H^1(I)$ such that
(\ref{eq:weak_gkgnu}) holds and
\begin{align}
& \int_I g(\vec x)\,\vec x_t\,.\,\vec\nu\,
\chi\,|\vec x_\rho|\drho
= \int_I g^{-\frac12}(\vec x)\,[\varkappa_g]_\rho\,\chi_\rho\,
|\vec x_\rho|^{-1} \drho 
- \tfrac12\, \int_I g^\frac12(\vec x)\,
\varkappa_g^3 \,\chi\,|\vec x_\rho| \drho 
\nonumber \\ & \hspace{5cm}
- \int_I S_0(\vec x)\,g^{\frac12}(\vec x)\,\varkappa_g
\,\chi\,|\vec x_\rho| \drho \quad \forall\ \chi \in H^1(I)\,.
\label{eq:sdwfwfa} 
\end{align}

For the numerical approximations based on $(\BGNwf)$ and $(\BGNwfwf)$ it does
not appear possible to prove stability results that show that 
discrete analogues of (\ref{eq:Wg}) decrease monotonically in time. Based on
the techniques in \cite{pwf}, it is possible to 
introduce alternative weak formulations, for which
semidiscrete continuous-in-time approximations admit such a stability result. 
We will present and analyse these alternative discretizations in the
forthcoming article \cite{hypbolpwf}. 

\subsection{Geodesic curve evolutions on surfaces via conformal maps}
\label{sec:conformal}
Let $\vec\Phi : H \to \bR^d$, $d\geq3$, 
be a conformal parameterization of the embedded 
two-dimensional Riemannian manifold $\mathcal{M} \subset \bR^d$, i.e.\ 
$\mathcal{M} = \vec\Phi(H)$ 
and $|\partial_{\vec\ek_1} \vec\Phi(\vec z)|^2 = |\partial_{\vec\ek_2} 
\vec\Phi(\vec z)|^2$ and
$\partial_{\vec\ek_1} \vec\Phi(\vec z) \,.\, 
\partial_{\vec\ek_2} \vec\Phi(\vec z) = 0$ for all $\vec z \in H$.
While such a parameterization in general does not exist, we recall
from \citet[\S5.10]{Taylor11I} that any orientable two-dimensional
Riemannian manifold can be covered with finitely many conformally 
parameterized patches.
Below we give some examples for $\mathcal{M} \subset \bR^3$, where such a
conformal parameterization exists.
Then the corresponding metric tensor is given by 
$g_{ij} = \partial_{\vec\ek_1} \vec\Phi \,.\, \partial_{\vec\ek_2} \vec\Phi
= g\,\delta_{ij}$ for the metric
\begin{equation} \label{eq:gconformal}
g(\vec z) = |\partial_{\vec\ek_1} \vec\Phi(\vec z)|^2
= |\partial_{\vec\ek_2} \vec\Phi(\vec z)|^2 \qquad \vec z \in H\,.
\end{equation}
We recall from \citet[4.26 in \S4E]{Kuhnel15} 
that for (\ref{eq:gconformal}) it holds that
\begin{equation} \label{eq:S0Gauss}
S_0(\vec z) = - \frac{\Delta\,\ln g(\vec z)}{2\,g(\vec z)}
= \mathcal{K}(\vec\Phi(\vec z))\qquad \vec z \in H\,,
\end{equation}
where $\mathcal{K}$ denotes the Gaussian curvature of $\mathcal{M}$.

An example for (\ref{eq:gconformal}) 
is the stereographic projection of the unit sphere, without the 
north pole, onto the plane, where
\begin{subequations}
\begin{align} \label{eq:gstereo}
\vec\Phi(\vec z) & = (1 + |\vec z|^2)^{-1}\,
(2\,\vec z\,.\,\vec\ek_1, 2\,\vec z\,.\,\vec\ek_2, 
|\vec z|^2 - 1)^T\,,\quad \nonumber \\ 
g(\vec z) & = 4\,(1 + |\vec z|^2)^{-2} \quad\text{and}\quad
H = \bR^2\,;
\end{align}
which yields a geometric interpretation to (\ref{eq:galpha}) with $\alpha=-1$. 
Further examples are the Mercator projection of the unit sphere, without
the north and the south pole, where
\begin{align} \label{eq:gMercator}
\vec\Phi(\vec z) & = \cosh^{-1}(\vec z\,.\,\vec\ek_1)\,
(\cos (\vec z\,.\,\vec\ek_2), \sin (\vec z\,.\,\vec\ek_2), 
\sinh (\vec z\,.\,\vec\ek_1))^T\,,\quad \nonumber \\ 
g(\vec z) & = \cosh^{-2}(\vec z\,.\,\vec\ek_1) \quad\text{and}\quad
H = \bR^2\,;
\end{align}
as well as the catenoid parameterization
\begin{align} \label{eq:gcatenoid}
\vec\Phi(\vec z) & = 
(\cosh (\vec z\,.\,\vec\ek_1)\,\cos (\vec z\,.\,\vec\ek_2), 
\cosh (\vec z\,.\,\vec\ek_1)\,\sin (\vec z\,.\,\vec\ek_2), 
\vec z\,.\,\vec\ek_1)^T\,,\quad \nonumber \\ 
g(\vec z) & = \cosh^2(\vec z\,.\,\vec\ek_1) \quad\text{and}\quad
H = \bR^2\,.
\end{align}
Based on \citet[p.\,593]{Sullivan11} we introduce the following conformal 
parameterization of a torus with large radius $R > 1$ and small radius $r=1$.
In particular, we let $\mathfrak s = [R^2 - 1]^\frac12$ and define
\begin{align} \label{eq:gtorus}
\vec\Phi(\vec z) & = 
\mathfrak s\,
([\mathfrak s^2 + 1]^\frac12- \cos (\vec z\,.\,\vec\ek_2))^{-1}\,
(\mathfrak s\, \cos \tfrac{\vec z\,.\,\vec\ek_1}{\mathfrak s},
\mathfrak s\, \sin \tfrac{\vec z\,.\,\vec\ek_1}{\mathfrak s},
\sin (\vec z\,.\,\vec\ek_2))^T\,,\quad \nonumber \\ 
g(\vec z) & = \mathfrak s^2\,
([\mathfrak s^2 + 1]^\frac12 - \cos (\vec z\,.\,\vec\ek_2))^{-2} 
\quad\text{and}\quad H = \bR^2\,. 
\end{align}
\end{subequations}
We observe that the parameterizations given in 
\mbox{(\ref{eq:gMercator}--d)} are not
bijective, since $\vec\Phi(H)$ covers the surface $\mathcal{M}$  
infinitely many times.

It can be shown that geodesic curvature flow, geodesic curve diffusion and 
geodesic elastic flow on $\mathcal{M} = \vec\Phi(H)$ reduce to 
(\ref{eq:Vgkg}), (\ref{eq:sdg}) and (\ref{eq:g_elastflow})
for the metric $g$ in $H$, respectively. See Appendix~\ref{sec:C}
for details.
Hence the numerical schemes introduced in this paper yield novel
numerical approximations for these geodesic evolution equations. As all the
computations take place in $H$, the discrete curve that approximates
$\vec\Phi(\vec x(I))$ will always lie on $\mathcal{M}$. This is similar to the
approach in \cite{MikulaS06}, where a (local) graph formulation for
$\mathcal{M}$ is employed. But it is fundamentally different from the direct 
approach considered in \cite{curves3d}, where $\vec\Phi(\vec x(I)) \subset
\bR^3$ is discretized. An advantage of the approach in this paper is
that one always stays on $\mathcal{M}$, whereas in the approach of
\cite{curves3d} the curve can leave $\mathcal{M}$ by a small error.
A disadvantage of the strategy in this paper, compared
to \cite{curves3d}, is that if $\overline{\mathcal{M}} \setminus \vec\Phi(H)$ 
is nonempty,
then curves going through these singular points cannot be considered, and
curves coming close to these singular points pose numerical challenges. For
example, the north pole of the unit sphere, i.e.\ $\vec\ek_3 \in \bR^3$, 
is such a singular point for (\ref{eq:gstereo}), while both the north and the
south pole, i.e.\ $\pm\vec\ek_3 \in \bR^3$, are such singular points for
(\ref{eq:gMercator}).
We also note that in the examples (\ref{eq:gcatenoid},d), any closed curve 
$\vec x(I)$ in $H$ will correspond to a curve $\vec\Phi(\vec x(I))$ on the
surface $\mathcal{M}$ that is homotopic to a point. 
In order to model other curves, the domain $H$ needs to be embedded in an
algebraic structure different to $\bR^2$. In particular, 
$H = \bR\times \RpiZ$ for (\ref{eq:gcatenoid}) and
$H = \RpisZ \times \RpiZ$ for (\ref{eq:gtorus}), respectively. 

\subsection{Geometric evolution equations for axisymmetric hypersurfaces}

We recall that the metric (\ref{eq:ge1}) is of relevance when considering
geometric evolution equations for axisymmetric hypersurfaces in $\bR^3$.
However, the natural gradient flows considered in that setting differ from the
flows considered in this paper. Let us briefly recall some geometric evolution
equations for closed hypersurfaces $\mathcal{S}(t)$ in $\bR^d$, $d\geq3$.
We refer to the review article \cite{DeckelnickDE05} for more details.
The mean curvature flow for $\mathcal{S}(t)$,
i.e.\ the $L^2\!\mid_{\mathcal{S}}$--gradient flow of surface area,
is given by the evolution law
\begin{equation} \label{eq:mcfS}
\mathcal{V}_{\mathcal{S}} = k_m\qquad\text{ on } \mathcal{S}(t)\,,
\end{equation}
where $\mathcal{V}_{\mathcal{S}}$ denotes the normal velocity of 
$\mathcal{S}(t)$ in the direction of the normal $\vec\normal_{\mathcal{S}}$.
Moreover, $k_m$ is the mean curvature of $\mathcal{S}(t)$, i.e.\ the sum of the
principal curvatures of $\mathcal{S}(t)$. 
The surface diffusion flow for $\mathcal{S}(t)$ is given by the evolution law
\begin{equation} \label{eq:sdS}
\mathcal{V}_{\mathcal{S}} =- \Delta_{\mathcal{S}}\,k_m \qquad\text{on }
\mathcal{S}(t)\,,
\end{equation}
where $\Delta_{\mathcal{S}}$ is the Laplace--Beltrami operator on
$\mathcal{S}(t)$.

For an axisymmetric hypersurface that is generated from the curve 
$\Gamma(t) = \vec x(t)$ by rotation around the $x_1$--axis, 
the total surface area is given by (\ref{eq:A}). Moreover, 
the mean curvature flow (\ref{eq:mcfS}) can be written in terms of the metric
(\ref{eq:ge1}) as
\begin{equation} \label{eq:aximcf}
\mathcal{V} = \varkappa - \frac{\vec\nu\,.\,\vec\ek_2}{\vec x\,.\,\vec\ek_2}
= g^\frac12(\vec x)\,\varkappa_g
\qquad \iff \qquad
\mathcal{V}_g = g(\vec x)\,\varkappa_g \,,
\end{equation}
see \cite{aximcf}, where we have noted (\ref{eq:varkappag}),
(\ref{eq:nug}) and (\ref{eq:Vg}). Hence (\ref{eq:aximcf}) differs from the
curvature flow (\ref{eq:Vgkg}) for (\ref{eq:ge1}) by a space-dependent
weighting factor.
We note that, in contrast to (\ref{eq:Vgkg}), the flow
(\ref{eq:aximcf}) is invariant under constant rescalings of $g$, e.g.\ both
(\ref{eq:ge1}) and (\ref{eq:gmu}) with $\mu=-1$ lead to the same flow
(\ref{eq:aximcf}).

Moreover, surface diffusion, (\ref{eq:sdS}), 
for axisymmetric hypersurfaces can be written,
in terms of the metric (\ref{eq:ge1}), as
\begin{align} \label{eq:axisd}
& 2\,\pi\,(\vec x\,.\,\vec\ek_2)\,\mathcal{V} = - 2\,\pi \left[
\vec x\,.\,\vec\ek_2 \left[
\varkappa - \frac{\vec\nu\,.\,\vec\ek_2}{\vec x\,.\,\vec\ek_2}\right]_s
\right]_s
= - 2\,\pi \left[
\vec x\,.\,\vec\ek_2 \left[ g^\frac12(\vec x)\,\varkappa_g\right]_s \right]_s
\nonumber \\ & \iff \qquad
\mathcal{V}_g =  - \left[ g^\frac12(\vec x)
\left[ g^\frac12(\vec x)\,\varkappa_g\right]_s \right]_s ,
\end{align}
see \cite{axisd}, where we have noted (\ref{eq:varkappag}), (\ref{eq:nug}) and
(\ref{eq:Vg}). Hence (\ref{eq:axisd}) is dramatically different from 
the curve diffusion flow (\ref{eq:sdg}) for (\ref{eq:ge1}).
Once again we note that, in contrast to (\ref{eq:sdg}), the flow
(\ref{eq:axisd}) is invariant under constant rescalings of $g$, e.g.\ both
(\ref{eq:ge1}) and (\ref{eq:gmu}) with $\mu=-1$ lead to the same flow
(\ref{eq:axisd}). Solutions of (\ref{eq:axisd}) conserve the quantity
$2\,\pi\,\int_{\Omega(t)} \vec z\,.\,\vec\ek_2 \dz =
\int_{\Omega(t)} g^\frac12(\vec z) \dz$
in time, which again differs from (\ref{eq:dAgdt0}), recall (\ref{eq:Ag}). 

\begin{rem} \label{rem:axi}
The metric (\ref{eq:ge1}) can be generalized to model the evolution
of axisymmetric hypersurfaces $\mathcal{S}(t)$ in $\bR^d$, $d\geq 3$. 
In particular, we let
\begin{equation} \label{eq:gaxi}
g(\vec z) = 
[\varsigma(d-1)]^2\,(\vec z\,.\,\vec\ek_2)^{2\,(d-2)}
\quad\text{ and } \quad
H = \bH^2\,,
\end{equation}
where $\varsigma(n) = {2\,\pi^{\frac{n}2}}[\Gamma(\frac{n}2)]^{-1}$ 
denotes the surface area of the $n$-dimensional unit ball.
Then mean curvature flow, (\ref{eq:mcfS}), is given by 
\begin{equation} \label{eq:aximcf2}
\mathcal{V} = \varkappa - 
(d-2)\,\frac{\vec\nu\,.\,\vec\ek_2}{\vec x\,.\,\vec\ek_2}
= g^\frac12(\vec x)\,\varkappa_g
\qquad \iff \qquad
\mathcal{V}_g = g(\vec x)\,\varkappa_g \,,
\end{equation}
in terms of the metric (\ref{eq:gaxi}), where we have recalled
(\ref{eq:varkappag}), (\ref{eq:nug}) and (\ref{eq:Vg}). 
We note that (\ref{eq:aximcf2}) collapses to (\ref{eq:aximcf})
in the case $d=3$.
Surface diffusion, (\ref{eq:sdS}), 
is still given by the last equation in (\ref{eq:axisd}), now for the metric
(\ref{eq:gaxi}). These results 
can be rigorously shown by extending the results in
Appendix~B in \cite{axisd} from $\bR^3$ to $\bR^d$, with the help of 
generalised spherical coordinates.

Using the techniques developed in the present paper, 
it is then possible to derive
weak formulations and stable finite element schemes 
for mean curvature flow and surface diffusion 
of axisymmetric hypersurfaces in $\bR^d$, $d\geq3$, similarly to
the special case $d=3$ treated in \cite{aximcf,axisd}. 
\end{rem}

In the recent paper \cite{axisd}, the authors considered numerical
approximations of Willmore flow for axisymmetric surfaces. The Willmore energy
for the surface $\mathcal{S}$ generated by $\Gamma(t)$ through rotation 
around the $x_1$--axis is given by  
\begin{equation} \label{eq:W}
W_{\mathcal{S}}(\vec x) = \pi\,\int_I \vec x\,.\,\vec\ek_2
\left(\varkappa - \frac{\vec\nu\,.\,\vec\ek_2}{\vec x\,.\,\vec\ek_2}\right)^2
|\vec x_\rho| \drho\,, 
\end{equation}
recall \cite{axisd}. In terms of the metric (\ref{eq:ge1}), 
on recalling (\ref{eq:varkappag}), (\ref{eq:nug}) and (\ref{eq:normg}), 
this can be rewritten as
\begin{equation} \label{eq:Wge1}
W_{\mathcal{S}}(\vec x) = \tfrac12\,\int_I  
g(\vec x)\,\varkappa_g^2 \, |\vec x_\rho|_g \drho\,, 
\end{equation}
which clearly differs from $W_g(\vec x) = \tfrac12\,\int_I 
\varkappa_g^2 \, |\vec x_\rho|_g \drho$, as defined
in (\ref{eq:Wg}).
Hence the flow (\ref{eq:g_elastflow}), for (\ref{eq:gmu}) with $\mu=-1$, 
has no relation at all to
the Willmore flow of axisymmetric surfaces. However, for the metric
(\ref{eq:ghypbol}) it holds, on recalling 
(\ref{eq:mu_varkappag}) for $\mu=1$, (\ref{eq:normg}), (\ref{eq:W}), 
(\ref{eq:varkappa}) and as $I$ is periodic, that
\begin{align} \label{eq:Whypbol}
W_g(\vec x) & = \tfrac12\,\int_I 
\varkappa_g^2 \, |\vec x_\rho|_g \drho
= \tfrac12\,\int_I \vec x\,.\,\vec\ek_2
\left(\varkappa + \frac{\vec\nu\,.\,\vec\ek_2}{\vec x\,.\,\vec\ek_2}\right)^2
|\vec x_\rho| \drho \nonumber \\ &
= (2\,\pi)^{-1}\,W_{\mathcal{S}}(\vec x)
+ 2\, \int_I \varkappa\,\vec\nu\,.\,\vec\ek_2\,|\vec x_\rho|  \drho
= (2\,\pi)^{-1}\,W_{\mathcal{S}}(\vec x)
+ 2\, \int_I \vec x_{ss}\,.\,\vec\ek_2 \,|\vec x_\rho| \drho \nonumber \\ &
= (2\,\pi)^{-1}\,W_{\mathcal{S}}(\vec x)\,,
\end{align}
see also \citet[\S2.2.1]{DallAcquaS17preprint}.
Hence there is a close relation between the hyperbolic elastic flow,
(\ref{eq:hb_elastflow}), and 
Willmore flow for axisymmetric surfaces. In particular, on 
recalling (\ref{eq:Whypbol}), (\ref{eq:newT67}), (\ref{eq:mu_S0}) and
(\ref{eq:gmu}) for $\mu=1$, (\ref{eq:Vg}) and
(\ref{eq:normg}), it holds that
\begin{align}
\ddt\,W_S(\vec x(t)) & = 2\,\pi\,\ddt\,W_g(\vec x(t))
= 2\,\pi\,\int_I \left[  
(\varkappa_g)_{s_gs_g} + \tfrac12\,\varkappa_g^3 - \varkappa_g 
\right] \mathcal{V}_g \,|\vec x_\rho|_g \drho \nonumber \\ &
= 2\,\pi\,\int_I \left[  
(\varkappa_g)_{s_gs_g} + \tfrac12\,\varkappa_g^3 - \varkappa_g 
\right] g(\vec x)\,\mathcal{V} \,|\vec x_\rho| \drho \nonumber \\ &
= 2\,\pi\,\int_I \left[  
(\varkappa_g)_{s_gs_g} + \tfrac12\,\varkappa_g^3 - \varkappa_g 
\right] g^{\frac32}(\vec x)\,
\vec x\,.\,\vec\ek_2\,\mathcal{V} \,|\vec x_\rho| \drho \,.
\label{eq:dtWS}
\end{align}
Hence Willmore flow for axisymmetric surfaces, i.e.\ the 
$L^2\!\mid_{\mathcal{S}}$--gradient flow of (\ref{eq:W}), can be written as
\begin{equation} \label{eq:new_WF}
\mathcal{V} = g^{\frac32}(\vec x)\,
\left( - (\varkappa_g)_{s_gs_g} - \tfrac12\,\varkappa_g^3 
+ \varkappa_g \right)
\iff
g^{-2}(\vec x)\,
\mathcal{V}_g = 
- (\varkappa_g)_{s_gs_g} - \tfrac12\,\varkappa_g^3 + \varkappa_g 
\, ,
\end{equation}
i.e.\ the two flows only differ via a space-dependent weighting, recall
(\ref{eq:hb_elastflow}).
In particular, steady states and minimizers of the two flows agree.

\setcounter{equation}{0}
\section{Finite element approximations} \label{sec:fd}

Let $[0,1]=\cup_{j=1}^J I_j$, $J\geq3$, be a
decomposition of $[0,1]$ into intervals given by the nodes $q_j$,
$I_j=[q_{j-1},q_j]$. 
For simplicity, and without loss of generality,
we assume that the subintervals form an equipartitioning of $[0,1]$,
i.e.\ that 
\begin{equation} \label{eq:Jequi}
q_j = j\,h\,,\quad \mbox{with}\quad h = J^{-1}\,,\qquad j=0,\ldots, J\,.
\end{equation}
Clearly, as $I=\RZ$ we identify $0=q_0 = q_J=1$.

The necessary finite element spaces are defined as follows:
\begin{align*}
V^h & = \{\chi \in C(I) : \chi\!\mid_{I_j} 
\mbox{ is linear}\ \forall\ j=1\to J\}
\quad\text{and}\quad \Vh = [V^h]^2\,.\end{align*}
Let $\{\chi_j\}_{j=1}^J$ denote the standard basis of $V^h$,
and let $\pi^h:C(I)\to V^h$ 
be the standard interpolation operator at the nodes $\{q_j\}_{j=1}^J$.

Let $(\cdot,\cdot)$ denote the $L^2$--inner product on $I$, and 
define the mass lumped $L^2$--inner product $(u,v)^h$,
for two piecewise continuous functions, with possible jumps at the 
nodes $\{q_j\}_{j=1}^J$, via
\begin{equation}
( u, v )^h = \tfrac12\sum_{j=1}^J h_j\,
\left[(u\,v)(q_j^-) + (u\,v)(q_{j-1}^+)\right],
\label{eq:ip0}
\end{equation}
where we define
$u(q_j^\pm)=\underset{\delta\searrow 0}{\lim}\ u(q_j\pm\delta)$.
The definition (\ref{eq:ip0}) naturally extends to vector valued functions.

Let $0= t_0 < t_1 < \ldots < t_{M-1} < t_M = T$ be a
partitioning of $[0,T]$ into possibly variable time steps 
$\ttau_m = t_{m+1} - t_{m}$, $m=0\to M-1$. 
We set $\ttau = \max_{m=0\to M-1}\ttau_m$.
For a given $\vec{X}^m\in \Vh$ we set
$\vec\nu^m = - \frac{[\vec X^m_\rho]^\perp}{|\vec X^m_\rho|}$, as the discrete
analogue to (\ref{eq:tau}). Given $\vec{X}^m\in \Vh$, the fully discrete
approximations we propose in this section will always seek a 
parameterization $\vec{X}^{m+1}\in \Vh$ at the new time level, together with a
suitable approximation of curvature. One class of schemes will rely on the
following discrete analogue of (\ref{eq:weak_varkappa}).
Let $\kappa^{m+1} \in V^h$ be such that
\begin{equation}
\left(\kappa^{m+1}\,\vec\nu^m, \vec\eta\,|\vec X^m_\rho|\right)^h
+ \left(\vec X^{m+1}_\rho, \vec\eta_\rho\,|\vec X^m_\rho|^{-1}\right) 
= 0 \qquad \forall\ \vec\eta \in \Vh\,.
\label{eq:fdb}
\end{equation}
We note that any of the schemes featuring the side constraint
(\ref{eq:fdb}), i.e.\ $(\BGNmckappa_m)^h$, $(\BGNsd_m)^h$
and $(\BGNwf_{m})^h$, below, exhibit a discrete tangential velocity
that leads to a good distribution of vertices.
In particular, a steady state $\Gamma^m = \vec X^m(I)$ 
will satisfy a weak equidistribution property,
i.e.\ any two neighbouring elements are either parallel or of the same length.
Moreover, for general evolutions the distribution of vertices
tends to equidistribution, with the convergence being faster for smaller time
step sizes. The reason is that any curve $\Gamma^m = \vec X^m(I)$, for which
there exists a $\kappa\in V^h$ such that
\begin{equation} \label{eq:quid}
\left(\kappa\,\vec\nu^m, \vec\eta\,|\vec X^m_\rho|\right)^h
+ \left(\vec X^{m}_\rho, \vec\eta_\rho\,|\vec X^m_\rho|^{-1}\right) 
= 0 \qquad \forall\ \vec\eta \in \Vh\,,
\end{equation}
can be shown to satisfy the weak equidistribution property. In particular, the
obvious semidiscrete variants of $(\BGNmckappa_m)^h$, 
$(\BGNsd_m)^h$ and $(\BGNwf_{m})^h$ 
satisfy the weak equidistribution property at every time $t>0$.
We refer to \citet[Rem.\ 2.4]{triplej} and to \cite{fdfi} for more details.

Two other classes of schemes, which will also exhibit nontrivial discrete
tangential motions, will be based on discrete analogues of 
(\ref{eq:weak_gkgnu}). The first variant is given as follows.
Let $\kappa_g^{m+1} \in V^h$ be such that
\begin{align}
& \left(g(\vec X^m)\,\kappa_g^{m+1}\,\vec\nu^m,
\vec\eta\,|\vec X^m_\rho|\right)^{h}
+ \left( \nabla\,g^\frac12(\vec X^{m}) ,
\vec\eta \, |\vec X^{m}_\rho| \right)^{h}
+ \left( g^\frac12(\vec X^{m})\,
\vec X^{m+1}_\rho,\vec\eta_\rho\, |\vec X^m_\rho|^{-1} \right)^{h}
\nonumber \\ & \hspace{9cm}
= 0 
\qquad \forall\ \vec\eta \in \Vh\,.
\label{eq:fdnewb}
\end{align}
Schemes based on (\ref{eq:fdnewb}) will still be linear, but their induced 
tangential motion does not lead to equidistribution. In order to allow for
stability proofs, we now adapt the time discretization in (\ref{eq:fdnewb}). 
In particular, we make use of a convex/concave
splitting of the energy density $g^\frac12$ in (\ref{eq:Lg}). This idea,
for the case of a scalar potential $\Psi : \bR \to \bR$,
goes back to \cite{ElliottS93}, and we adapt their approach to the situation
here, i.e.\ $g^\frac12 : \bR^2 \supset H \to \bRplus$.
In particular, we assume that we can split $g^\frac12$ into
\begin{equation} \label{eq:gsplit}
g^\frac12 = g^\frac12_+ + g^\frac12_-
\quad\text{ such that $\,\pm g^\frac12_\pm\,$ is convex on $H$.}
\end{equation}
Note that such a splitting exists if $D^2\,g^\frac12$ is bounded from below on
$H$, in the sense that there exists a symmetric positive semidefinite 
matrix $A \in \bR^{2\times2}$
such that $D^2\, g^\frac12(\vec z) + A$ is symmetric positive semidefinite 
for all $\vec z \in H$. For example, the splitting can then be chosen such that
$g^\frac12_+(\vec z) = g^\frac12(\vec z) + \tfrac12\,\vec z\,.\,A\,\vec z$
and $g^\frac12_-(\vec z) = - \tfrac12\,\vec z\,.\,A\,\vec z$.
It follows from the splitting in (\ref{eq:gsplit}) that
\begin{equation} \label{eq:gsplitstab}
\nabla\,[g^\frac12_+(\vec u) + g^\frac12_-(\vec v)]\,.\,(\vec u - \vec v) \geq
g^\frac12(\vec u) - g^\frac12(\vec v) \qquad \forall\ \vec u, \vec v \in H\,.
\end{equation}
The alternative discrete analogue of (\ref{eq:weak_gkgnu}),
compared to (\ref{eq:fdnewb}), is then given as
follows. Let $\kappa_g^{m+1} \in V^h$ be such that
\begin{align}
& \left(g(\vec X^m)\,\kappa_g^{m+1}\,\vec\nu^m,
\vec\eta\,|\vec X^m_\rho|\right)^{h}
+ \left( \nabla\,[g^\frac12_+(\vec X^{m+1}) + g^\frac12_-(\vec X^{m})],
\vec\eta \, |\vec X^{m+1}_\rho| \right)^{h}
\nonumber \\ & \hspace{4cm}
+ \left( g^\frac12(\vec X^{m})\,
\vec X^{m+1}_\rho,\vec\eta_\rho\, |\vec X^m_\rho|^{-1} \right)^{h}
= 0 \quad \forall\ \vec\eta \in \Vh\,.
\label{eq:fdnonlinearb}
\end{align}
We note that, in contrast to (\ref{eq:fdnewb}), the side constraint 
(\ref{eq:fdnonlinearb}) will lead to nonlinear schemes.

We observe that in the cases (\ref{eq:ghypbol}--c) and (\ref{eq:gmu}) with 
$\mu \in \bR\setminus (-1,0)$ a splitting of the form (\ref{eq:gsplit}) 
exists. In particular, for $\mu \in \bR\setminus (-1,0)$ the function
$g^\frac12(\vec z) = (\vec z\,.\,\vec\ek_2)^{-\mu}$ is convex on
$H = \bH^2$, since $D^2\,g^\frac12(\vec z) = 
\mu\,(\mu+1)\,(\vec z\,.\,\vec\ek_2)^{-(\mu+2)}\,
\vec\ek_2\otimes\vec\ek_2$
is positive semidefinite for $\vec z \in \bH^2$. Hence 
we can choose 
\begin{equation} \label{eq:gmusplit}
g^\frac12_+(\vec z) = g^\frac12(\vec z) = (\vec z\,.\,\vec\ek_2)^{-\mu}
\quad\text{and}\quad
g^\frac12_-(\vec z) = 0\,,
\end{equation}
with $\nabla\,g^\frac12_+(\vec z) 
= - \mu\,(\vec z\,.\,\vec\ek_2)^{-(\mu+1)}\,\vec\ek_2$.
Moreover, for the class of metrics (\ref{eq:galpha}) a splitting of the form  
(\ref{eq:gsplit}) also exists. To this end, we note that 
$D^2\,g^\frac12(\vec z) = 
16\,\alpha^2\,(1 - \alpha\,|\vec z|^2)^{-3}\,\vec z \otimes \vec z
+ 4\,\alpha\,(1 - \alpha\,|\vec z|^2)^{-2} \,\mat\Id$. Clearly, if $\alpha>0$
then $g^\frac12$ is convex on $H$. If $\alpha \leq 0$, on the other hand,
then $D^2\,g^\frac12$ is clearly the sum of a positive semidefinite and a 
negative semidefinite matrix, with $A=-4\,\alpha\,\mat\Id$ being such that
$D^2\, g^\frac12 + A$ is symmetric positive semidefinite on $H$. 
Hence we can choose
\begin{equation} \label{eq:galphasplit}
\begin{cases}
g^\frac12_+(\vec z) = g^\frac12(\vec z) 
\quad\text{and}\quad
g^\frac12_-(\vec z) = 0 & \alpha > 0 \,,\\
g^\frac12_+(\vec z) = g^\frac12(\vec z) - 2\,\alpha\,|\vec z|^2
\quad\text{and}\quad
g^\frac12_-(\vec z) = 2\,\alpha\,|\vec z|^2
 & \alpha \leq 0\,.
\end{cases}
\end{equation}
Similarly, for the metric (\ref{eq:gMercator}) we note that
\[D^2\,g^\frac12(\vec z) = (\tanh^2(\vec z\,.\,\vec\ek_1) -
\cosh^{-2}(\vec z\,.\,\vec\ek_1))\,\cosh^{-1}(\vec z\,.\,\vec\ek_1)\,
\vec\ek_1\otimes\vec\ek_1\,,\] 
and so we can choose
\begin{equation} \label{eq:gMercatorplit}
g^\frac12_+(\vec z) = g^\frac12(\vec z) + \tfrac12\,(\vec z\,.\,\vec\ek_1)^2
\quad\text{and}\quad
g^\frac12_-(\vec z) = - \tfrac12\,(\vec z\,.\,\vec\ek_1)^2\,.
\end{equation}
For the metric (\ref{eq:gcatenoid}) we observe that
$D^2\,g^\frac12(\vec z) = \cosh(\vec z\,.\,\vec\ek_1)\,
\vec\ek_1\otimes\vec\ek_1$, and so we can choose
\begin{equation} \label{eq:gcatenoidsplit}
g^\frac12_+(\vec z) = g^\frac12(\vec z) 
\quad\text{and}\quad
g^\frac12_-(\vec z) = 0\,.
\end{equation}
Finally, for the metric (\ref{eq:gtorus}) we note that
\[D^2\,g^\frac12(\vec z) = \mathfrak s \left[
\frac{2\,\sin^2(\vec z\,.\,\vec\ek_2)}
{([\mathfrak s^2 + 1]^\frac12 - \cos (\vec z\,.\,\vec\ek_2))^{3}} - 
\frac{\cos (\vec z\,.\,\vec\ek_2)}
{([\mathfrak s^2 + 1]^\frac12 - \cos (\vec z\,.\,\vec\ek_2))^{2}} 
\right] \vec\ek_2\otimes\vec\ek_2\,,\] 
and so we can choose
\begin{equation} \label{eq:gtorussplit}
g^\frac12_+(\vec z) = g^\frac12(\vec z) 
+ \tfrac12\,\mathfrak s\,([\mathfrak s^2 + 1]^\frac12 - 1)^{-2}
\,(\vec z\,.\,\vec\ek_2)^2
\quad\text{and}\quad
g^\frac12_-(\vec z) = g^\frac12(\vec z) - g^\frac12_+(\vec z) \,.
\end{equation}

For the metrics we consider in this paper, we summarize in 
Table~\ref{tab:g} the quantities that are necessary in order to implement the
numerical schemes presented below.
\begin{table}
\center
\def\arraystretch{1.75}%
\begin{tabular}{|c|c|c|c|c|}
\hline
$g$ & 
$\tfrac12\,\vec\nu\,.\,\nabla\,\ln g(\vec x)$ & 
$\nabla\,g^\frac12(\vec x)$ &
$\nabla\,g^\frac12_-(\vec x)$ &
$S_0(\vec x)$ \\
\hline 
(\ref{eq:gmu}) & 
$- \mu\,\frac{\vec\nu\,.\,\vec\ek_2}{\vec x\,.\,\vec\ek_2}$ &
$- \frac\mu{(\vec x\,.\,\vec\ek_2)^{\mu+1}}\,\vec\ek_2$ &
$0$ & 
$-\mu\,(\vec x\,.\,\vec\ek_2)^{2\,(\mu-1)}$ \\
(\ref{eq:galpha}) & 
$\frac{2\,\alpha\,\vec x\,.\,\vec\nu}{1 - \alpha\,|\vec x|^2}$ & 
$\frac{4\,\alpha}{( 1 - \alpha\, |\vec x|^2)^{2}}\,\vec x$ & 
$4\, [\alpha]_-\,\vec x$ &
$-\alpha$ \\ 
(\ref{eq:gMercator}) & $-\tanh(\vec x\,.\,\vec\ek_1)\,\vec \nu\,.\,\vec \ek_1$
& 
$-\frac{\tanh(\vec x\,.\,\vec\ek_1)}{\cosh(\vec x\,.\,\vec\ek_1)}\,\vec\ek_1$ & 
$- (\vec x\,.\,\vec\ek_1)\,\vec\ek_1$ & $1$ 
\\
(\ref{eq:gcatenoid}) & $\tanh(\vec x\,.\,\vec\ek_1)\,\vec \nu \,.\,\vec \ek_1$
& 
$\sinh(\vec x\,.\,\vec\ek_1)\,\vec\ek_1$ &
$0$ & $- \cosh^{-4}(\vec x\,.\,\vec\ek_1)$
\\
(\ref{eq:gtorus}) & $-\frac{\sin(\vec x\,.\,\vec\ek_2)\,\vec \nu\,.\,\vec\ek_2}
{[\mathfrak s^2 + 1]^\frac12 - \cos(\vec x\,.\,\vec\ek_2)}$
& 
$-\frac{\mathfrak s\,\sin(\vec x\,.\,\vec\ek_2)}
{([\mathfrak s^2 + 1]^\frac12 - \cos(\vec x\,.\,\vec\ek_2))^2}\,\vec\ek_2$ & 
$- \frac{\mathfrak s\,\vec x\,.\,\vec\ek_2}
{([\mathfrak s^2 + 1]^\frac12 - 1)^{2}}\,\vec\ek_2$ & 
$\frac{[\mathfrak s^2 + 1]^\frac12 \cos(\vec x\,.\,\vec\ek_2)  - 1}
{\mathfrak s^2}$
\\
\hline
\end{tabular}
\caption{Expressions for terms that are relevant for the implementation of the
presented finite element approximations. 
Here $[\alpha]_- := \min\{0,\alpha\}$.}
\label{tab:g}
\end{table}%

\subsection{Curvature flow}

We consider the following fully discrete analogue of $(\BGNmckappa)$, i.e.\
(\ref{eq:xtweak}), (\ref{eq:weak_varkappa}). \\
\noindent
$(\BGNmckappa_m)^h$:
Let $\vec X^0 \in \Vh$. For $m=0,\ldots,M-1$, 
find $(\vec X^{m+1}, \kappa^{m+1}) \in \Vh \times V^h$ such that
(\ref{eq:fdb}) holds and 
\begin{align}
& \left(g(\vec X^m)\,
\frac{\vec X^{m+1} - \vec X^m}{\ttau_m}, \chi\,\vec\nu^m\,|\vec
X^m_\rho|\right)^h
= \left(\kappa^{m+1} - 
 \tfrac12\,\vec\nu^m\,.\,\nabla\,\ln g(\vec X^m) , 
\chi\,|\vec X^m_\rho|\right)^h 
\nonumber \\ & \hspace{11cm}
\quad \forall\ \chi \in V^h\,. \label{eq:fda}
\end{align}
We remark that the scheme $(\BGNmckappa_m)^h$,
in the case (\ref{eq:geuclid}), collapses to
the scheme \citet[(2.3a,b)]{triplejMC}, with $f = \id$, 
for Euclidean curve shortening flow.

We make the following mild assumption.
\begin{tabbing}
$(\mathfrak A)^{h}$\quad \=
Let $|\vec{X}^m_\rho| > 0$ for almost all $\rho\in I$, and let
$\dim \spa \mathcal Z^{h} = 2$, where \\ \> $\mathcal Z^{h} = 
\left\{ \left( g(\vec X^m)\,\vec\nu^m , \chi |\vec X^m_\rho| \right)^{h} 
: \chi \in V^h \right \} \subset \bR^2$.
\end{tabbing}

\begin{lem} \label{lem:ex}
Let the assumption $(\mathfrak A)^h$ hold.
Then there exists a unique solution \linebreak
$(\vec X^{m+1}, \kappa^{m+1}) \in \Vh \times V^h$ to 
$(\BGNmckappa_m)^h$.
\end{lem}
\begin{proof}
As (\ref{eq:fda}), (\ref{eq:fdb}) is linear, existence follows from uniqueness. 
To investigate the latter, we consider the system: 
Find $(\vec X,\kappa) \in \Vh \times V^h$ such that
\begin{subequations}
\begin{align}
\left(g(\vec X^m)\,\frac{\vec X}{\ttau_m}, \chi\,\vec\nu^m\,|\vec
X^m_\rho|\right)^h
= \left(\kappa ,
\chi\,|\vec X^m_\rho|\right)^h 
\qquad \forall\ \chi \in V^h\,, \label{eq:proofa}\\
\left(\kappa\,\vec\nu^m, \vec\eta\,|\vec X^m_\rho|\right)^h
+ \left(\vec X_\rho, \vec\eta_\rho\,|\vec X^m_\rho|^{-1}\right) = 0 
\qquad \forall\ \vec\eta \in \Vh\,.
\label{eq:proofb}
\end{align}
\end{subequations}
Choosing $\chi=\pi^h[g^{-1}(\vec X^m)\,\kappa]\in V^h$ in (\ref{eq:proofa}) and 
$\vec\eta= \vec X \in \Vh$ in (\ref{eq:proofb}) yields that
\begin{equation} \label{eq:unique0}
\left(|\vec X_\rho|^2, |\vec X^m_\rho|^{-1}\right)
+\ttau_m \left(g^{-1}(\vec X^m)\,|\kappa|^2 , |\vec X^m_\rho|\right)^h = 0\,.
\end{equation}
It follows from (\ref{eq:unique0}) that $\kappa = 0$ and that
$\vec X \equiv \vec X^c\in\bR^2$; and hence that
\begin{equation}
0 = \left(g(\vec X^m)\,\vec X^c, \chi\,\vec\nu^m\,|\vec X^m_\rho|\right)^h = 
\vec X^c \,. \left(g(\vec X^m)\,\vec\nu^m, \chi\,|\vec X^m_\rho|\right)^h 
\quad \forall\ \chi \in V^h \,. \label{eq:unique1}
\end{equation}
It follows from (\ref{eq:unique1}) and
assumption $(\mathfrak A)^h$ that $\vec X^c=\vec0$.
Hence we have shown that $(\BGNmckappa_m)^h$ has a unique solution
$(\vec X^{m+1},\kappa^{m+1}) \in \Vh\times V^h$.
\end{proof}

We consider the following fully discrete analogue of $(\GDmckappa)$, i.e.\
(\ref{eq:Dziuka},b). \\ 
\noindent
$(\GDmckappa_m)^h$:
Let $\vec X^0 \in \Vh$. For $m=0,\ldots,M-1$, 
find $(\vec X^{m+1}, \vec\kappa^{m+1}) \in \Vh \times \Vh$ such that
\begin{subequations}
\begin{align}
& 
\left(g(\vec X^m)\,
\frac{\vec X^{m+1} - \vec X^m}{\ttau_m}, \vec\chi\,|\vec X^m_\rho|
\right)^h
= \left(\vec\kappa^{m+1} - 
\tfrac12\,[\vec\nu^m\,.\,\nabla\,\ln
g(\vec X^m)]\,\vec\nu^m , \vec\chi\,|\vec X^m_\rho|\right)^h 
\nonumber \\ & \hspace{11cm}
\qquad \forall\ \vec\chi \in \Vh\,, \label{eq:fddziuka}\\
& \left(\vec\kappa^{m+1}, \vec\eta\,|\vec X^m_\rho|\right)^h
+ \left(\vec X^{m+1}_\rho, \vec\eta_\rho\,|\vec X^m_\rho|^{-1}\right)
= 0 \qquad \forall\ \vec\eta \in \Vh\,.
\label{eq:fddziukb}
\end{align}
\end{subequations}
We remark that the scheme $(\GDmckappa_m)^h$,
in the case (\ref{eq:geuclid}), collapses to
the scheme in \citet[\S6]{Dziuk94} for Euclidean curve shortening flow.

\begin{lem} \label{lem:GDex}
There exists a unique solution 
$(\vec X^{m+1}, \vec\kappa^{m+1}) \in \Vh \times \Vh$ to 
$(\GDmckappa_m)^h$.
\end{lem}
\begin{proof}
As (\ref{eq:fddziuka},b) is linear, existence follows from uniqueness. 
To investigate the latter, we consider the system: 
Find $(\vec X,\vec\kappa) \in \Vh \times \Vh$ such that
\begin{subequations}
\begin{align}
\left(g(\vec X^m)\,
\frac{\vec X}{\ttau_m}, \vec\chi\,|\vec X^m_\rho|\right)^h
= \left(\vec\kappa ,
\vec\chi\,|\vec X^m_\rho|\right)^h 
\qquad \forall\ \vec\chi \in \Vh\,, \label{eq:proofGDa}\\
\left(\vec\kappa, \vec\eta\,|\vec X^m_\rho|\right)^h
+ \left(\vec X_\rho, \vec\eta_\rho\,|\vec X^m_\rho|^{-1}\right) = 0 
\qquad \forall\ \vec\eta \in \Vh\,.
\label{eq:proofGDb}
\end{align}
\end{subequations}
Choosing $\vec\chi=\vec\pi^h[g^{-1}(\vec X^m)\,\vec\kappa]\in \Vh$ 
in (\ref{eq:proofGDa}) and 
$\vec\eta= \vec X \in \Vh$ in (\ref{eq:proofGDb}) yields that
\begin{equation} \label{eq:uniqueGD0}
\left(|\vec X_\rho|^2, |\vec X^m_\rho|^{-1}\right)
+\ttau_m \left(g^{-1}(\vec X^m)\,|\vec\kappa|^2 , 
|\vec X^m_\rho|\right)^h = 0\,.
\end{equation}
It follows from (\ref{eq:uniqueGD0}) that $\vec\kappa = \vec 0$ and then
from (\ref{eq:proofGDa}) that $\vec X = \vec 0$.
Hence we have shown that (\ref{eq:fddziuka},b) has a unique solution
$(\vec X^{m+1},\vec\kappa^{m+1}) \in \Vh\times\Vh$.
\end{proof}

We consider the following two fully discrete analogues of $(\BGNmc)$, i.e.\
(\ref{eq:bgnnewa}), (\ref{eq:weak_gkgnu}). \\ \noindent 
$(\BGNmc_m)^{h}$:
Let $\vec X^0 \in \Vh$. For $m=0,\ldots,M-1$, 
find $(\vec X^{m+1}, \kappa_g^{m+1}) \in \Vh \times V^h$ such that
(\ref{eq:fdnewb}) holds and
\begin{align}
& \left(g(\vec X^m)\,
\frac{\vec X^{m+1} - \vec X^m}{\ttau_m}, \chi\,\vec\nu^m\,
|\vec X^m_\rho|\right)^{h}
= \left(g^\frac12(\vec X^m)\,\kappa_g^{m+1},
\chi\,|\vec X^m_\rho|\right)^{h} \qquad \forall\ \chi \in V^h\,.
\label{eq:fdnewa} 
\end{align}
\noindent
$(\BGNmc_{m,\star})^{h}$:
Let $\vec X^0 \in \Vh$. For $m=0,\ldots,M-1$, 
find $(\vec X^{m+1}, \kappa_g^{m+1}) \in \Vh \times V^h$ such that
(\ref{eq:fdnonlinearb}) and (\ref{eq:fdnewa}) hold.

We remark that the schemes $(\BGNmc_m)^h$ and
$(\BGNmc_{m,\star})^{h}$, with (\ref{eq:gmusplit}), 
in the case (\ref{eq:geuclid}), collapse to
the scheme \citet[(2.3a,b)]{triplejMC}, with $f = \id$, 
for Euclidean curve shortening flow.

We consider the following two fully discrete analogues of $(\GDmc)$, i.e.\
(\ref{eq:Dziuknewa},b).  
\\ \noindent
$(\GDmc_m)^{h}$:
Let $\vec X^0 \in \Vh$. For $m=0,\ldots,M-1$, 
find $(\vec X^{m+1},\vec\kappa_g^{m+1})\in \Vh \times \Vh$ such that
\begin{subequations}
\begin{align}
& \left(g(\vec X^m)\,
\frac{\vec X^{m+1} - \vec X^m}{\ttau_m}, \vec\chi\,|\vec X^m_\rho|\right)^{h}
= \left(g(\vec X^m)\,\vec\kappa_g^{m+1},
\vec\chi\,|\vec X^m_\rho|\right)^{h} \qquad \forall\ \vec\chi \in \Vh\,,
\label{eq:fddziuknewa} \\
&
\left(g^\frac32(\vec X^m)\,\vec\kappa_g^{m+1},
\vec\eta\,|\vec X^m_\rho|\right)^{h}
+ \left( \nabla\,g^\frac12(\vec X^{m}) ,
\vec\eta \, |\vec X^{m}_\rho| \right)^{h}
+ \left( g^\frac12(\vec X^{m})\,
\vec X^{m+1}_\rho,\vec\eta_\rho\, |\vec X^m_\rho|^{-1} \right)^{h}
\nonumber \\ & \hspace{9cm}
= 0 
\qquad \forall\ \vec\eta \in \Vh\,.
\label{eq:fddziuknewb}
\end{align}
\end{subequations}
\noindent
$(\GDmc_{m,\star})^{h}$:
Let $\vec X^0 \in \Vh$. For $m=0,\ldots,M-1$, 
find $(\vec X^{m+1},\vec\kappa_g^{m+1})\in \Vh \times \Vh$ such that
(\ref{eq:fddziuknewa}) holds and
\begin{align}
&
\left(g^\frac32(\vec X^m)\,\vec\kappa_g^{m+1},
\vec\eta\,|\vec X^m_\rho|\right)^{h}
+ \left( \nabla\,[g^\frac12_+(\vec X^{m+1}) + g^\frac12_-(\vec X^{m})],
\vec\eta \, |\vec X^{m+1}_\rho| \right)^{h}
\nonumber \\ & \hspace{4cm}
+ \left( g^\frac12(\vec X^{m})\,
\vec X^{m+1}_\rho,\vec\eta_\rho\, |\vec X^m_\rho|^{-1} \right)^{h}
= 0 \quad \forall\ \vec\eta \in \Vh\,.
\label{eq:dziukfdnonlinearb}
\end{align}
We remark that the schemes $(\GDmc_m)^h$ and
$(\GDmc_{m,\star})^{h}$, with (\ref{eq:gmusplit}), 
in the case (\ref{eq:geuclid}), collapse to
the scheme in \citet[\S6]{Dziuk94} for Euclidean curve shortening flow.

Overall we observe that $(\BGNmc_m)^{h}$ and $(\GDmc_m)^{h}$ are
linear schemes, while $(\BGNmc_{m,\star})^{h}$ and
$(\GDmc_{m,\star})^{h}$ are nonlinear. For the linear schemes we can
prove existence and uniqueness, while for the nonlinear schemes we can prove
unconditional stability. 

\begin{lem} \label{lem:exnew}
Let the assumption $(\mathfrak A)^{h}$ hold.
Then there exists a unique solution \linebreak
$(\vec X^{m+1},\kappa_g^{m+1}) \in \Vh \times V^h$ to $(\BGNmc_m)^{h}$.
\end{lem}
\begin{proof}
As (\ref{eq:fdnewa}), (\ref{eq:fdnewb}) is linear, 
existence follows from uniqueness. 
To investigate the latter, we consider the system: 
Find $(\vec X, \kappa_g) \in \Vh\times V^h$ such that
\begin{subequations}
\begin{align}
&
\left(g(\vec X^m)\,\frac{\vec X}{\ttau_m}, 
\chi\,\vec\nu^m\,|\vec X^m_\rho|\right)^{h}
= \left(g^\frac12(\vec X^m)\,\kappa_g,
\chi\,|\vec X^m_\rho|\right)^{h} \qquad \forall\ \chi \in V^h\,,
\label{eq:proofnewa} \\
& \left(g(\vec X^m)\,\kappa_g\,\vec\nu^m,\vec\eta\,|\vec X^m_\rho|\right)^{h}
+ \left(g^\frac12(\vec X^m)\,
\vec X_\rho, \vec\eta_\rho\,|\vec X^m_\rho|^{-1}\right) 
= 0 \qquad \forall\ \vec\eta \in \Vh\,.
\label{eq:proofnewb}
\end{align}
\end{subequations}
Choosing $\chi = \kappa_g \in V^h$ in (\ref{eq:proofnewa}) and 
$\vec\eta= \vec X \in \Vh$ in (\ref{eq:proofnewb}) yields that
\begin{equation} \label{eq:uniquenew0}
\left( g^\frac12(\vec X^m) \,
|\vec X_\rho|^2, |\vec X^m_\rho|^{-1}\right)
+ \ttau_m \left(g^\frac12(\vec X^m)\,
|\kappa_g|^2, |\vec X^m_\rho| \right)^{h} = 0\,.
\end{equation}
It immediately
follows from (\ref{eq:uniquenew0}) that $\kappa_g = 0$,
and that $\vec X \equiv \vec X^c\in\bR^2$.
Hence it follows from (\ref{eq:proofnewa}) that
$\vec X^c\,.\,\vec z = 0$ for all $\vec z \in \mathcal Z^{h}$,
and so assumption $(\mathfrak A)^{h}$ yields that $\vec X^c = \vec 0$.
Hence we have shown that $(\BGNmc_m)^{h}$ has a unique solution
$(\vec X^{m+1},\kappa_g^{m+1}) \in \Vh \times V^h$.
\end{proof}

\begin{lem} \label{lem:exdziuk}
Let $|\vec{X}^m_\rho| > 0$ for almost all $\rho\in I$.
Then there exists a unique solution 
$(\vec X^{m+1},\vec\kappa_g^{m+1}) \in \Vh \times \Vh$ to $(\GDmc_m)^{h}$. 
\end{lem}
\begin{proof}
As (\ref{eq:fddziuknewa},b) is linear, existence follows from uniqueness. 
To investigate the latter, we consider the system: 
Find $(\vec X,\vec\kappa_g) \in \Vh \times \Vh$ such that
\begin{subequations}
\begin{align}
&
\left(g(\vec X^m)\,\frac{\vec X}{\ttau_m},
\vec\chi\,|\vec X^m_\rho|\right)^{h}
= \left(g(\vec X^m)\,\vec\kappa_g,
\vec\chi\,|\vec X^m_\rho|\right)^{h} \qquad \forall\ \vec\chi \in 
\Vh\,, \label{eq:proofdziuka} \\
& \left(g^\frac32(\vec X^m)\,\vec\kappa_g,
\vec\eta\,|\vec X^m_\rho|\right)^{h}
+ \left(g^\frac12(\vec X^m)\,
\vec X_\rho, \vec\eta_\rho\,|\vec X^m_\rho|^{-1}\right) 
= 0
\qquad \forall\ \vec\eta \in \Vh\,.
\label{eq:proofdziukb}
\end{align}
\end{subequations}
Choosing $\vec\chi = \vec\pi^h[g^\frac12(\vec X^m)\,\vec\kappa_g \in \Vh$ in
(\ref{eq:proofdziuka}) and  
$\vec\eta= \vec X \in \Vh$ in (\ref{eq:proofdziukb}) yields that
\begin{equation} \label{eq:uniquedziuk0}
\left( g^\frac12(\vec X^m) \,
|\vec X_\rho|^2, |\vec X^m_\rho|^{-1}\right)
+ \ttau_m \left(g^\frac32(\vec X^m)\,|\vec\kappa_g|^2, 
|\vec X^m_\rho| \right)^{h} = 0
\end{equation}
It immediately follows from (\ref{eq:uniquedziuk0}) that
$\vec\kappa_g = \vec0$, and that $\vec X = \vec X^c \in \bR^2$.
Combined with (\ref{eq:proofdziuka}) these imply that
$\vec X^c = \vec 0$.
Hence we have shown that $(\GDmc_m)^{h}$ has a unique solution
$(\vec X^{m+1},\vec\kappa_g^{m+1}) \in \Vh \times \Vh$.
\end{proof}

On recalling (\ref{eq:Lg}), for $\vec Z \in \Vh$ we let
\begin{equation} \label{eq:Lgh}
L_g^{h}(\vec Z) = \left( g^\frac12(\vec Z),|\vec Z_\rho| \right)^{h} .
\end{equation}
We now prove discrete analogues of (\ref{eq:gL2gradflow}) and 
(\ref{eq:gL2gradflow2}) for the schemes  $(\BGNmc_{m,\star})^{h}$ and
$(\GDmc_{m,\star})^{h}$, respectively.

\begin{thm} \label{thm:stab}
Let $(\vec X^{m+1},\kappa_g^{m+1})$ be a solution to 
$(\BGNmc_{m,\star})^{h}$, 
or let \linebreak 
$(\vec X^{m+1},\vec\kappa_g^{m+1})$ be a solution to 
$(\GDmc_{m,\star})^{h}$. 
Then it holds that
\begin{align}
L_g^{h}(\vec X^{m+1}) + \ttau_m
\begin{cases}
\left(g^\frac12(\vec X^m)\,|\kappa_g^{m+1}|^2,
|\vec X^m_\rho|\right)^{h} \\
\left(g^\frac32(\vec X^m)\,|\vec\kappa_g^{m+1}|^2,
|\vec X^m_\rho|\right)^{h}
\end{cases}
 \leq L_g^{h}(\vec X^m)\,,
\label{eq:stab}
\end{align}
respectively. 
\end{thm}
\begin{proof}
Choosing $\chi = \ttau_m\,\kappa_g^{m+1}$ in 
(\ref{eq:fdnewa}) and
$\vec\eta = \vec X^{m+1} - \vec X^m$ in (\ref{eq:fdnonlinearb}) 
yields that
\begin{align}
& - \ttau_m \left(g^\frac12(\vec X^m)\,|\kappa_g^{m+1}|^2,
|\vec X^m_\rho|\right)^{h}
\nonumber \\ & \hspace{4mm}
= \left( \nabla\,[g^\frac12_+(\vec X^{m+1}) + g^\frac12_-(\vec X^{m})],
(\vec X^{m+1} - \vec X^m) \, |\vec X^{m+1}_\rho| \right)^{h}
\nonumber \\ & \hspace{4mm}\qquad
+ \left( g^\frac12(\vec X^{m})\,
\vec X^{m+1}_\rho,(\vec X^{m+1}_\rho - \vec X^m_\rho)\, 
|\vec X^m_\rho|^{-1} \right)^{h}
\nonumber \\ & \hspace{4mm}
\geq \left( g^\frac12(\vec X^{m+1}) - g^\frac12(\vec X^{m}), 
|\vec X^{m+1}_\rho| \right)^{h}
+ \left( g^\frac12(\vec X^{m}), |\vec X^{m+1}_\rho| - |\vec X^m_\rho|
\right)^{h}
\nonumber \\ & \hspace{4mm}
= \left( g^\frac12(\vec X^{m+1})\,|\vec X^{m+1}_\rho| - 
g^\frac12(\vec X^{m})\,|\vec X^m_\rho|, 1 \right)^{h}
= L_g^{h}(\vec X^{m+1}) - L_g^{h}(\vec X^{m})\,,
\label{eq:stab1}
\end{align}
where we have used (\ref{eq:gsplitstab}) and the inequality
$\vec a\,.\,(\vec a - \vec b) \geq |\vec b|\,(|\vec a| - |\vec b|)$
for $\vec a$, $\vec b \in \bR^2$.
This proves the desired result 
(\ref{eq:stab}) for $(\BGNmc_{m,\star})^{h}$. The proof for
$(\GDmc_{m,\star})^{h}$ is analogous.
\end{proof}

\begin{rem} \label{rem:oneliner}
We observe that in most of the above fully discrete schemes it is possible to
eliminate the discrete curvatures, $\kappa_g^{m+1}$ or
$\vec\kappa_g^{m+1}$, to derive discrete analogues of (\ref{eq:gL2flowbgn}) 
and (\ref{eq:gL2flow}), respectively. 
To this end,
let $\vec\omega^m \in \Vh$ be the mass-lumped 
$L^2$--projection of $\vec\nu^m$ onto $\underline V^h$, i.e.\
\begin{equation} \label{eq:omegam}
\left(\vec\omega^m, \vec\varphi \, |\vec X^m_\rho| \right)^h 
= \left( \vec\nu^m, \vec\varphi \, |\vec X^m_\rho| \right)
= \left( \vec\nu^m, \vec\varphi \, |\vec X^m_\rho| \right)^h
\qquad \forall\ \vec\varphi\in\underline V^h\,.
\end{equation}
Then, on recalling {\rm (\ref{eq:omegam})} and
on choosing $\chi = \pi^h[g^\frac12(\vec X^m)\,\vec\eta\,.\,\vec\omega^m] 
\in \Vh$ in
{\rm (\ref{eq:fdnewa})} for $\vec\eta \in \Vh$, the scheme
$(\BGNmc_{m,\star})^h$ reduces to: 
Find $\vec X^{m+1}\in \Vh$ such that
\begin{align}
& \left( g^\frac32(\vec X^m)\,
\frac{\vec X^{m+1} - \vec X^m}{\ttau_m}\,.\,\vec\omega^m, 
\vec\eta\,.\,\vec\omega^m\,|\vec X^m_\rho|\right)^{h}
+ \left( \nabla\,[g^\frac12_+(\vec X^{m+1}) + g^\frac12_-(\vec X^{m})],
\vec\eta \, |\vec X^{m+1}_\rho| \right)^{h}
\nonumber \\ & \hspace{4cm}
+ \left( g^\frac12(\vec X^{m})\,
\vec X^{m+1}_\rho,\vec\eta_\rho\, |\vec X^m_\rho|^{-1} \right)^{h}
= 0 \quad \forall\ \vec\eta \in \Vh\,.
\label{eq:fdnonlinear}
\end{align}
and similarly for $(\BGNmc_m)^h$, $(\GDmc_m)^{h}$ and
$(\GDmc_{m,\star})^{h}$. 
A related variant to {\rm (\ref{eq:fdnonlinear})} is given by:
Find $\vec X^{m+1}\in \Vh$ such that
\begin{align}
& \left( g^\frac32(\vec X^m)\,
\frac{\vec X^{m+1} - \vec X^m}{\ttau_m}\,.\,\vec\nu^m, 
\vec\eta\,.\,\vec\nu^m\,|\vec X^m_\rho|\right)^{h}
+ \left( \nabla\,[g^\frac12_+(\vec X^{m+1}) + g^\frac12_-(\vec X^{m})],
\vec\eta \, |\vec X^{m+1}_\rho| \right)^{h}
\nonumber \\ & \hspace{4cm}
+ \left( g^\frac12(\vec X^{m})\,
\vec X^{m+1}_\rho,\vec\eta_\rho\, |\vec X^m_\rho|^{-1} \right)^{h}
= 0 \quad \forall\ \vec\eta \in \Vh\,.
\label{eq:fdnonlinearnu}
\end{align}
Similarly to Theorem~\ref{thm:stab}, the scheme {\rm (\ref{eq:fdnonlinearnu})}
can also be shown to be unconditionally stable, i.e.\ a solution to
{\rm (\ref{eq:fdnonlinearnu})} satisfies
\begin{align}
L_g^{h}(\vec X^{m+1}) + \ttau_m
\left( g^\frac32(\vec X^m)\left|\frac{\vec X^{m+1} - \vec X^m}{\ttau_m}
\,.\,\vec\nu^m\right|^2, |\vec X^m_\rho|\right)^{h} 
\leq L_g^{h}(\vec X^m)\,.
\label{eq:stabnu}
\end{align}
\end{rem}

\begin{rem} \label{rem:gmu}
Note that in the case $\mu=-1$, the function
$g^\frac12_+(\vec z) = g^\frac12(\vec z) = \vec z\,.\,\vec\ek_2$ is linear, 
and $\nabla\,g^\frac12(\vec z) = \vec\ek_2$.
As a consequence, the numerical integration in the second and third terms in
{\rm (\ref{eq:fdnewb})}, {\rm (\ref{eq:fdnonlinearb})},
{\rm (\ref{eq:fddziuknewb})} and {\rm (\ref{eq:dziukfdnonlinearb})}
plays no role.
In fact, in this case the schemes 
$(\BGNmc_{m})^{h}$, $(\GDmc_{m})^{h}$ and 
$(\BGNmc_{m,\star})^{h}$, $(\GDmc_{m,\star})^{h}$, 
with {\rm (\ref{eq:gmusplit})},  
collapse to their namesakes in \cite{aximcf}, if we account for
the space-dependent weighting factor that
differentiates {\rm (\ref{eq:aximcf})} from {\rm (\ref{eq:Vgkg})}. 
\end{rem}

\begin{rem} \label{rem:Vhpartial}
Using the techniques from \cite{aximcf}, 
it is straightforward to adapt the presented schemes to deal with open curves,
with fixed endpoints. These schemes then allow to compute approximations to
geodesics in the hyperbolic plane, for example.
In particular, we replace $I=\RZ$ by $I=[0,1]$ and define
$\Vhpartial = \{ \vec\eta \in \Vh : \vec\eta(0) = \vec\eta(1) = \vec 0\}$.
Then in place of $(\BGNmckappa_m)^h$ we seek $(\vec X^{m+1}, \kappa^{m+1})
\in \Vh \times V^h$, with $\vec X^{m+1} - \vec X^m \in \Vhpartial$,
such that {\rm (\ref{eq:fda})} holds, as well as {\rm (\ref{eq:fdb})},
with $\Vh$ replaced by $\Vhpartial$. For later reference, we call this adapted
scheme $(\BGNmckappa_m^\partial)^h$.
\end{rem}

\subsection{Curve diffusion}

We consider the following fully discrete approximation of $(\BGNsd)$, i.e.\
(\ref{eq:sdweaka}), (\ref{eq:weak_varkappa}), where, in order to make 
the approximation more practical, we introduce an
auxiliary variable. \\ \noindent
$(\BGNsd_{m})^{h}$:
Let $\vec X^0 \in \Vh$. For $m=0,\ldots,M-1$, 
find $(\vec X^{m+1}, \kappa^{m+1}, \mathfrak{k}^{m+1}) \in \Vh \times V^h
\times V^h$ such that
(\ref{eq:fdb}) holds and
\begin{subequations}
\begin{align}
& \left(g(\vec X^m)\,
\frac{\vec X^{m+1} - \vec X^m}{\ttau_m}, \chi\,\vec\nu^m\,|\vec
X^m_\rho|\right)^{h}
= \left(g^{-\frac12}(\vec X^m)\left[ \mathfrak{k}^{m+1} - Z^m \right]_\rho, 
\chi_\rho\,|\vec X^m_\rho|^{-1}\right)^{h} \nonumber \\ & \hspace{9cm}
\quad \forall\ \chi \in V^h\,, \label{eq:sdfd3a} \\
& \left( g^{\frac12}(\vec X^m)\,
\mathfrak{k}^{m+1}, \zeta \,|\vec X^m_\rho| \right)^h
= \left( \kappa^{m+1}, \zeta \,|\vec X^m_\rho|\right)^h
\quad \forall\ \zeta \in V^h\,, \label{eq:sdfd3b}
\end{align}
where $Z^m \in V^h$ is such that
\begin{equation} \label{eq:Zm}
\left(g^{\frac12}(\vec X^m)\, Z^m, \xi \,|\vec X^m_\rho| \right)^h
= \tfrac12 \left( \vec\nu^m\,.\,\nabla\,\ln g(\vec X^m) , 
\xi \,|\vec X^m_\rho| \right)^h
\quad \forall\ \xi \in V^h\,.
\end{equation}
\end{subequations}
We note that it does not appear possible to prove the existence of a
unique solution to $(\BGNsd_m)^{h}$. 

We consider the following two fully discrete analogues of $(\BGNsdstab)$, i.e.\
(\ref{eq:sdstaba}), (\ref{eq:weak_gkgnu}). The first scheme will be linear,
while the second scheme will be nonlinear, and will admit a stability proof.
\\ \noindent
$(\BGNsdstab_m)^{h}$:
Let $\vec X^0 \in \Vh$. For $m=0,\ldots,M-1$, 
find $(\vec X^{m+1}, \kappa_g^{m+1}) \in \Vh \times V^h$ such that
(\ref{eq:fdnewb}) holds and 
\begin{align}
& \left(g(\vec X^m)\,
\frac{\vec X^{m+1} - \vec X^m}{\ttau_m}, \chi\,\vec\nu^m\,
|\vec X^m_\rho|\right)^{h}
= \left(g^{-\frac12}(\vec X^m)\,[\kappa_g^{m+1}]_\rho,
\chi_\rho\,|\vec X^m_\rho|^{-1} \right)^{h}
\quad \forall\ \chi \in V^h\,. \label{eq:fdsdsa}
\end{align}
\noindent
$(\BGNsdstab_{m,\star})^{h}$:
Let $\vec X^0 \in \Vh$. For $m=0,\ldots,M-1$, 
find $(\vec X^{m+1}, \kappa_g^{m+1}) \in \Vh \times V^h$ such that
(\ref{eq:fdnonlinearb}) and (\ref{eq:fdsdsa}) hold.

We remark that the schemes $(\BGNsd_m)^h$, $(\BGNsdstab_m)^h$ and
$(\BGNsdstab_{m,\star})^{h}$, with (\ref{eq:gmusplit}),  
in the case (\ref{eq:geuclid}), collapse to
the scheme \citet[(2.2a,b)]{triplej} for Euclidean curve/surface diffusion.

\begin{lem} \label{lem:exsd}
Let the assumption $(\mathfrak A)^{h}$ hold.
Then there exists a unique solution
$(\vec X^{m+1},$ $\kappa_g^{m+1}) \in \Vh \times V^h$ 
to $(\BGNsdstab_m)^{h}$.
\end{lem}
\begin{proof}
As (\ref{eq:fdsdsa}), (\ref{eq:fdnewb}) is linear, 
existence follows from uniqueness. 
To investigate the latter, we consider the system: 
Find $(\vec X, \kappa_g) \in \Vh\times V^h$ such that
\begin{subequations}
\begin{align}
&
\left(g(\vec X^m)\,\frac{\vec X}{\ttau_m}, 
\chi\,\vec\nu^m\,|\vec X^m_\rho|\right)^{h}
= \left(g^{-\frac12}(\vec X^m)\,[\kappa_g]_\rho,
\chi_\rho\,|\vec X^m_\rho|^{-1}\right)^{h}
\qquad \forall\ \chi \in V^h\,,
\label{eq:proofsda} \\
& \left(g(\vec X^m)\,\kappa_g\,\vec\nu^m,
\vec\eta\,|\vec X^m_\rho|\right)^{h}
+ \left(g^\frac12(\vec X^m)\,
\vec X_\rho, \vec\eta_\rho\,|\vec X^m_\rho|^{-1}\right)^{h}
= 0 \qquad \forall\ \vec\eta \in \Vh\,.
\label{eq:proofsdb}
\end{align}
\end{subequations}
Choosing $\chi = \kappa_g \in V^h$ in (\ref{eq:proofsda}) and 
$\vec\eta= \vec X \in \Vh$ in (\ref{eq:proofsdb}) yields that
\begin{equation} \label{eq:uniquesd0}
\left( g^\frac12(\vec X^m) \,
|\vec X_\rho|^2, |\vec X^m_\rho|^{-1}\right)^{h}
+ \ttau_m \left(g^{-\frac12}(\vec X^m) \,
|[\kappa_g]_\rho|^2, |\vec X^m_\rho|^{-1} \right)^{h} = 0\,.
\end{equation}
It follows from (\ref{eq:uniquesd0})
that $\kappa_g = \kappa^c \in \bR$ and 
$\vec X \equiv \vec X^c\in\bR^2$.
Hence it follows from (\ref{eq:proofsda}) that
$\vec X^c\,.\,\vec z = 0$ for all $\vec z \in \mathcal{Z}^{h}$,
and so assumption $(\mathfrak A)^{h}$ yields that $\vec X^c = \vec 0$.
Similarly, it follows from (\ref{eq:proofsdb}) and the fact that 
$\mathcal{Z}^{h}$ must contain a nonzero vector that $\kappa^c=0$.
Hence we have shown that $(\BGNsdstab_m)^{h}$ has a unique solution
$(\vec X^{m+1},\kappa_g^{m+1}) \in \Vh \times V^h$.
\end{proof}

We now prove a discrete analogue of (\ref{eq:gH-1gradflow}), recall also
(\ref{eq:Fstab}), for the scheme $(\BGNsdstab_{m,\star})^{h}$.

\begin{thm} \label{thm:stabsd}
Let $(\vec X^{m+1},\kappa_g^{m+1})$ be a solution to 
$(\BGNsdstab_{m,\star})^{h}$. 
Then it holds that
\begin{align}
L_g^{h}(\vec X^{m+1}) + \ttau_m
\left(g^{-\frac12}(\vec X^m)\,|[\kappa_g]_\rho|^2, |\vec X^m_\rho|^{-1}
\right)^{h}
 \leq L_g^{h}(\vec X^m)\,.
\label{eq:stabsd}
\end{align}
\end{thm}
\begin{proof}
Choosing $\chi = \ttau_m\,\kappa_g^{m+1}$ in 
(\ref{eq:fdsdsa}) and
$\vec\eta = \vec X^{m+1} - \vec X^m$ in (\ref{eq:fdnonlinearb}) we obtain,
similarly to (\ref{eq:stab1}), that
\begin{align}
& - \ttau_m \left(g^{-\frac12}(\vec X^m) \,
|[\kappa_g^{m+1}]_\rho|^2,
|\vec X^m_\rho|^{-1} \right)^{h} 
\nonumber \\ & \hspace{1cm}
= \left( \nabla\,[g^\frac12_+(\vec X^{m+1}) + g^\frac12_-(\vec X^{m})],
(\vec X^{m+1} - \vec X^m) \, |\vec X^{m+1}_\rho| \right)^{h}
\nonumber \\ & \hspace{1cm}\qquad
+ \left( g^\frac12(\vec X^{m})\,
\vec X^{m+1}_\rho,(\vec X^{m+1}_\rho - \vec X^m_\rho)\, 
|\vec X^m_\rho|^{-1} \right)^{h}
\leq L_g^{h}(\vec X^{m+1}) - L_g^{h}(\vec X^{m})\,.
\label{eq:stabsd1}
\end{align}
This proves the desired result (\ref{eq:stabsd}).
\end{proof}

\subsection{Elastic flow}

We consider the following fully discrete finite element approximation of
$(\BGNwf)$, i.e.\ (\ref{eq:sdwfa}) and (\ref{eq:weak_varkappa}),
similarly to the approximation $(\BGNsd_{m})^h$ for $(\BGNsd)$.
\\ \noindent
$(\BGNwf_{m})^{h}$:
Let $\vec X^0 \in \Vh$ and $\kappa^0 \in V^h$. For $m=0,\ldots,M-1$, 
find $(\vec X^{m+1}, \kappa^{m+1}, \mathfrak{k}^{m+1}) 
\in \Vh \times V^h \times V^h$ such that
(\ref{eq:fdb}), (\ref{eq:sdfd3b}) hold and
\begin{align}
& \left(g(\vec X^m)\,
\frac{\vec X^{m+1} - \vec X^m}{\ttau_m}, \chi\,\vec\nu^m\,|\vec
X^m_\rho|\right)^{h}
= \left( g^{-\frac12}(\vec X^m) \left[ \mathfrak{k}^{m+1} 
- Z^m \right]_\rho, \chi_\rho\,|\vec X^m_\rho|^{-1}\right)^{h} 
\nonumber \\ & \qquad
- \tfrac12 \left( g^{-1}(\vec X^m)
\left[ \kappa^m - 
\tfrac12\,\vec\nu^m\,.\,\nabla\,\ln g(\vec X^m) \right]^3, 
\chi\,|\vec X^m_\rho| \right)^{h}
\nonumber \\ & \qquad
- \left( S_0(\vec X^m)
\left[ \kappa^m - \tfrac12\,\vec\nu^m\,.\,\nabla\,\ln g(\vec X^m) 
\right], \chi\,|\vec X^m_\rho| \right)^{h}
\quad \forall\ \chi \in V^h\,, \label{eq:bgnfda}
\end{align}
where $Z^m \in V^h$ is defined by (\ref{eq:Zm}). 

We consider the following fully discrete finite element approximation of
$(\BGNwfwf)$, i.e.\ (\ref{eq:sdwfwfa}) and (\ref{eq:weak_gkgnu}). 
\\ \noindent
$(\BGNwfwf_m)^{h}$:
Let $\vec X^0 \in \Vh$ and $\kappa_g^0 \in V^h$. For $m=0,\ldots,M-1$, 
find $(\vec X^{m+1}, \kappa_g^{m+1}) \in \Vh \times V^h$ such that
(\ref{eq:fdnewb}) holds and
\begin{align}
& \left( g(\vec X^m)\,
\frac{\vec X^{m+1} - \vec X^m}{\ttau_m}, \chi\,\vec\nu^m\,
|\vec X^m_\rho|\right)^{h}
= \left( g^{-\frac12}(\vec X^m)\,[\kappa_g^{m+1}]_\rho,
\chi_\rho\,|\vec X^m_\rho|^{-1} \right)^{h}
\nonumber \\ & \
- \tfrac12 \left( g^\frac12(\vec X^m)\,
(\kappa_g^{m})^3, \chi\,|\vec X^m_\rho| \right)^{h}
- \left( S_0(\vec X^m)\,g^\frac12(\vec X^m)\,
\kappa_g^{m}, \chi\,|\vec X^m_\rho| \right)^{h}
\quad \forall\ \chi \in V^h\,. \label{eq:wfwfa}
\end{align}
Clearly, for the metric (\ref{eq:geuclid}) we have that $S_0=0$, and so 
the last terms in (\ref{eq:bgnfda}) and (\ref{eq:wfwfa}) vanish. 
In fact, in this case the schemes $(\BGNwf_{m})^h$ and $(\BGNwfwf_{m})^h$
collapse to the scheme \citet[(2.45a,b)]{triplej}, with $\lambda_m=0$, 
for Euclidean elastic flow.

\begin{rem} \label{rem:lambda}
It is often of interest to add a length penalization term to the energy {\rm
(\ref{eq:Wg})}, and hence consider the $L^2$--gradient flow of
\begin{equation} \label{eq:Wglambda}
W_g^\lambda(\vec x) = W_g(\vec x) + \lambda\,L_g(\vec x)\,,
\end{equation}
recall {\rm (\ref{eq:Lg})}, for some $\lambda \in \bRgeq$,
see e.g.\ \cite{DallAcquaS17preprint}. It is straightforward to
generalize our weak formulations and finite element approximations to this
case. For example, the scheme $(\BGNwfwf_m)^{h}$ is adapted by adding the
term $\lambda\,(g^\frac12(\vec X^m)\,\kappa_g^{m+1},
\chi\,|\vec X^m_\rho|)^{h}$ to the right hand side of {\rm (\ref{eq:wfwfa})},
and we call this new scheme $(\BGNwfwf_m^{\lambda})^{h}$ for later 
reference.
\end{rem}

\setcounter{equation}{0}
\section{Numerical results} \label{sec:nr}

We recall from (\ref{eq:Ag}) that
\begin{equation} \label{eq:Agmu}
A_g(\vec x) = \int_\Omega g(\vec z) \dz
= \int_I \vec\phi_g(\vec x)\,.\,\vec\nu\, |\vec x_\rho| \drho\,,
\quad\text{where}\quad
\nabla\,.\,\vec\phi_g = g \quad\text{in } H \,,
\end{equation}
if $\nu \circ \vec x^{-1}$ denotes the outer normal on 
$\partial\Omega = \Gamma = \vec x(I)$.
With this in mind, we define the following approximation of $A_g(\vec X^m)$,
\begin{equation} \label{eq:Agh}
A_g^h(\vec X^m) = 
\left(\vec\phi_g(\vec X^m)\,.\,\vec\nu^m , |\vec X^m_\rho| \right)^h .
\end{equation}
For the different metrics we consider, the function $\vec\phi_g$ can be chosen
as follows.
\begin{alignat*}{2} 
& \text{(\ref{eq:gmu})} && \quad
\vec\phi_g(\vec z) = 
\begin{cases}
(1 - 2\,\mu)^{-1}\,(\vec
z\,.\,\vec\ek_2)^{1 - 2\,\mu}\,\vec\ek_2 
& \mu \not= \frac12\,,\\ 
\ln (\vec z\,.\,\vec\ek_2)\,\vec\ek_2 & \mu = \frac12\,,
\end{cases} \nonumber \\
& \text{(\ref{eq:galpha})} && \quad
\vec\phi_g(\vec z) = 
\begin{cases}
2\,[\alpha\,|\vec z|^2\,(1 - \alpha\,|\vec z|^2)]^{-1}\,\vec z 
& \alpha \not= 0\,,\\ 
2\,\vec z & \alpha = 0\,,
\end{cases} \nonumber \\
& \text{(\ref{eq:gMercator})} && \quad
\vec\phi_g(\vec z) = \tanh (\vec z\,.\,\vec\ek_1)\,\vec\ek_1\,,\nonumber \\
& \text{(\ref{eq:gcatenoid})} && \quad
\vec\phi_g(\vec z) = \tfrac12\,( \vec z\,.\,\vec\ek_1 
+ \sinh (\vec z\,.\,\vec\ek_1)\, \cosh (\vec z\,.\,\vec\ek_1) ) \,\vec\ek_1\,,
\nonumber \\
& \text{(\ref{eq:gtorus})} && \quad
\vec\phi_g(\vec z) = \tfrac{2\,[\mathfrak s^2+1]^\frac12}{\mathfrak s}\,
\arctan\big(\tfrac{[\mathfrak s^2+1]^\frac12+1}{\mathfrak s}\,
\tan \tfrac{\vec z\,.\,\vec\ek_2}{2} \big)
+\frac{\sin \vec z\,.\,\vec\ek_2}
{[\mathfrak s^2+1]^\frac12 - \cos \vec z\,.\,\vec\ek_2}\,.
\end{alignat*}

For solutions of the scheme $(\BGNwf_m)^h$, we define
\begin{equation} \label{eq:Wgh}
W_g^m = \tfrac12\left( g^{-\frac12}(\vec X^m)\left[\kappa^m 
- \tfrac{1}{2} \, \vec\nu^m\,.\,\nabla\,
\ln g(\vec X^m) \right]^2, |\vec X^m_\rho| \right)^h
\end{equation}
as the natural discrete analogue of (\ref{eq:Wg}), while for solutions of the
scheme $(\BGNwfwf_m)^h$ we define
\begin{equation} \label{eq:tildeWgh}
\widetilde W_g^m =
\tfrac12\left( g^{\frac12}(\vec X^m)\left[\kappa^m_g \right]^2,
|\vec X^m_\rho| \right)^h .
\end{equation}
On recalling (\ref{eq:varkappa}), and given $\Gamma^0 = \vec X^0(\overline I)$,
we define the initial data $\kappa^0 \in V^h$ 
for the scheme $(\BGNwf_m)^h$ via
$\kappa^0 = \pi^h\left[\frac{\vec\kappa^0\,.\,\vec\omega^0}{|\vec\omega^0|}
\right]$, where we recall (\ref{eq:omegam}), and 
where $\vec\kappa^0\in \Vh$ is such that
\begin{equation*} 
\left( \vec\kappa^{0},\vec\eta\, |\vec X^0_\rho| \right)^h
+ \left( \vec{X}^{0}_\rho , \vec\eta_\rho\,|\vec X^0_\rho|^{-1} \right)
 = 0 \quad \forall\ \vec\eta \in \Vh\,.
\end{equation*}
With this definition of $\kappa^0$, we define the initial data
$\kappa_g^0 \in V^h$ for the scheme $(\BGNwfwf_m)^h$ via
$\kappa_g^0 = \pi^h\left[ g^{-\frac12}(\vec X^0)\left[\kappa^0 
- \tfrac{1}{2} \, \vec\omega^0\,.\,\nabla\, \ln g(\vec X^0)\right]\right]$.

We also consider the ratio
\begin{equation} \label{eq:ratio}
\ratio^m = \dfrac{\max_{j=1\to J} |\vec{X}^m(q_j) - \vec{X}^m(q_{j-1})|}
{\min_{j=1\to J} |\vec{X}^m(q_j) - \vec{X}^m(q_{j-1})|}
\end{equation}
between the longest and shortest element of $\Gamma^m$, and are often
interested in the evolution of this ratio over time.

\subsection{The hyperbolic plane, and (\ref{eq:gmu}) for $\mu \in \bR$}

Unless otherwise stated, all our computations in this section 
are for the hyperbolic plane,
i.e.\ (\ref{eq:ghypbol}) or, equivalently, (\ref{eq:gmu}) with $\mu=1$.

\subsubsection{Curvature flow}

{From} the Appendix~\ref{sec:A1} we recall the true solution 
(\ref{eq:app_ansatz}) with (\ref{eq:true_Vgkg}) 
for hyperbolic curvature flow, (\ref{eq:Vgkg}) in the case (\ref{eq:ghypbol}). 
We use this true solution for a convergence test for the various schemes for
curvature flow. Here
we start with a nonuniform partitioning of a circle of radius 
$r(0)=1$ centred at $a(0)\,\vec\ek_2$, where $a(0) = 2$.
In particular, we choose $\vec X^0 \in \Vh$ with
\begin{equation} \label{eq:X0}
\vec X^0(q_j) = a(0)\,\vec\ek_2 + 
r(0) \begin{pmatrix} 
\cos[2\,\pi\,q_j + 0.1\,\sin(2\,\pi\,q_j)] \\
\sin[2\,\pi\,q_j + 0.1\,\sin(2\,\pi\,q_j)]
\end{pmatrix}, \quad j = 1,\ldots,J\,,
\end{equation}
recall (\ref{eq:Jequi}). 
We compute the error
\begin{equation} \label{eq:errorXx}
\errorXx = \max_{m=1,\ldots,M} \max_{j=1,\ldots,J} | 
|\vec X^m(q_j) - a(t_m)\,\vec\ek_2| - r(t_m)|
\end{equation}
over the time interval $[0,0.1]$ between 
the true solution (\ref{eq:app_ansatz}) 
and the discrete solutions for the schemes
$(\BGNmckappa_m)^h$ and $(\GDmckappa_m)^h$. We note that the extinction time 
for (\ref{eq:true_Vgkg}) is $T_0 = -\frac12\,\ln\frac34 = 0.144$. 
Here we use the time step size $\ttau=0.1\,h^2_{\Gamma^0}$,
where $h_{\Gamma^0}$ is the maximal edge length of $\Gamma^0$.
The computed errors are reported in Table~\ref{tab:mcf1}.
The same errors for the schemes $(\BGNmc_m)^h$, $(\GDmc_m)^h$, 
$(\BGNmc_{m,\star})^h$ and $(\GDmc_{m,\star})^h$ can be seen in
Table~\ref{tab:mcf_schemesCD}. We observe that all schemes exhibit
second order convergence rates, with the smallest errors produced by
$(\BGNmckappa_m)^h$ and $(\BGNmc_{m,\star})^h$.
\begin{table}
\center
\begin{tabular}{|rr|c|c|c|c|}
\hline
& & \multicolumn{2}{c|}{$(\BGNmckappa_m)^h$}&
\multicolumn{2}{c|}{$(\GDmckappa_m)^h$} \\
$J$  & $h_{\Gamma^0}$ & $\errorXx$ & EOC & $\errorXx$ & EOC \\ \hline
32   & 2.1544e-01 & 2.7956e-02 & ---      & 4.7884e-02 & --- \\
64   & 1.0792e-01 & 7.6597e-03 & 1.872810 & 1.3493e-02 & 1.832236 \\         
128  & 5.3988e-02 & 1.9572e-03 & 1.969971 & 3.4819e-03 & 1.955728 \\
256  & 2.6997e-02 & 4.9196e-04 & 1.992498 & 8.7754e-04 & 1.988657 \\
512  & 1.3499e-02 & 1.2315e-04 & 1.998231 & 2.1982e-04 & 1.997249 \\
\hline
\end{tabular}
\caption{Errors for the convergence test for 
(\ref{eq:app_ansatz}) with (\ref{eq:true_Vgkg}),
with $r(0) = 1$, $a(0) = 2$, over the time interval $[0,0.1]$.}
\label{tab:mcf1}
\end{table}%
\begin{table}
\begin{tabular}{|r|c|c|c|c|c|c|}
\hline
 & $(\BGNmc_m)^h$ & $(\GDmc_m)^h$ & 
\multicolumn{2}{c|}{$(\BGNmc_{m,\star})^h$} &
\multicolumn{2}{c|}{$(\GDmc_{m,\star})^h$} \\ 
$J$ & $\errorXx$ & $\errorXx$ & $\errorXx$ & EOC & $\errorXx$ & EOC \\ \hline
32  & 3.4212e-02 & 5.2395e-02 & 2.7155e-02 & ---      & 4.9299e-02 & --- \\
64  & 9.3803e-03 & 1.4744e-02 & 7.5112e-03 & 1.855491 & 1.3918e-02 & 1.825973 \\
128 & 2.3970e-03 & 3.8047e-03 & 1.9237e-03 & 1.965475 & 3.5945e-03 & 1.953402 \\
256 & 6.0251e-04 & 9.5887e-04 & 4.8381e-04 & 1.991478 & 9.0613e-04 & 1.988107 \\
512 & 1.5082e-04 & 2.4019e-04 & 1.2112e-04 & 1.998003 & 2.2699e-04 & 1.997089 \\
\hline
\end{tabular}
\caption{Errors for the convergence test for 
(\ref{eq:app_ansatz}) with (\ref{eq:true_Vgkg}),
with $r(0) = 1$, $a(0) = 2$, over the time interval $[0,0.1]$.}
\label{tab:mcf_schemesCD}
\end{table}%

For the scheme $(\BGNmckappa_m)^h$
we show the evolution of a cigar shape in Figure~\ref{fig:mc_cigar}.
The discretization parameters are $J=128$ and $\ttau = 10^{-4}$.
Rotating the initial shape by $90^\circ$ degrees yields the evolution in
Figure~\ref{fig:mc_cigar2}. We note that in both cases the curve shrinks to a
point.
\begin{figure}
\center
\mbox{
\includegraphics[angle=-90,width=0.4\textwidth]{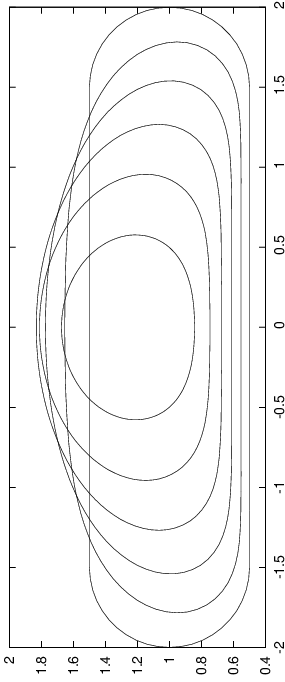}
\includegraphics[angle=-90,width=0.3\textwidth]{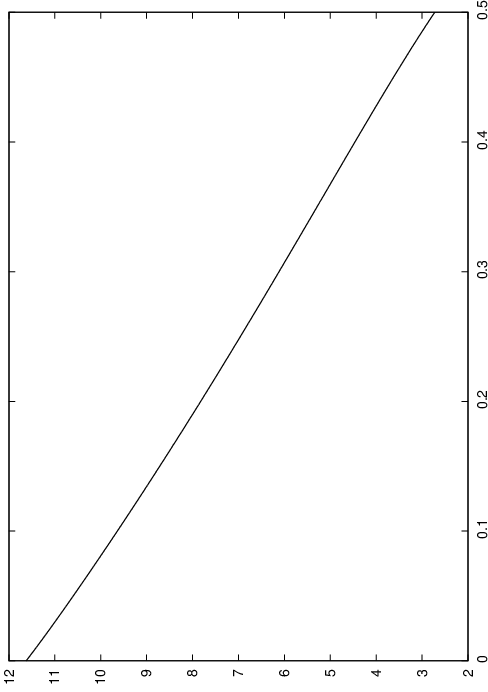}
\includegraphics[angle=-90,width=0.3\textwidth]{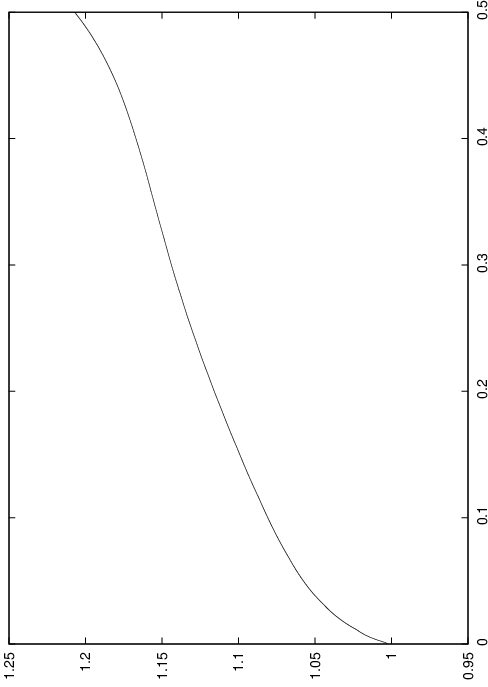}}
\caption{$(\BGNmckappa_m)^h$
Curvature flow towards extinction. 
Solution at times $t=0,0.1,\ldots,0.5$.
On the right are plots of the discrete energy $L_g^h(\vec X^m)$
and of the ratio (\ref{eq:ratio}).} 
\label{fig:mc_cigar}
\end{figure}%
\begin{figure}
\center
\begin{minipage}{0.3\textwidth}
\includegraphics[angle=-90,width=0.6\textwidth]{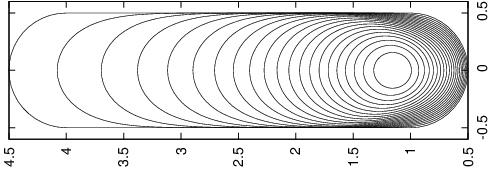}
\end{minipage} \qquad
\begin{minipage}{0.5\textwidth}
\includegraphics[angle=-90,width=0.7\textwidth]{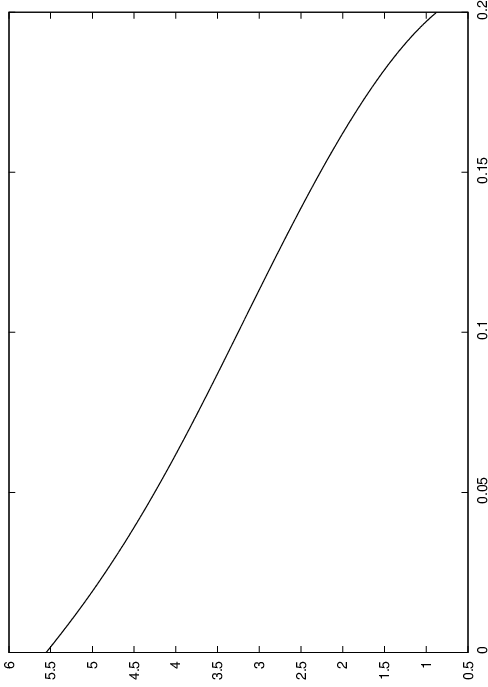}\\
\includegraphics[angle=-90,width=0.7\textwidth]{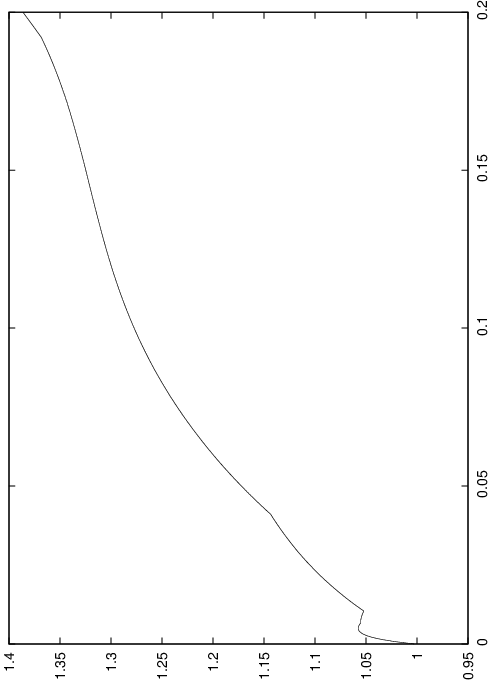}
\end{minipage}
\caption{
$(\BGNmckappa_m)^h$
Curvature flow towards extinction. 
Solution at times $t=0,0.01,\ldots,0.2$.
On the right are plots of the discrete energy $L_g^h(\vec X^m)$
and of the ratio (\ref{eq:ratio}).} 
\label{fig:mc_cigar2}
\end{figure}%
The same computations for the remaining schemes, i.e.\
$(\GDmckappa_m)^h$, $(\BGNmc_{m})^h$, $(\GDmc_{m})^h$, 
$(\BGNmc_{m,\star})^h$ and $(\GDmc_{m,\star})^h$, yield very similar results,
with the main difference being the evolution of the ratio 
(\ref{eq:ratio}). 
For the simulation in Figure~\ref{fig:mc_cigar}, 
we present the plots of this quantity for these alternative schemes
in Figure~\ref{fig:mc_cigar_r}, where we
observe that the obtained curves are far from being equidistributed.
In particular, the ratio for the schemes $(\GDmckappa_m)^h$, $(\GDmc_{m})^h$ 
$(\GDmc_{m,\star})^h$ reaches almost $60$, while it remains bounded below $3$
for the schemes $(\BGNmc_{m})^h$ and $(\BGNmc_{m,\star})^h$. This compares with
a final ratio of about $1.2$ in Figure~\ref{fig:mc_cigar}.
\begin{figure}
\center
\includegraphics[angle=-90,width=0.3\textwidth]{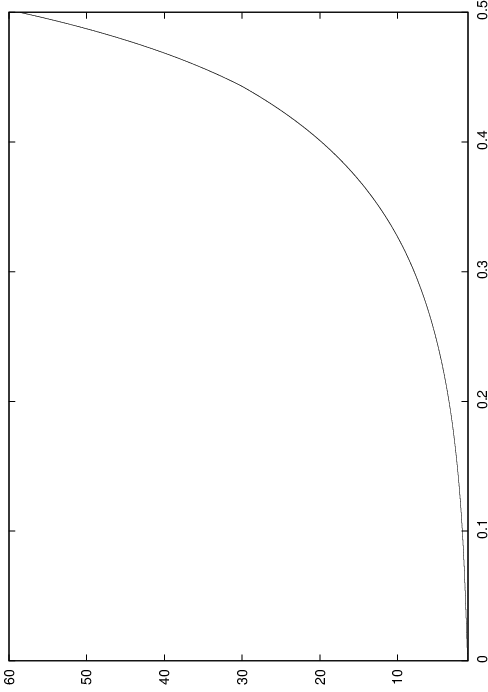}
\includegraphics[angle=-90,width=0.3\textwidth]{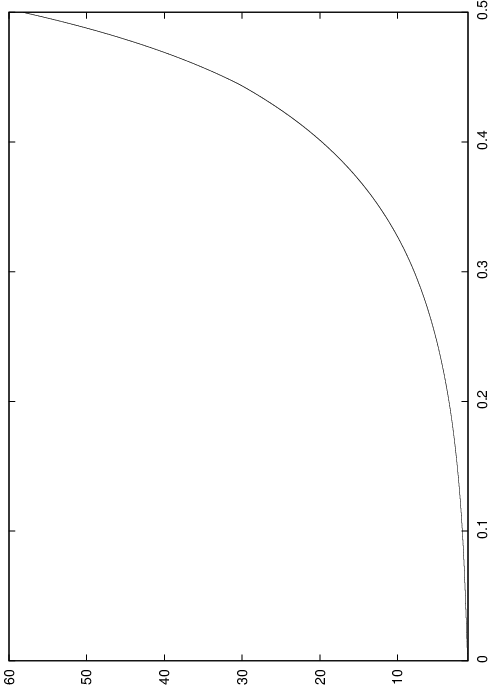}
\includegraphics[angle=-90,width=0.3\textwidth]{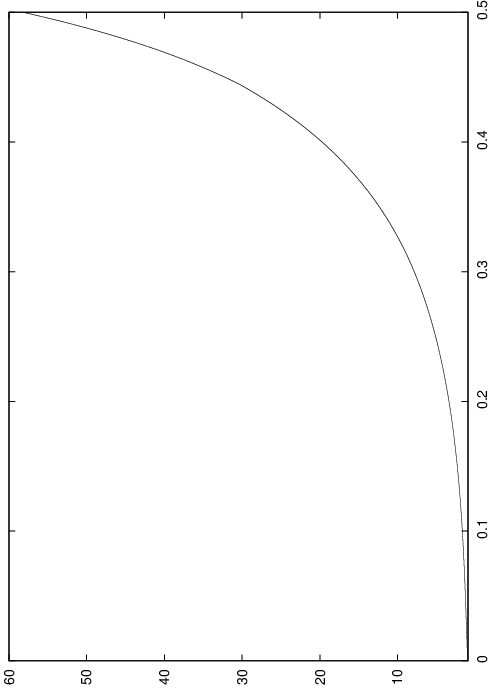}
\includegraphics[angle=-90,width=0.3\textwidth]{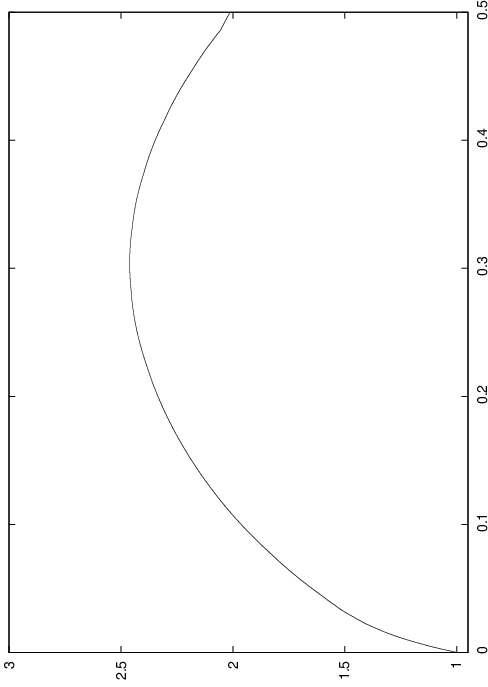}
\includegraphics[angle=-90,width=0.3\textwidth]{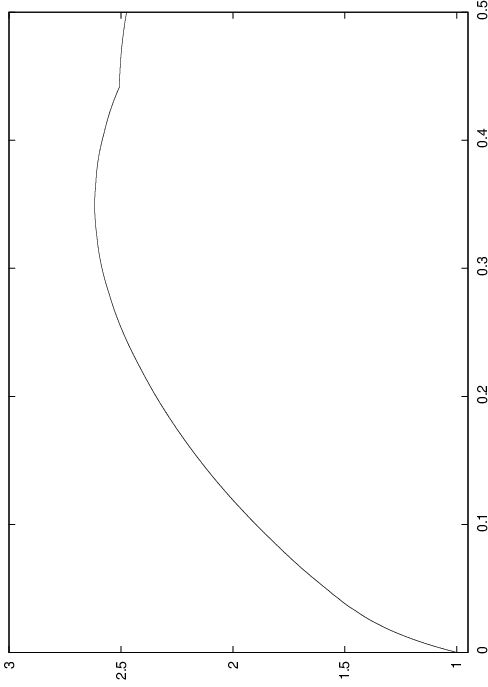}
\caption{
The ratio plots (\ref{eq:ratio}) for the schemes 
$(\GDmckappa_m)^h$, $(\GDmc_{m})^h$, $(\GDmc_{m,\star})^h$, 
$(\BGNmc_{m})^h$ and $(\BGNmc_{m,\star})^h$ for simulations as in 
Figure~\ref{fig:mc_cigar}.} 
\label{fig:mc_cigar_r}
\end{figure}%

We now employ the scheme $(\BGNmckappa_m^\partial)^h$ to compute some 
geodesics. To this end, we use as initial data a straight line
segment between the two fixed endpoints, and let the scheme run until time 
$T=10$, at which point the discrete energy $L_g^h(\vec X^m)$ is almost constant
in time. For each run the discretization parameters are $J=128$ and 
$\ttau = 10^{-4}$. For the hyperbolic plane, 
we show the final curves $\Gamma^M$ in Figure~\ref{fig:mc_geodesics}.
\begin{figure}
\center
\includegraphics[angle=-90,width=0.55\textwidth]{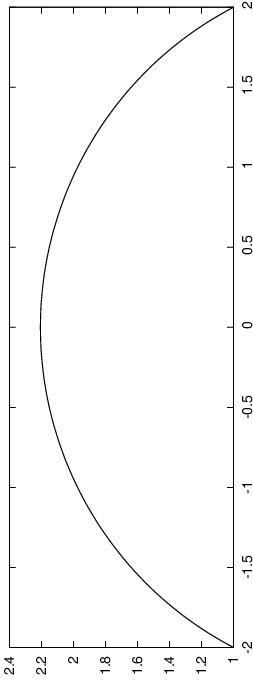}
\includegraphics[angle=-90,width=0.4\textwidth]{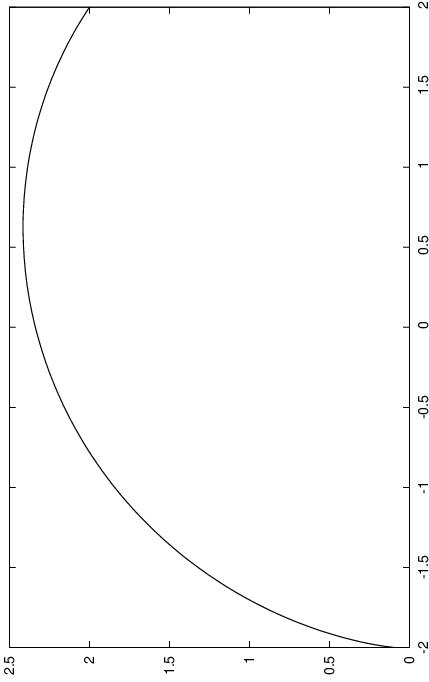}
\caption{
$(\BGNmckappa_m^\partial)^h$
Geodesics in the hyperbolic plane, obtained with curvature flow. 
The left geodesic connects the points
$(\pm2,1)^T$, with $L_g^h(\vec X^M) = 2.887$, 
while the right geodesics connects
$(-2,0.1)^T$ and $(2,2)^T$, with $L_g^h(\vec X^M) = 4.620$.
} 
\label{fig:mc_geodesics}
\end{figure}%
Repeating the first of the two geodesic computations for the metric 
(\ref{eq:gmu}) with $\mu = 0.1$ and $\mu = 2$ yields the results in
Figure~\ref{fig:mc_geodesics2}.
\begin{figure}
\center
\includegraphics[angle=-90,width=0.6\textwidth]{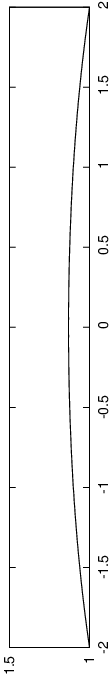}
\includegraphics[angle=-90,width=0.35\textwidth]{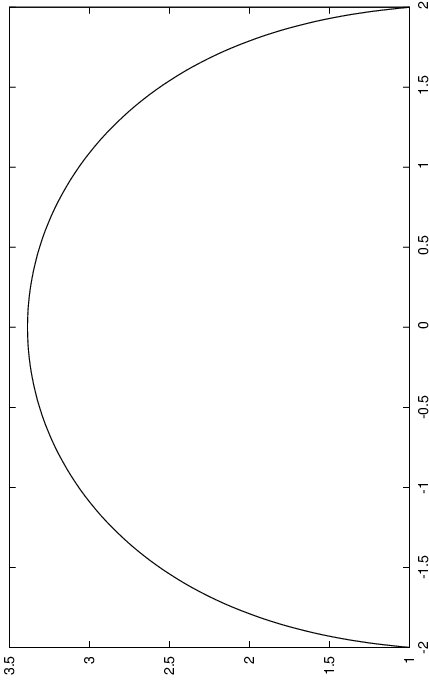}
\caption{
$(\BGNmckappa_m^\partial)^h$
Geodesics connecting the points $(\pm2,1)^T$ for (\ref{eq:gmu}) 
with $\mu = 0.1$ (left) and $\mu = 2$ (right). The discrete lengths are
$L_g^h(\vec X^M) = 3.977$ and $L_g^h(\vec X^M) = 1.645$, respectively.
} 
\label{fig:mc_geodesics2}
\end{figure}%

\subsubsection{Curve diffusion}
For curve diffusion in the hyperbolic plane, circles are steady state 
solutions. This follows from the fact that, analogously to the Euclidean case,
circles in the hyperbolic plane have constant curvature, see
(\ref{eq:app_varkappag}) in Appendix~\ref{sec:A1}.
For the scheme $(\BGNsd_m)^h$
we now show the evolutions of two cigar shapes towards a circle.
The discretization parameters are $J=128$ and $\ttau = 10^{-4}$.
In Figure~\ref{fig:sd_cigar} the initial shape is aligned horizontally,
whereas in Figure~\ref{fig:sd_cigar2} it is aligned vertically.
The relative area losses, measured in terms of (\ref{eq:Agh}), were
$-0.24\%$ and $0.04\%$ for these two simulations.
Repeating the simulations for the schemes $(\BGNsdstab_m)^h$ and 
$(\BGNsdstab_{m,\star})^h$ produces nearly identical results, with the main
difference being the larger ratios (\ref{eq:ratio}). For the simulations
corresponding to Figure~\ref{fig:sd_cigar}, the ratio reaches a value of 
about $6$, and the relative area loss is $-0.01\%$ for both
$(\BGNsdstab_m)^h$ and $(\BGNsdstab_{m,\star})^h$.
For the runs shown in Figure~\ref{fig:sd_cigar2}
the ratio (\ref{eq:ratio}) reaches a value around $2$,
and the relative area losses are $0.13\%$ and $0.14\%$, respectively.
\begin{figure}
\center
\begin{minipage}{0.35\textwidth}
\includegraphics[angle=-90,width=0.9\textwidth]{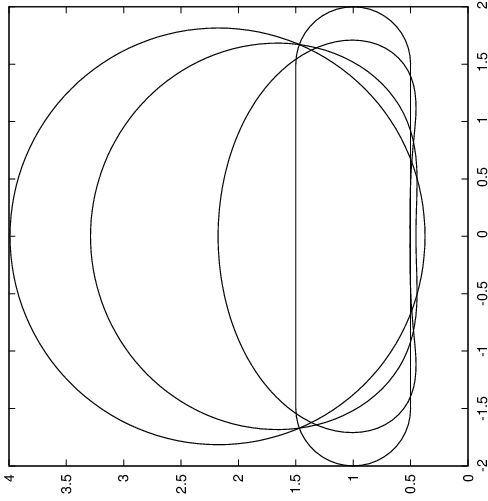}
\end{minipage} \
\begin{minipage}{0.6\textwidth}
\includegraphics[angle=-90,width=0.45\textwidth]{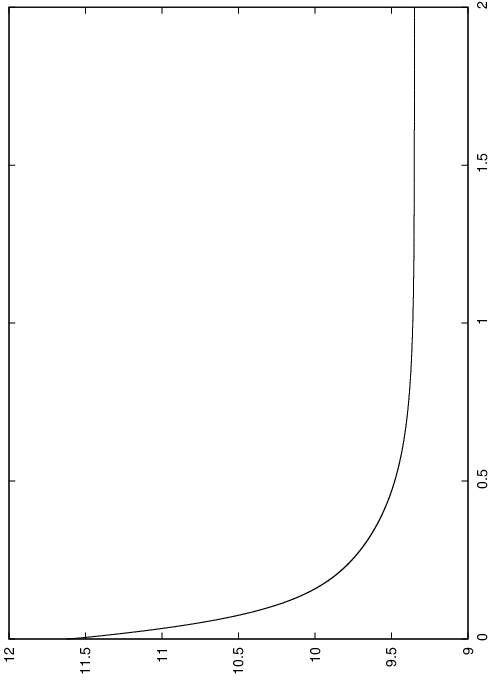}
\includegraphics[angle=-90,width=0.45\textwidth]{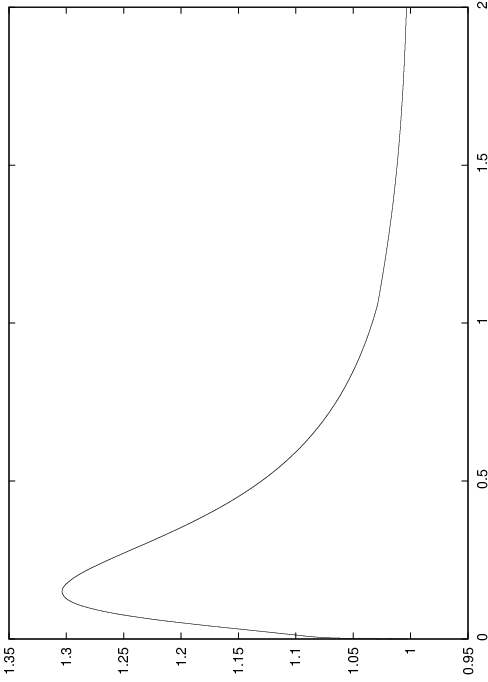}
\includegraphics[angle=-90,width=0.45\textwidth]{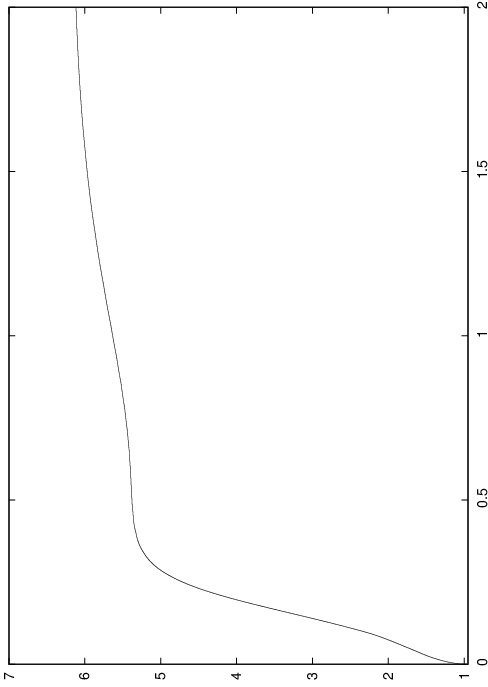}
\includegraphics[angle=-90,width=0.45\textwidth]{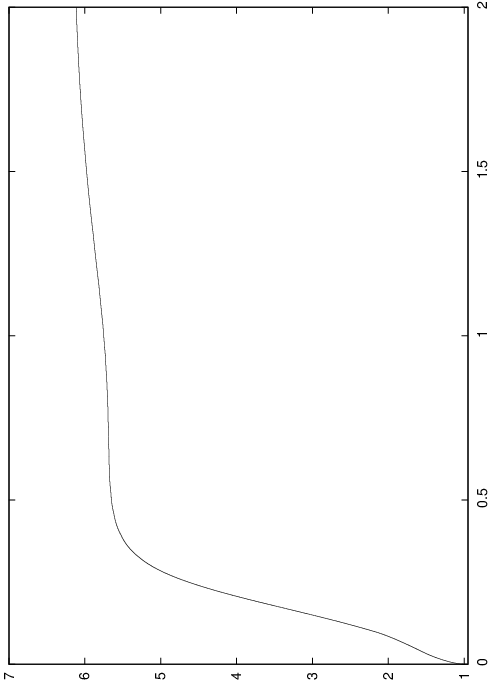}
\end{minipage}
\caption{
$(\BGNsd_m)^h$
Curve diffusion towards a circle. 
Solution at times $t=0,0.1,0.5,2$.
On the right are plots of the discrete energy $L_g^h(\vec X^m)$
and of the ratio (\ref{eq:ratio}), with plots of the ratio (\ref{eq:ratio})
for the schemes $(\BGNsdstab_m)^h$ and $(\BGNsdstab_{m,\star})^h$ below.}
\label{fig:sd_cigar}
\end{figure}%
\begin{figure}
\center
\center
\begin{minipage}{0.35\textwidth}
\includegraphics[angle=-90,width=0.6\textwidth]{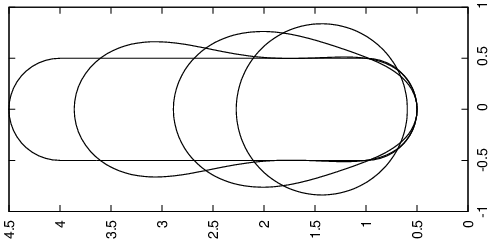}
\end{minipage} \
\begin{minipage}{0.6\textwidth}
\includegraphics[angle=-90,width=0.45\textwidth]{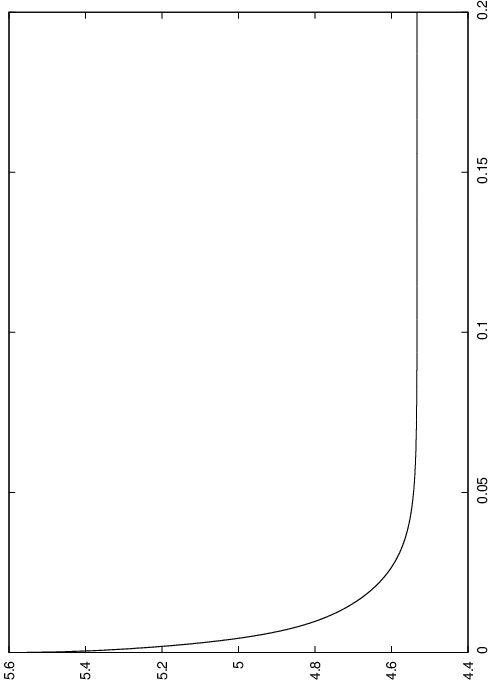}
\includegraphics[angle=-90,width=0.45\textwidth]{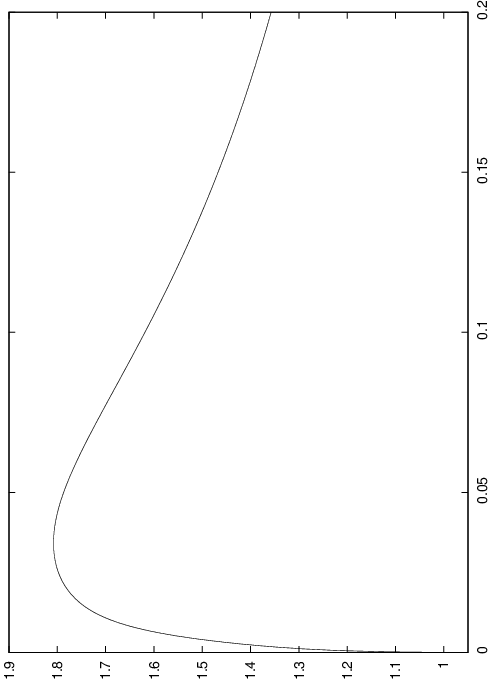}
\includegraphics[angle=-90,width=0.45\textwidth]{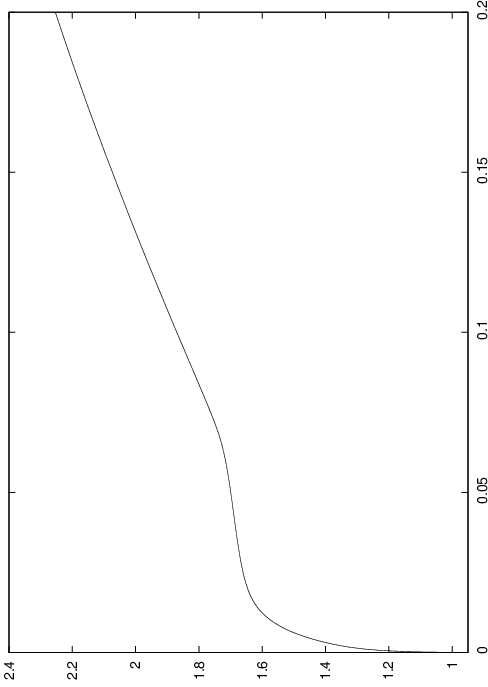}
\includegraphics[angle=-90,width=0.45\textwidth]{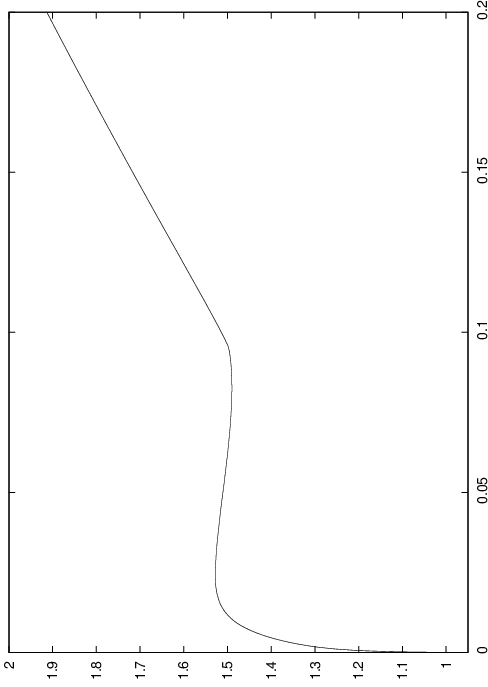}
\end{minipage}
\caption{
$(\BGNsd_m)^h$
Curve diffusion towards a circle. 
Solution at times $t=0,10^{-3},10^{-2},0.2$.
On the right are plots of the discrete energy $L_g^h(\vec X^m)$
and of the ratio (\ref{eq:ratio}), with plots of the ratio (\ref{eq:ratio})
for the schemes $(\BGNsdstab_m)^h$ and $(\BGNsdstab_{m,\star})^h$ below.} 
\label{fig:sd_cigar2}
\end{figure}%

For the metric (\ref{eq:gmu}) with $\mu\not\in\{0,1\}$, circles are in general
not steady state solutions for curve diffusion. 
We demonstrate this with 
numerical experiments for the metrics (\ref{eq:gmu}) with $\mu=0.1$ and
$\mu=2$. For the case $\mu=0.1$ we start from 
the initial data (\ref{eq:X0}) with $a(0) = 1.01$ and $r(0) = 1$,
and compute the evolution with  the scheme $(\BGNsd_{m})^h$ with the
discretization parameters $J=128$ and $\ttau = 10^{-3}$.
The results are shown in Figure~\ref{fig:sdmu01_circle}, 
where we note that the relative area loss, 
measured in terms of (\ref{eq:Agh}), was $-0.04\%$ for this 
experiment. The final shape has height $2.224$ and width $2.233$.
For the case $\mu=2$ we use the initial data (\ref{eq:X0}) with $a(0) = 2$ 
and $r(0) = 1$, and leave all the remaining parameters unchanged.
The evolution is shown in Figure~\ref{fig:sdmu2_circle2},
with a relative area loss of $0.22\%$. 
The final shape has height $0.03617$ and width $0.03609$.
\begin{figure}
\center
\includegraphics[angle=-90,width=0.3\textwidth]{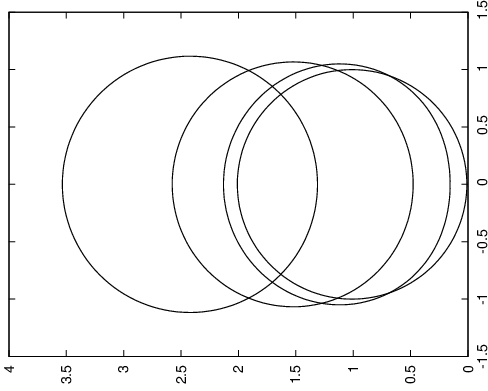}
\includegraphics[angle=-90,width=0.3\textwidth]{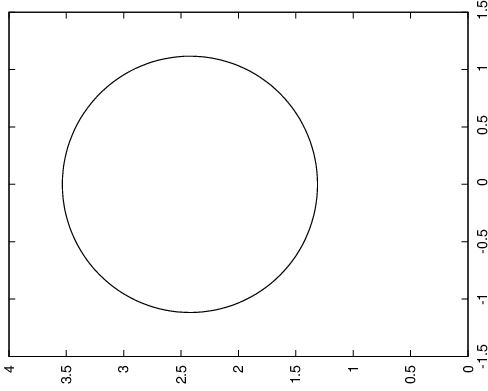}
\includegraphics[angle=-90,width=0.37\textwidth]{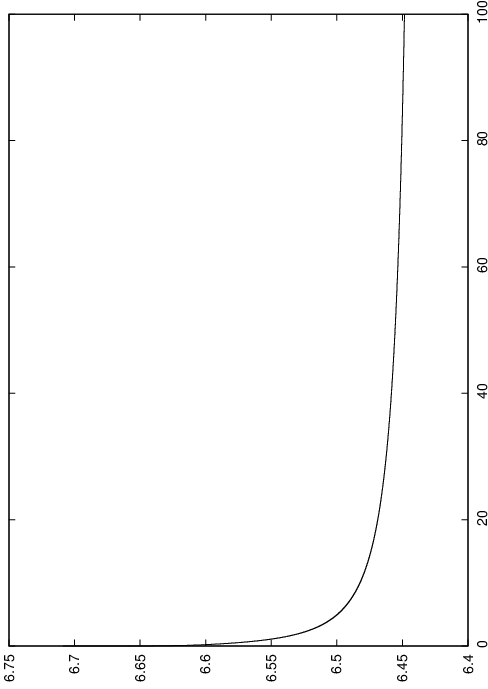}
\caption{
$(\BGNsd_m)^h$
Curve diffusion for (\ref{eq:gmu}) with $\mu=0.1$, starting from a circle. 
Solution at times $t=0,1,10,100$, and separately at time $t=100$.
On the right is a plot of the discrete energy $L_g^h(\vec X^m)$.}
\label{fig:sdmu01_circle}
\end{figure}%
\begin{figure}
\center
\includegraphics[angle=-90,width=0.3\textwidth]{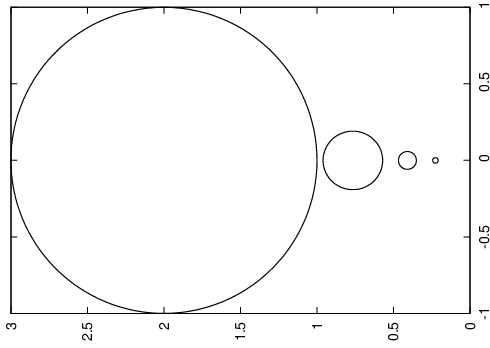}
\qquad
\includegraphics[angle=-90,width=0.5\textwidth]{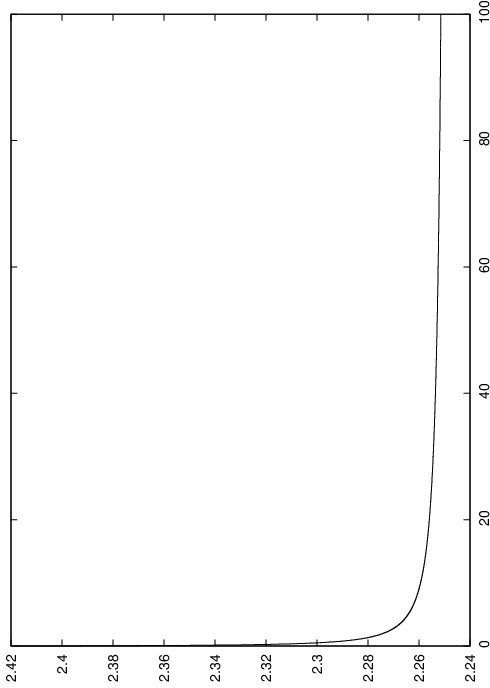}
\caption{
$(\BGNsd_m)^h$
Curve diffusion for (\ref{eq:gmu}) with $\mu=2$, starting from a circle. 
Solution at times $t=0,1,10,100$. 
On the right is a plot of the discrete energy $L_g^h(\vec X^m)$.}
\label{fig:sdmu2_circle2}
\end{figure}%

\subsubsection{Elastic flow}

For hyperbolic elastic flow, (\ref{eq:hb_elastflow}), we recall the
true solution (\ref{eq:app_ansatz}) with (\ref{eq:ODEa},b) from 
Appendix~\ref{sec:A1}.
We use this true solution for a convergence test for our two schemes for
elastic flow. Similarly to Table~\ref{tab:mcf1}
we start with the initial data (\ref{eq:X0}) 
with $r(0)=1$ and $a(0) = 1.1$.
We compute the error $\errorXx$ over the time interval $[0,1]$ between
the true solution (\ref{eq:app_ansatz}) 
and the discrete solutions for the schemes
$(\BGNwf_m)^h$ and $(\BGNwfwf_m)^h$. 
We recall from Appendix~\ref{sec:A1} 
that the circle will sink and shrink. In fact,
at time $T=1$ it holds that $r(T) = 0.645$ and $a(T) = 0.792$,
so that $\sigma(T) = \frac{a(T)}{r(T)} = 1.227 < 2^\frac12$, 
see Appendix~\ref{sec:A1}.
Here we use the time step size $\ttau=0.1\,h^2_{\Gamma^0}$,
where $h_{\Gamma^0}$ is the maximal edge length of $\Gamma^0$.
The computed errors are reported in Table~\ref{tab:wf1}.
\begin{table}
\center
\begin{tabular}{|rr|c|c|c|c|}
\hline
& & \multicolumn{2}{c|}{$(\BGNwf_m)^h$}&
\multicolumn{2}{c|}{$(\BGNwfwf_m)^h$} \\
$J$  & $h_{\Gamma^0}$ & $\errorXx$ & EOC & $\errorXx$ & EOC \\ \hline
32   & 2.1544e-01 & 3.5987e-02 & ---      & 3.1536e-02 & --- \\
64   & 1.0792e-01 & 8.7266e-03 & 2.049469 & 7.9745e-03 & 1.988856 \\         
128  & 5.3988e-02 & 2.1624e-03 & 2.014294 & 1.9958e-03 & 1.999924 \\
256  & 2.6997e-02 & 5.3929e-04 & 2.003821 & 4.9957e-04 & 1.998529 \\
512  & 1.3499e-02 & 1.3474e-04 & 2.000990 & 1.2489e-04 & 2.000136 \\
\hline
\end{tabular}
\caption{Errors for the convergence test for (\ref{eq:app_ansatz}) with
(\ref{eq:ODEa},b), with $r(0) = 1$, $a(0) = 1.1$,
over the time interval $[0,1]$.}
\label{tab:wf1}
\end{table}%
We repeat the convergence test with the initial data $r(0)=1$ and $a(0) = 2$,
so that the circle will now raise and expand. In fact,
at time $T=1$ it holds that $r(T) = 1.677$ and $a(T) = 2.411$.
so that $\sigma(T) = \frac{a(T)}{r(T)} = 1.437 > 2^\frac12$, 
see Appendix~\ref{sec:A1}.
The computed errors are reported in Table~\ref{tab:wf2}.
\begin{table}
\center
\begin{tabular}{|rr|c|c|c|c|}
\hline
& & \multicolumn{2}{c|}{$(\BGNwf_m)^h$}&
\multicolumn{2}{c|}{$(\BGNwfwf_m)^h$} \\
$J$  & $h_{\Gamma^0}$ & $\errorXx$ & EOC & $\errorXx$ & EOC \\ \hline
32   & 2.1544e-01 & 1.8228e-01 & ---      & 4.0407e-02 & --- \\
64   & 1.0792e-01 & 4.3289e-02 & 2.079649 & 1.0436e-02 & 1.958277 \\         
128  & 5.3988e-02 & 1.0699e-02 & 2.018035 & 2.6286e-03 & 1.990692 \\
256  & 2.6997e-02 & 2.6668e-03 & 2.004616 & 6.5835e-04 & 1.997688 \\
512  & 1.3499e-02 & 6.6621e-04 & 2.001168 & 1.6467e-04 & 1.999384 \\
\hline
\end{tabular}
\caption{Errors for the convergence test for (\ref{eq:app_ansatz}) with
(\ref{eq:ODEa},b), with $r(0) = 1$, $a(0) = 2$,
over the time interval $[0,1]$.}
\label{tab:wf2}
\end{table}%

For the scheme $(\BGNwf_m)^h$
we show the evolution of a cigar shape in Figure~\ref{fig:wf_cigar}.
The discretization parameters are $J=128$ and $\ttau = 10^{-4}$.
Rotating the initial shape by $90^\circ$ degrees yields the evolution in
Figure~\ref{fig:wf_cigar2}. 
As expected, in both cases the curve evolves to a circle.
In the first case, at time $t_m=4$ it holds
that $\sigma^m = \frac{a^m}{r^m} = 1.412 < 2^\frac12$, where
$a^m = 
\tfrac12\,(\max_I \vec X^m\,.\,\vec\ek_2 + \min_I \vec X^m\,.\,\vec\ek_2)$
and 
$r^m =
\tfrac12\,(\max_I \vec X^m\,.\,\vec\ek_2 - \min_I \vec X^m\,.\,\vec\ek_2)$,
and so the approximate circle is going to continue to
sink and shrink. This is evidenced by the
plot of $a^m$ over time in Figure~\ref{fig:wf_cigar},
where we see that $a^m$ eventually decreases.
In the second simulation, on the other hand, we observe
at time $t_m=2$ that $\sigma_m = 1.415 > 2^\frac12$, and so here the 
approximate circle will continue to rise and expand, which can also 
be seen from the plot of $a^m$ over time in Figure~\ref{fig:wf_cigar2}.
\begin{figure}
\center
\begin{minipage}{0.45\textwidth}
\includegraphics[angle=-90,width=0.99\textwidth]{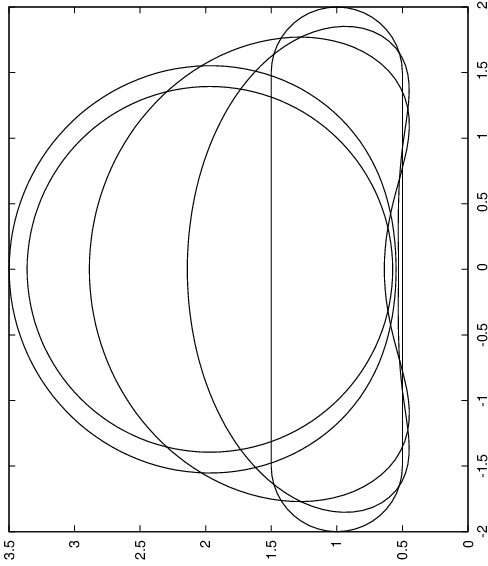}
\end{minipage}
\begin{minipage}{0.5\textwidth}
\includegraphics[angle=-90,width=0.7\textwidth]{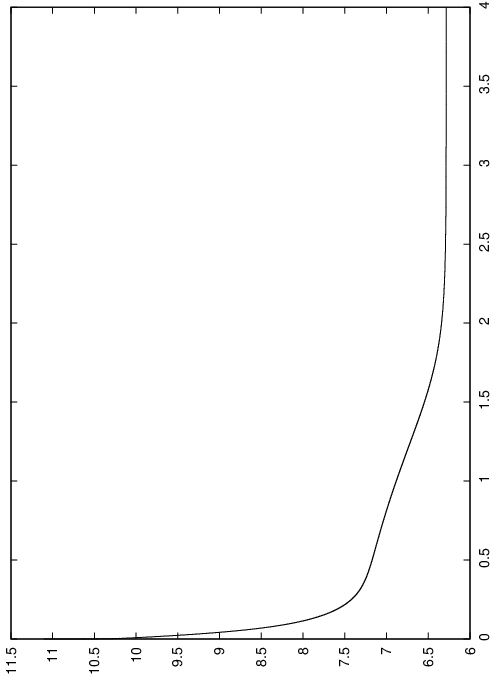}
\includegraphics[angle=-90,width=0.7\textwidth]{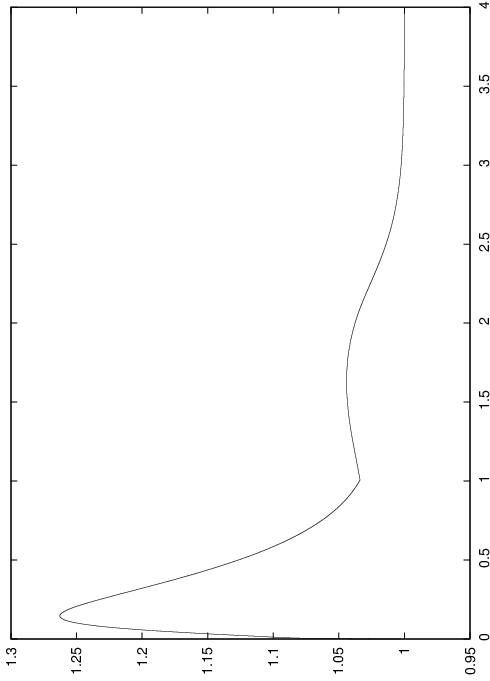}
\includegraphics[angle=-90,width=0.7\textwidth]{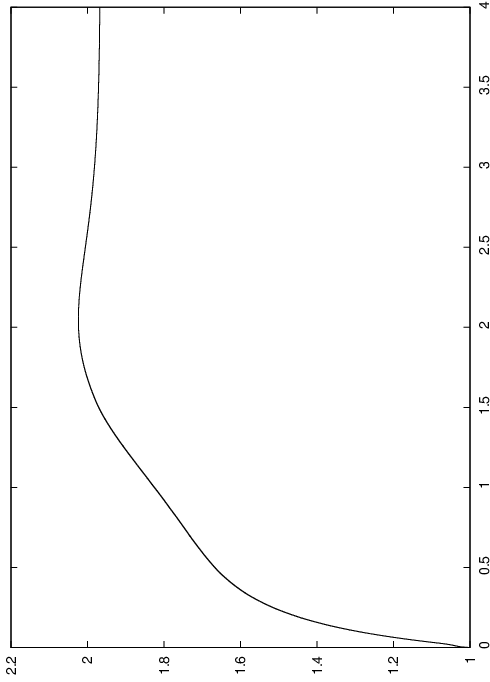}
\end{minipage}
\caption{
$(\BGNwf_m)^h$
Hyperbolic elastic flow towards a sinking and shrinking circle. 
Solution at times $t=0,0.1,0.5,2,4$.
On the right are plots of the discrete energy (\ref{eq:Wgh}), 
of the ratio (\ref{eq:ratio}) and of $a^m$.} 
\label{fig:wf_cigar}
\end{figure}%
\begin{figure}
\center
\begin{minipage}{0.4\textwidth}
\includegraphics[angle=-90,width=0.95\textwidth]{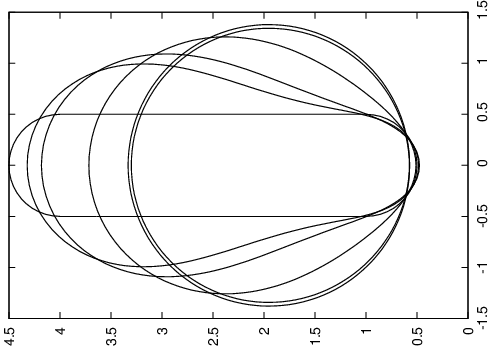}
\end{minipage} \qquad
\begin{minipage}{0.5\textwidth}
\includegraphics[angle=-90,width=0.7\textwidth]{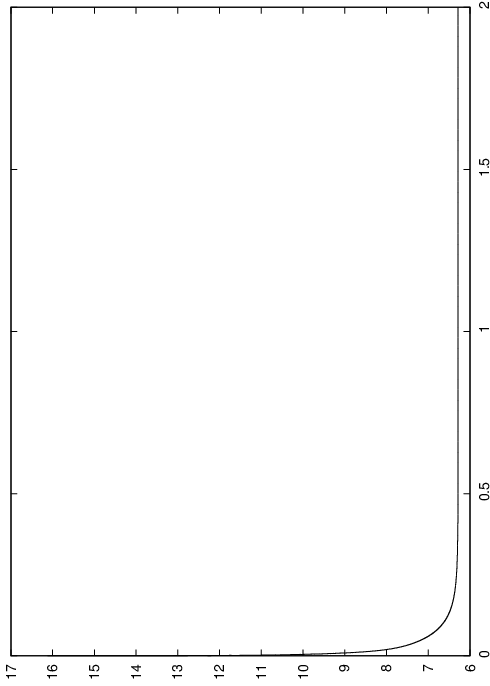}\\
\includegraphics[angle=-90,width=0.7\textwidth]{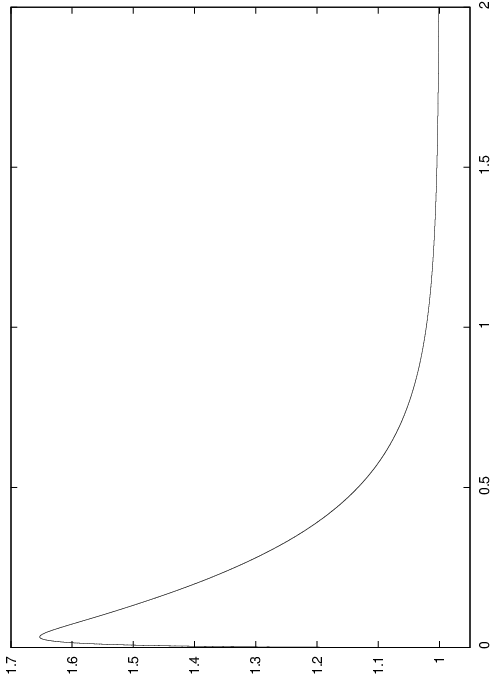}\\
\includegraphics[angle=-90,width=0.7\textwidth]{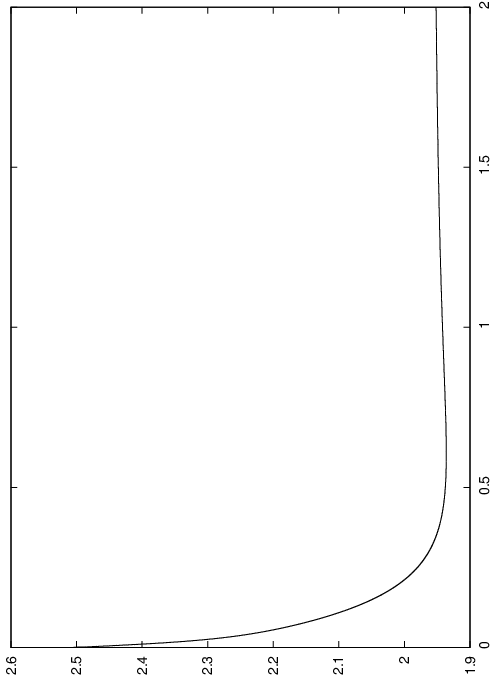}\\
\end{minipage}
\caption{
$(\BGNwf_m)^h$
Hyperbolic elastic flow towards a rising and expanding circle. 
Solution at times $t=0,0.01,0.02,0.1,0.5,2$.
On the right are plots of the discrete energy (\ref{eq:Wgh}),  
of the ratio (\ref{eq:ratio}) and of $a^m$.} 
\label{fig:wf_cigar2}
\end{figure}%
The same computations for the scheme $(\BGNwfwf_m)^h$ yield almost identical
results, with the main difference being the evolution of the ratio 
(\ref{eq:ratio}). We present the plots of this quantity for the scheme 
$(\BGNwfwf_m)^h$ for these two simulations in Figure~\ref{fig:wfwf_r}, where we
observe that the obtained curves are far from being equidistributed, although
the ratio (\ref{eq:ratio}) remains bounded, eventually settling on a value
close to $4$.
\begin{figure}
\center
\includegraphics[angle=-90,width=0.4\textwidth]{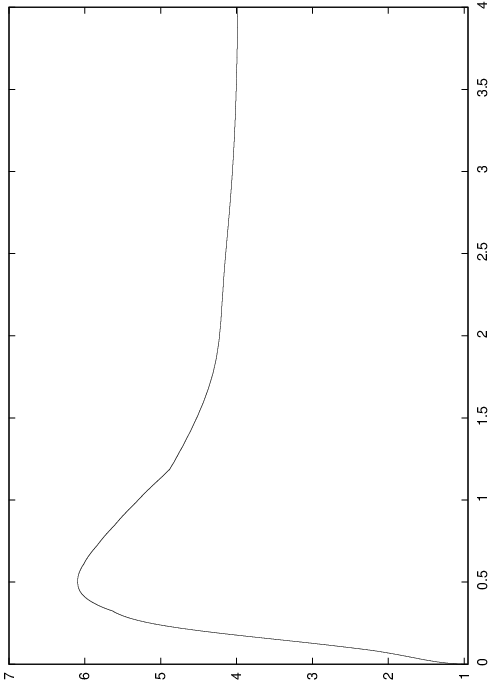}
\includegraphics[angle=-90,width=0.4\textwidth]{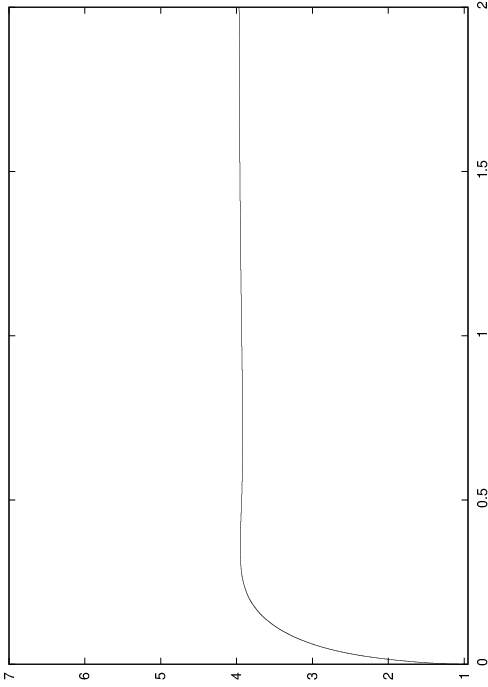}
\caption{
$(\BGNwfwf_m)^h$
The ratio plots (\ref{eq:ratio}) for the two simulations as in 
Figure~\ref{fig:wf_cigar} and \ref{fig:wf_cigar2}.} 
\label{fig:wfwf_r}
\end{figure}%

Finally, on recalling Remark~\ref{rem:lambda}, we repeat the simulation in
Figure~\ref{fig:wf_cigar} now for the scheme $(\BGNwfwf_m^{\lambda})^h$ 
with $\lambda=1$. As expected, the length penalization means that now the 
evolution reaches a steady state, as can be seen from the plots in
Figure~\ref{fig:wfwf_cigar_lambda}.
\begin{figure}
\center
\begin{minipage}{0.45\textwidth}
\includegraphics[angle=-90,width=0.99\textwidth]{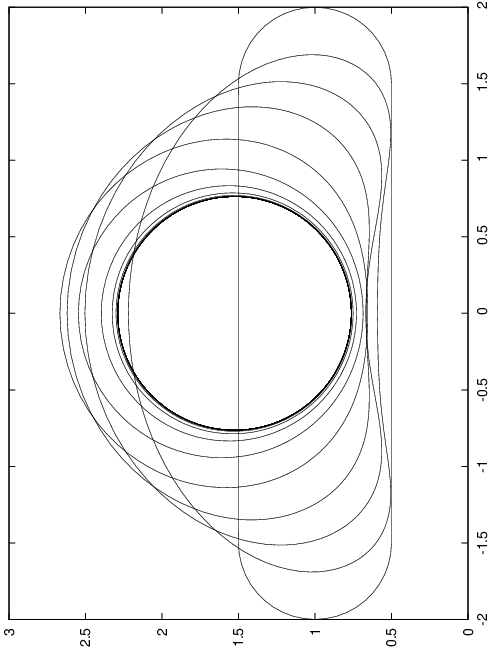}
\end{minipage}
\begin{minipage}{0.5\textwidth}
\includegraphics[angle=-90,width=0.7\textwidth]{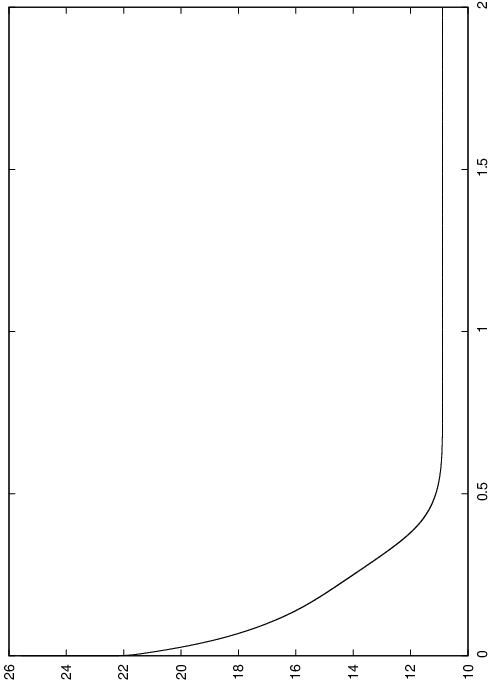}
\includegraphics[angle=-90,width=0.7\textwidth]{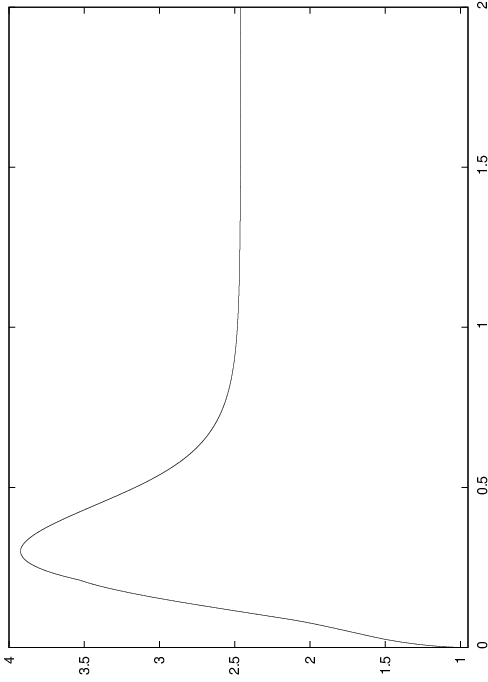}
\end{minipage}
\caption{
$(\BGNwfwf_m^{\lambda})^h$
Generalized hyperbolic elastic flow, with $\lambda=1$, towards a circle. 
Solution at times $t=0,0.1,\ldots,2$.
On the right are plots of the discrete energy $\widetilde W_g^m + 
\lambda\,L_g^h(\vec X^m)$, and of the ratio (\ref{eq:ratio}).} 
\label{fig:wfwf_cigar_lambda}
\end{figure}%

\subsection{The elliptic plane, and (\ref{eq:galpha}) for $\alpha \in \bR$}

Unless otherwise stated, all our computations in this section 
are for the elliptic plane,
i.e.\ (\ref{eq:gstereo}) or, equivalently, (\ref{eq:galpha}) with $\alpha=-1$.

Similarly to Figure~\ref{fig:mc_geodesics}, we use the 
scheme $(\BGNmckappa_m^\partial)^h$ to compute some geodesics
in the elliptic plane. Here it can happen that a finite geodesic does not
exist, and so the evolution of curvature flow will yield a curve that expands
continuously. We visualize this effect in Figure~\ref{fig:ellipt_geodesic}.
Here the initial curve consists of two straight line segments which
connect the points $(\pm9,\mp1)^T$ with $(9,9)^T$ in $H$. 
As the discretization parameters we choose $J=128$ and 
$\ttau = 10^{-4}$.
\begin{figure}
\center
\includegraphics[angle=-90,width=0.45\textwidth]{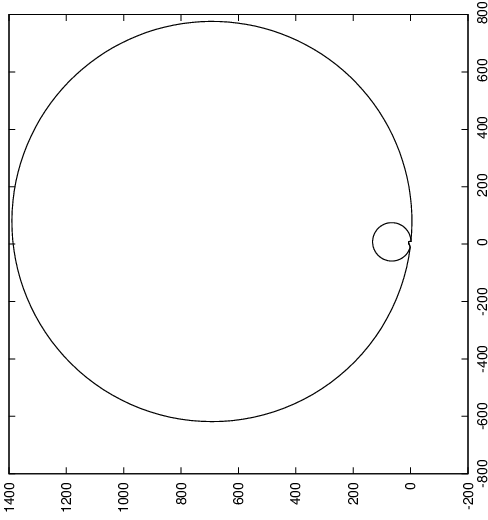}
\includegraphics[angle=-90,width=0.45\textwidth]{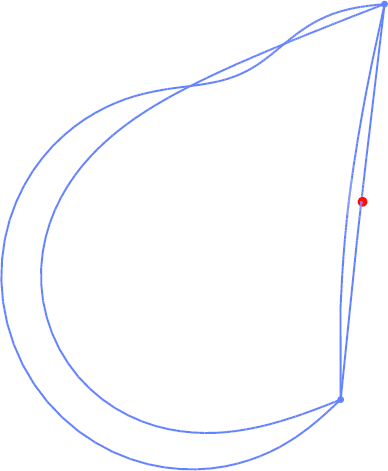}
\caption{
$(\BGNmckappa_m^\partial)^h$
Curvature flow towards an infinite geodesic in the elliptic plane.
The solutions $\vec X^m$
at times $t = 10^{-3}, 10^{-2}, 0.1, 1$. On the right we visualize
$\vec\Phi(\vec X^m)$ at the same times, for (\ref{eq:gstereo}), with the
north pole, $\vec\ek_3$, represented by a red dot.
} 
\label{fig:ellipt_geodesic}
\end{figure}%

\subsection{Geodesic evolution equations}

In order to demonstrate the possibility to compute geodesic evolution laws with
the introduced approximations, we present a computation for geodesic curvature
flow on a Clifford torus. To this end, we employ the metric induced by
(\ref{eq:gtorus}) with $\mathfrak s = 1$, so that the torus has radii $r=1$ and
$R = 2^\frac12$. As initial data we choose a circle in $H$ with radius $4$ and
centre $(0,2)^T$. For the simulation in Figure~\ref{fig:mctorus} we use the
scheme $(\BGNmckappa_m)^h$ with the discretization parameters 
$J=256$ and $\ttau=10^{-3}$. In $H$ it can be observed that the initial circle
deforms and shrinks to a point. On the surface $\mathcal{M} = \vec\Phi(H)$,
the initial curve is homotopic to a point, and so unravels and then shrinks to
a point.
\begin{figure}
\center
\includegraphics[angle=-90,width=0.45\textwidth]{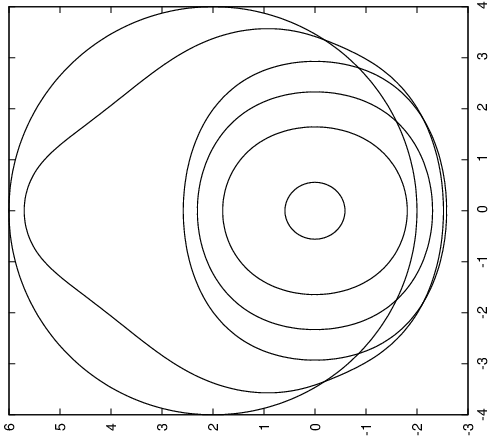}\quad
\includegraphics[angle=-90,width=0.45\textwidth]{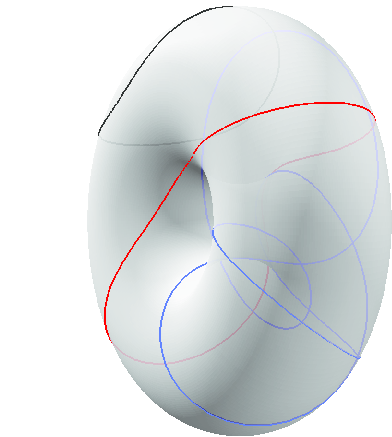}
\caption{
$(\BGNmckappa_m)^h$
Geodesic curvature flow on a Clifford torus.
The solutions $\vec X^m$
at times $t = 0, 1, 10, 20, 30, 39$. On the right we visualize
$\vec\Phi(\vec X^m)$ at times $t=0,30,39$, for (\ref{eq:gtorus}) with 
$\mathfrak s=1$.
} 
\label{fig:mctorus}
\end{figure}%

\section*{Conclusions}
We have derived and analysed various finite element schemes for the numerical
approximation of curve evolutions in two-dimensional Riemannian manifolds. 
The considered
evolution laws include curvature flow, curve diffusion and
elastic flow. The Riemannian manifolds that can be considered in our framework
include the hyperbolic plane, the hyperbolic disk and the elliptic plane. 
More generally, any metric conformal to the two-dimensional Euclidean metric
can be considered. We mention that locally this is
always possible for two-dimensional Riemannian manifolds. An example of this
are two-dimensional surfaces in $\bR^d$, $d\geq3$, which are conformally
parameterized.
Our approach also allows
computations for geometric evolution equations of axisymmetric hypersurfaces in
$\bR^d$, $d\geq3$.

For the standard Euclidean plane our proposed schemes collapse to variants
introduced by the authors in much earlier papers, see 
\cite{triplej,triplejMC}.

\begin{appendix}

\renewcommand{\theequation}{A.\arabic{equation}}
\setcounter{equation}{0}
\section{Some exact circular solutions} \label{sec:A}

Here we state some exact solutions for the three geometric evolution 
equations we consider, i.e.\ (\ref{eq:Vgkg}), (\ref{eq:sdg}) and
(\ref{eq:g_elastflow}), for selected metrics $g$. 

\subsection{The hyperbolic plane} \label{sec:A1}

Here we consider circular solutions in the hyperbolic plane, based on the 
exact solution for hyperbolic elastic flow 
from \citet[Lemma~3.1]{DallAcquaS18preprint}.

In particular, we make the ansatz
\begin{equation} \label{eq:app_ansatz}
\vec x(\rho, t) = a(t)\,\vec\ek_2 + r(t)\left[
\cos2\,\pi\,\rho\,\vec\ek_1 + \sin2\,\pi\,\rho\,\vec\ek_2 \right]
\qquad \rho \in I\,,
\end{equation}
for $a(t) > r(t) > 0$ for all $t \in [0,T]$. Then it follows from
(\ref{eq:mu_varkappag}) for $\mu=1$ that 
\begin{equation} \label{eq:app_varkappag}
\varkappa_g(\rho,t) = \frac{a(t)}{r(t)} \qquad \qquad \rho \in I\,,\
t \in [0,T]\,.
\end{equation}
Moreover, it holds that
\begin{align}
\mathcal{V}_g = (\vec x\,.\,\vec\ek_2)^{-1}\,\vec x_t\,.\,\vec\nu = 
- \left( a(t) + r(t)\,\sin2\,\pi\,\rho \right)^{-1}\,
\left[ a'(t)\,\sin2\,\pi\,\rho + r'(t) \right] .
\label{eq:app_Vg}
\end{align}

We now consider curvature flow, (\ref{eq:Vgkg}). 
With the ansatz (\ref{eq:app_ansatz}), on noting (\ref{eq:app_varkappag})
and (\ref{eq:app_Vg}), we have that (\ref{eq:Vgkg}) reduces to
\begin{equation} \label{eq:app_Vgkg}
a'(t)\,\sin2\,\pi\,\rho + r'(t) = - (a(t) + r(t)\,\sin2\,\pi\,\rho)\,
\frac{a(t)}{r(t)} \,.
\end{equation}
Differentiating (\ref{eq:app_Vgkg}) with respect to $\rho$ yields that
$a'(t) = - a(t)$, and hence $a(t) = e^{-t}\,a(0)$. 
Combining this with (\ref{eq:app_Vgkg}) yields that
\begin{equation*} 
r'(t) = - a(t) \,\frac{a(t)}{r(t)} 
\qquad \Rightarrow \quad
\tfrac12\,\ddt r^2(t) = - a^2(t) = - e^{-2\,t}\,a^2(0)\,.
\end{equation*}
Hence 
\begin{equation*} 
r^2(t) - r^2(0) = - 2\,a^2(0)\,\int_0^t e^{-2\,u}\;{\rm d}u = 
- 2\,a^2(0)\left[ -\tfrac12\,e^{-2\,t} + \tfrac12 \right]
= a^2(0)\left[ e^{-2\,t} -1 \right] ,
\end{equation*}
and so (\ref{eq:app_ansatz}) with
\begin{equation} \label{eq:true_Vgkg}
a(t) = e^{-t}\,a(0)\,,\quad
r(t) = \left( r^2(0) - a^2(0) \left[ 1 - e^{-2\,t} \right]\right)^\frac12
\end{equation}
is a solution to (\ref{eq:Vgkg}). We observe that circles move towards the
$\vec\ek_1$--axis and shrink as they do so. The finite extinction time is 
$T_0 = -\frac12\,\ln \left[1 - \left(\frac{r(0)}{a(0)}\right)^2 \right]$.

As regards (\ref{eq:sdg}), it is obvious from (\ref{eq:app_varkappag}) that 
any solution of the form (\ref{eq:app_ansatz}) satisfies $\mathcal{V}_g = 0$, 
and so circles are stationary solutions for curve diffusion.

Finally, for the elastic flow (\ref{eq:g_elastflow}), we recall
the exact solution for the hyperbolic elastic flow,
(\ref{eq:hb_elastflow}), from \citet[Lemma~3.1]{DallAcquaS18preprint}.

With the ansatz (\ref{eq:app_ansatz}), on noting (\ref{eq:app_varkappag})
and (\ref{eq:app_Vg}), we have that (\ref{eq:hb_elastflow}) reduces to
\begin{equation} \label{eq:app_elastflow}
a'(t)\,\sin2\,\pi\,\rho + r'(t) = 
- (a(t) + r(t)\,\sin2\,\pi\,\rho)
\left(\frac{a(t)}{r(t)} - \frac{a^3(t)}{2\,r^3(t)} \right).
\end{equation}
Differentiating (\ref{eq:app_elastflow}) with respect to $\rho$ yields that
\begin{equation} \label{eq:app_ar2}
a'(t) =  - r(t) \left( \frac{a(t)}{r(t)} - \frac{a^3(t)}{2\,r^3(t)} \right),
\end{equation}
and combining this with (\ref{eq:app_elastflow}) yields that
\begin{equation} \label{eq:app_ar3}
r'(t) = - a(t) \left( \frac{a(t)}{r(t)} - \frac{a^3(t)}{2\,r^3(t)} \right) .
\end{equation}
On setting 
\begin{equation} \label{eq:sigma}
\sigma(t) = \frac{a(t)}{r(t)} > 1  \qquad \Rightarrow \quad 
\sigma'(t) = \frac{a'(t)}{r(t)} - \frac{a(t)\,r'(t)}{r^2(t)}\,,
\end{equation}
it follows from (\ref{eq:app_ar2}) and (\ref{eq:app_ar3}) that
\begin{equation} \label{eq:app_ar4}
\sigma'(t) = \sigma(t)\,(1 - \tfrac12\,\sigma^2(t))\,(\sigma^2(t) - 1)\,,
\end{equation}
which agrees with (3.4) in \cite{DallAcquaS18preprint} for $\lambda=0$.
If $\sigma$ denotes a solution to (\ref{eq:app_ar4}), then it follows from
(\ref{eq:app_ar2}) that $a$ and $r$ satisfy 
\begin{subequations}
\begin{align}
\ddt\,\ln a(t) &= \frac{a'(t)}{a(t)} = 
\tfrac12\,\sigma^2(t) - 1 \quad\Rightarrow\
a(t) = a(0)\,\exp\left(-t + \tfrac12\,\int_0^t \sigma^2(u)\;{\rm d}u \right)
, \label{eq:ODEa} \\
r(t) &= \frac{a(t)}{\sigma(t)} \,.
\label{eq:ODEr}
\end{align}
\end{subequations}
On recalling that $\sigma(t)>1$, we note that
$\sigma(t) = 2^\frac12$ is the only steady state solution of 
(\ref{eq:app_ar4}), and hence circles with ratios 
$\sigma(t) = 2^\frac12$ are steady state solutions
of (\ref{eq:hb_elastflow}). 
Moreover, circles with $\sigma(t) > 2^\frac12$ will rise and
expand indefinitely in time, reducing the ratio $\sigma(t) > 2^\frac12$ as they
do so. On the other hand, circles with $\sigma(t) < 2^\frac12$ will sink and
shrink indefinitely in time, increasing the ratio $\sigma(t) < 2^\frac12$
as they do so.

In order to compute solutions to (\ref{eq:app_ar4}) in practice, we
let $F(y) = y^{-1}\,|1-\tfrac12\,y^2|^{-\frac12}\,(y^2-1)$,
so that $F\in C^\infty((1,2^\frac12)\cup(2^\frac12,\infty))$.
Then $F'(y) = y^{-2}\,(1-\tfrac12\,y^2)^{-1}\,|1-\tfrac12\,y^2|^{-\frac12}$, 
and hence a solution $\sigma$ to (\ref{eq:app_ar4}) satisfies 
\begin{align}
\ddt\, F(\sigma(t)) = \sigma'(t)\,F'(\sigma(t)) = 
\frac{\sigma^2(t)-1}{\sigma(t)\,|1-\tfrac12\,\sigma^2(t)|^\frac12} 
= F(\sigma(t))\,,
\label{eq:app_dFdt}
\end{align}
which means that a solution $\sigma$ to (\ref{eq:app_ar4}) satisfies 
the nonlinear equation
\begin{equation} \label{eq:app_sigma_nonlinear}
F(\sigma(t)) = F(\sigma(0))\,e^t\,.
\end{equation}

\subsection{The hyperbolic disk and the elliptic plane} \label{sec:A2}

Here we consider the metric (\ref{eq:galpha}).
For $\alpha=1$ we then obtain exact solutions for the hyperbolic disk,
while $\alpha=-1$ corresponds to the elliptic plane.
In the latter case these solutions can be related to the
exact solutions for the corresponding geodesic flows on the sphere 
from \cite{curves3d}, recall \S\ref{sec:conformal}.

In particular, on making the ansatz
\begin{equation} \label{eq:appB_ansatz}
\vec x(\rho, t) = r(t)\left[
\cos2\,\pi\,\rho\,\vec\ek_1 + \sin2\,\pi\,\rho\,\vec\ek_2 \right]
\qquad \rho \in I\,,
\end{equation}
for $r(t) > 0$ for all $t \in [0,T]$, it follows from
(\ref{eq:alpha_varkappag}) that 
\begin{equation} \label{eq:appB_varkappag}
\varkappa_g(\rho,t) = \tfrac12\left(1 + \alpha\,r^2(t)\right) [r(t)]^{-1}
 \qquad \qquad \rho \in I\,,\ t \in [0,T]\,.
\end{equation}
Moreover, it holds that
\begin{align}
\mathcal{V}_g = g^\frac12(\vec x)\,\vec x_t\,.\,\vec\nu = 
- 2\,(1 - \alpha\,r^2(t))^{-1}\,r'(t)\,.
\label{eq:appB_Vg}
\end{align}

We now consider curvature flow, (\ref{eq:Vgkg}). 
It follows from (\ref{eq:appB_Vg}) and (\ref{eq:appB_varkappag}) that
\begin{equation} \label{eq:appB_ode}
\ddt\,r^2(t) = \tfrac12 \left( \alpha^2\,r^4(t) - 1 \right)\,.
\end{equation}
Clearly, if $\alpha = 0$ then $r(t) = [r^2(0) - \tfrac12\,t]^\frac12$ is the
well-known shrinking circle solution for Euclidean curvature flow.
For $\alpha \not = 0$, in order to compute solutions to 
(\ref{eq:appB_ode}) in practice, we let $G(y) = \left|(1 + \alpha\,y^2)^{-1}\,
(1 - \alpha\,y^2)\right|^\frac1\alpha$, so that 
$G \in C^\infty((0,|\alpha|^{-\frac12}) \cup (|\alpha|^{-\frac12},\infty))$, 
recall also (\ref{eq:galpha}).
Then $G'(y) = 4\,y\,(\alpha^2\,y^4 - 1)^{-1}\,G(y)$
and hence a solution to (\ref{eq:appB_ode}) satisfies 
\begin{align}
\ddt\, G(r(t)) = r'(t)\,G'(r(t)) = G(r(t))\,,
\label{eq:appB_dGdt}
\end{align}
which means that a solution to (\ref{eq:appB_ode}) satisfies 
the nonlinear equation
\begin{equation} \label{eq:appB_Gr}
G(r(t)) = G(r(0))\,e^t\,,
\end{equation}
which can be inverted explicitly.
In the case $\alpha=-1$, we recall from (\ref{eq:gstereo}) that the circle 
(\ref{eq:appB_ansatz}) of radius $r(t)$ in the elliptic plane corresponds to a
circle of radius $R(t) = 2\,(1 + r^2(t))^{-1}\,r(t)$, 
and at height $(r^2(t) + 1)^{-1}\,(r^2(t) - 1)$, on the unit sphere in $\bR^3$.
It can be easily shown that if $r(t)$ satisfies (\ref{eq:appB_Gr}) for
$\alpha=-1$, then
\begin{equation} \label{eq:c3d_Vgkg}
R(t) = [1 - (1-R^2(0))\,e^{2\,t}]^{\frac12}\,, 
\end{equation}
which is the solution of geodesic curvature flow on the unit sphere,
given by \citet[(5.6)]{curves3d}. Observe that for $R(0) \in (0,1)$, the 
finite extinction time is $T_0 = \tfrac12\,\ln\frac1{1-R^2(0)}$.

As regards (\ref{eq:sdg}), it is obvious from (\ref{eq:appB_varkappag}) that 
any solution of the form (\ref{eq:appB_ansatz}) satisfies 
$\mathcal{V}_g = 0$, and so circles centred at the origin
are stationary solutions to curve diffusion for (\ref{eq:galpha}).

Finally, we consider the elastic flow (\ref{eq:g_elastflow}).
With the ansatz (\ref{eq:appB_ansatz}), on noting (\ref{eq:appB_varkappag}),
(\ref{eq:appB_Vg}) and (\ref{eq:alpha_S0}), 
we have that (\ref{eq:g_elastflow}) reduces to
\begin{equation} \label{eq:appB_elastflow}
- 2\,(1 - \alpha\,r^2(t))^{-1}\, r'(t) = 
- \tfrac1{16}\,(1 + \alpha\,r^2(t))^3\,r^{-3}(t) + \tfrac12\,\alpha\,
(1 + \alpha\,r^2(t))\,r^{-1}(t) \,.
\end{equation}
This implies the ODE
\begin{align} \label{eq:appB_ODE}
\ddt\, r^4(t) & = \tfrac1{8}\,(1 - \alpha^2\,r^4(t))\,
(1 - 6\,\alpha\,r^2(t) + \alpha^2\,r^4(t)) \,.
\end{align}
In the case $\alpha=-1$, we recall from (\ref{eq:gstereo}) that the circle 
(\ref{eq:appB_ansatz}) of radius $r(t)$ in the elliptic plane corresponds to a
circle of radius $R(t) = 2\,(1 + r^2(t))^{-1}\,r(t)$, 
and at height $(r^2(t) + 1)^{-1}\,(r^2(t) - 1)$, on the unit sphere in $\bR^3$.
It can be easily shown that if $r(t)$ satisfies (\ref{eq:appB_ODE}) for
$\alpha=-1$, then $R(t)$ satisfies $\ddt\,R^4(t) = 2\,(1 - R^4(t))$, i.e.\
\begin{equation} \label{eq:c3d_elastflow}
R(t) = [1 - (1-R^4(0))\,e^{-2\,t}]^{\frac14}, 
\end{equation}
which is the solution for geodesic elastic flow on the unit sphere given by
\citet[(5.7)]{curves3d}. Hence in the case $\alpha=-1$ we can obtain $r(t)$
from
\begin{equation} \label{eq:appB_elliptr}
r(t) = R^{-1}(t)\,
\begin{cases}
(1 + [1 - R^2(t)]^\frac12) & r(0) \geq 1\,, \\
(1 - [1 - R^2(t)]^\frac12) & r(0) < 1\,,
\end{cases}
\quad\text{where}\quad
R(t) = [1 - (1-R^4(0))\,e^{-2\,t}]^{\frac14}\,,
\end{equation}
with $R(0) = 2\,(1 + r^2(0))^{-1}\,r(0)$.

Finally, for the case $\alpha=1$, it follows from (\ref{eq:appB_ODE}) that
\begin{align} \label{eq:appB_ODE1}
\ddt\, r^4(t) & = \tfrac1{8}\,(1 - r^4(t))\,(1 - 6\,r^2(t) + r^4(t)) \,.
\end{align}
We note, on recalling (\ref{eq:galpha}),
that $r(t) = 2^\frac12 - 1$ is a stable stationary solution to
(\ref{eq:appB_ODE1}). Hence circles with larger radii will shrink, and circles
with smaller radii will expand. In order to solve (\ref{eq:appB_ODE1}) in
practice, we define
$Q(y) = (1+y^2)^{-1}\,|1 - 6\,y^2 + y^4|^{-\frac12}\,(1-y^2)^2$,
so that $Q\in C^\infty((0,2^\frac12-1) \cup (2^\frac12-1,1))$ with
$Q'(y) 
= 32\,(1-y^4)^{-1}\,(1 - 6\,y^2 + y^4)^{-1}\,y^3\,Q(y)$. 
Hence 
$\ddt\,Q(r(t)) = Q(r(t))$, and so a solution to (\ref{eq:appB_ODE1}) satisfies
the nonlinear equation
\begin{equation} \label{eq:appB_Qr}
Q(r(t)) = Q(r(0))\,e^t\,.
\end{equation}
We remark that an alternative to (\ref{eq:appB_elliptr}) for the case
$\alpha=-1$ is to solve, similarly to (\ref{eq:appB_Qr}), the nonlinear 
equation $Q_-(r(t)) = Q_-(r(0))\,e^t$, where
$Q_-(y) = (1-y^2)^{-1}\,(1 + 6\,y^2 + y^4)^{-\frac12}\,(1+y^2)^2$.

\renewcommand{\theequation}{B.\arabic{equation}}
\setcounter{equation}{0}
\section{Geodesic curve evolution equations} \label{sec:C}

Let $\vec\Phi : H \to \bR^d$ 
be a conformal parameterization of an embedded 
two-dimensional Riemannian manifold 
$\mathcal{M} \subset \bR^d$, i.e.\ $\mathcal{M} = \vec\Phi(H)$ 
and $|\partial_{\vec\ek_1} \vec\Phi(\vec z)|^2 = |\partial_{\vec\ek_2} 
\vec\Phi(\vec z)|^2$ and
$\partial_{\vec\ek_1} \vec\Phi(\vec z) \,.\, 
\partial_{\vec\ek_2} \vec\Phi(\vec z) = 0$ for all $\vec z \in H$. 
Given the parameterization $\vec x : I \to H$ 
of the closed curve $\Gamma \subset H$, we let $\vec y = \vec\Phi \circ \vec x$
be a parameterization of $\mathcal{G} \subset \mathcal{M}$. 
We now show that geodesic curvature flow, geodesic curve 
diffusion and geodesic elastic flow for $\mathcal{G} = \vec y(I)$ on 
$\mathcal{M}$ reduce to (\ref{eq:Vgkg}), (\ref{eq:sdg}) and 
(\ref{eq:g_elastflow}) for the metric $g$ defined by
\begin{equation} \label{eq:gconformal2}
g(\vec z) = |\partial_{\vec\ek_1} \vec\Phi(\vec z)|^2
= |\partial_{\vec\ek_2} \vec\Phi(\vec z)|^2 \qquad \vec z \in H\,.
\end{equation}
For later use we observe that 
\begin{equation} \label{eq:Dphi2}
D\,\vec\Phi(\vec z)\,\vec v\,.\,D\,\vec\Phi(\vec z)\,\vec w
= g(\vec z)\,\vec v\,.\,\vec w \qquad \forall\ z \in H\,,\ \vec v,\vec w \in
\bR^2\,,
\end{equation}
where $D\,\vec\Phi(\vec z) = [\partial_{\vec\ek_1}\,\vec\Phi(\vec z)\
\partial_{\vec\ek_2}\,\vec\Phi(\vec z)] \in \bR^{d\times2}$ for $\vec z\in H$.
A simple computation, on noting 
(\ref{eq:Dphi2}) and (\ref{eq:normg}), yields that 
\begin{equation} \label{eq:normyrho}
\vec y_\rho = D\,\vec\Phi(\vec x)\,\vec x_\rho  \quad\Rightarrow\quad
|\vec y_\rho| = 
g^\frac12(\vec x)\,\vec x_\rho = |\vec x_\rho|_g\,.
\end{equation}
Hence it follows from (\ref{eq:sg}) that
\begin{equation} \label{eq:partialsy}
\partial_{s_y} = |\vec y_\rho|^{-1}\,\partial_\rho 
= |\vec x_\rho|_g^{-1}\,\partial_\rho = \partial_{s_g}\,,
\end{equation}
and so the unit tangent to the curve $\vec y(I)$ is given by
\begin{align}
\vec\tau_{\mathcal{M}} = 
\vec y_{s_y} 
= D\,\vec\Phi(\vec x)\,\vec x_{s_g} = D\,\vec\Phi(\vec x)\,\vec\tau_g\,,
\label{eq:tauM}
\end{align}
on recalling (\ref{eq:nug}). 
Similarly, we define the normal $\vec\nu_\mathcal{M}$ 
as the unit normal to $\vec y(I)$ that is perpendicular to 
$\vec\tau_{\mathcal{M}}$ and that lies in the tangent space to $\mathcal{M}$,
i.e.\
\begin{align} \label{eq:nuM}
\vec\nu_{\mathcal{M}} & = 
( \vec x_{s_g}\,.\,\vec\ek_1\,\partial_{\vec\ek_2}\,\vec\Phi (\vec x)
- \vec x_{s_g}\,.\,\vec\ek_2\,\partial_{\vec\ek_1}\,\vec\Phi (\vec x) )
= D\,\vec\Phi(\vec x)\,
 [- \vec x_{s_g}^\perp] = D\,\vec\Phi(\vec x)\,\vec\nu_g\,,
\end{align}
where we have recalled (\ref{eq:nug}). Note that (\ref{eq:nuM}), 
in the case $d=3$, agrees with the definition of $\vec\nu_{\mathcal{M}}$ in 
\citet[p.\ 10]{curves3d}. We further note from (\ref{eq:tauM}) 
that $\vec{y}_{s_ys_y}$ is perpendicular to $\vec{\tau}_{\mathcal{M}}$, 
and hence
\begin{equation} \label{eq:yss}
\vec y_{s_ys_y} =
\varkappa_{\mathcal{M}}\,\vec{\nu}_{\mathcal{M}}+\vec{\varkappa}_F\,,
\end{equation}
where $\vec\varkappa_F$
is normal to $\mathcal{M}$, and where $\varkappa_{\mathcal{M}}$
denotes the geodesic curvature of $\vec y(I)$.

Clearly, it follows from (\ref{eq:normyrho}) 
that the length of $\vec y(I)$ is given by
\begin{equation} \label{eq:lengthy}
L_{\mathcal{M}} (\vec y) = \int_I |\vec y_\rho| \drho
= \int_I |\vec x_\rho|_g \drho = L_g(\vec x)\,.
\end{equation}
We compute, on noting (\ref{eq:lengthy}), (\ref{eq:partialsy}), (\ref{eq:yss}),
(\ref{eq:nuM}), (\ref{eq:Dphi2}), (\ref{eq:Vg})
and (\ref{eq:nug}), that
\begin{align}
\ddt\,L_{\mathcal{M}} (\vec y) &
= \int_I \frac{\vec y_\rho}{|\vec y_\rho|}\,.\,(\vec y_\rho)_t \drho
= - \int_I \left(\frac{\vec y_\rho}{|\vec y_\rho|}\right)_\rho\,.\,\vec y_t 
\drho
= - \int_I \vec y_{s_ys_y} \,.\,\vec y_t \,|\vec y_\rho| \drho \nonumber \\ &
= - \int_I \varkappa_{\mathcal{M}}\,\vec\nu_{\mathcal{M}}
\,.\,\vec y_t \,|\vec y_\rho| \drho
= - \int_I \varkappa_{\mathcal{M}}\,D\,\vec\Phi(\vec x)\,\vec\nu_g
\,.\,D\,\vec\Phi(\vec x)\,\vec x_t \,|\vec x_\rho|_g \drho
\nonumber \\ &
= - \int_I \varkappa_{\mathcal{M}}\,g(\vec x)\,\vec\nu_g\,.\,\vec x_t 
\,|\vec x_\rho|_g \drho
= - \int_I \varkappa_{\mathcal{M}}\,\mathcal{V}_g
\,|\vec x_\rho|_g \drho\,.
\label{eq:dLMdt}
\end{align}
It follows from (\ref{eq:dLMdt}), (\ref{eq:dLdtV}) and
(\ref{eq:lengthy}) that 
\begin{equation} \label{eq:kgkM}
\varkappa_g = \varkappa_{\mathcal{M}}\,.
\end{equation}
In addition we have from (\ref{eq:nuM}), (\ref{eq:Dphi2}), (\ref{eq:nug})  
and (\ref{eq:Vg}) that
\begin{equation} \label{eq:VM}
\mathcal{V}_{\mathcal{M}} = \vec y_t\,.\,\vec\nu_{\mathcal{M}}
= D\,\vec\Phi(\vec x)\,\vec x_t\,.\,D\,\vec\Phi(\vec x)\,\vec\nu_g
= \mathcal{V}_g\,.
\end{equation}
Hence the flow (\ref{eq:Vgkg}) for $\vec x(I)$ in $H$ is equivalent to
the flow 
\begin{equation} \label{eq:VMkM}
\mathcal{V}_{\mathcal{M}} = \varkappa_{\mathcal{M}}
\end{equation}
for $\vec y(I)$ on $\mathcal{M}$, which is the so-called geodesic curvature
flow, see also \citet[(2.19)]{curves3d} for the case $d=3$. 

Similarly, it follows from (\ref{eq:VM}), (\ref{eq:kgkM}) and
(\ref{eq:partialsy}) that the flow (\ref{eq:sdg}) for $\vec x(I)$ in $H$ 
is equivalent to the the geodesic curve diffusion flow 
\begin{equation} \label{eq:VMsd}
\mathcal{V}_{\mathcal{M}} = - (\varkappa_{\mathcal{M}})_{s_gs_g}\,,
\end{equation}
see also \citet[(2.20)]{curves3d} for the case $d=3$.
Finally, in order to relate (\ref{eq:g_elastflow}) for $d=3$ to
geodesic elastic flow, i.e.\ the $L^2$--gradient flow of
\begin{equation} \label{eq:elasty}
\tfrac12\,\int_I \varkappa_{\mathcal{M}}^2\,|\vec y_\rho| \drho
= \tfrac12\,\int_I \varkappa_g^2\, |\vec x_\rho|_g \drho
= W_g(\vec x)\,,
\end{equation}
recall (\ref{eq:kgkM}), (\ref{eq:normyrho}) and (\ref{eq:Wg}), we note from
from \citet[(2.32)]{curves3d} that geodesic elastic flow for $\vec y(I)$ on
$\mathcal{M}$ is given by
\begin{equation} \label{eq:VMelast}
\mathcal{V}_{\mathcal{M}} = - (\varkappa_{\mathcal{M}})_{s_ys_y} 
- \tfrac12\,\varkappa_{\mathcal{M}}^3 
- \mathcal{K}(\vec y)\,\varkappa_{\mathcal{M}}\,,
\end{equation}
where $\mathcal{K}(\vec z)$ denotes the Gaussian curvature of $\mathcal{M}$ at
$\vec z \in \mathcal{M}$.
It follows from (\ref{eq:VM}), (\ref{eq:kgkM}), (\ref{eq:partialsy}) and
(\ref{eq:S0Gauss})
that (\ref{eq:VMelast}) and (\ref{eq:g_elastflow}) are equivalent.

\end{appendix}

\noindent
{\large\bf Acknowledgements}\\
The authors gratefully acknowledge the support 
of the Regensburger Universit\"atsstiftung Hans Vielberth.

\providecommand\noopsort[1]{}\def\soft#1{\leavevmode\setbox0=\hbox{h}\dimen7=\ht0\advance
  \dimen7 by-1ex\relax\if t#1\relax\rlap{\raise.6\dimen7
  \hbox{\kern.3ex\char'47}}#1\relax\else\if T#1\relax
  \rlap{\raise.5\dimen7\hbox{\kern1.3ex\char'47}}#1\relax \else\if
  d#1\relax\rlap{\raise.5\dimen7\hbox{\kern.9ex \char'47}}#1\relax\else\if
  D#1\relax\rlap{\raise.5\dimen7 \hbox{\kern1.4ex\char'47}}#1\relax\else\if
  l#1\relax \rlap{\raise.5\dimen7\hbox{\kern.4ex\char'47}}#1\relax \else\if
  L#1\relax\rlap{\raise.5\dimen7\hbox{\kern.7ex
  \char'47}}#1\relax\else\message{accent \string\soft \space #1 not
  defined!}#1\relax\fi\fi\fi\fi\fi\fi}

\end{document}